\documentclass[reqno,12pt,a4paper]{amsart}

\usepackage{amsmath,amssymb,amsthm,graphicx,mathrsfs,url}
\usepackage{a4wide}
\usepackage{enumitem}
\usepackage{mathtools}
\usepackage[utf8]{inputenc}
\usepackage[dvipsnames]{xcolor}
\usepackage[colorlinks=true,linkcolor=Red,citecolor=Green]{hyperref}
\hypersetup{pdfstartview=XYZ}
\usepackage{tikz}
\usetikzlibrary{positioning}
\usepackage{tikz-cd}
\usepackage{mathdots}

\usepackage{extarrows}
\usepackage{oplotsymbl}

\theoremstyle{plain}
\newtheorem{introthm}{Theorem}

\newtheorem{introcor}[introthm]{Corollary}

\newtheorem{thm}{Theorem}
\newtheorem{lem}{Lemma}[section]

\newtheorem{prop}[lem]{Proposition}

\newtheorem{cor}[lem]{Corollary}
\newtheorem{fact}[lem]{Fact}

\theoremstyle{definition}
\newtheorem{definition}[lem]{Definition}
\newtheorem{ex}[lem]{Example}
\theoremstyle{remark}
\newtheorem{rem}{Remark}[section]
 
\numberwithin{equation}{section}

\newcommand{\C}{\mathbb{C}}
\newcommand{\R}{\mathbb{R}}
\newcommand{\Q}{\mathbb{Q}}
\newcommand{\Z}{\mathbb{Z}}

\newcommand{\N}{\mathbb{N}}

\newcommand{\norm}[1]{\left\Vert#1\right\Vert}

\DeclareMathOperator{\ad}{ad}

\DeclareMathOperator{\SL}{SL}

\def\Ddots{\mathinner{\mkern1mu\raise\p@
\vbox{\kern7\p@\hbox{.}}\mkern2mu
\raise4\p@\hbox{.}\mkern2mu\raise7\p@\hbox{.}\mkern1mu}}
\makeatother

\newcommand{\F}{{\mathcal F}}
\newcommand{\Fbb}{{\mathbb F}}

\newcommand{\M}{{\mathcal M}}
\newcommand{\T}{{\rm T}}

\newcommand{\Xbb}{{\mathbb X}}

\renewcommand{\epsilon}{\vararepsilon}

\newcommand{\bdm}{\begin{displaymath}}
\newcommand{\edm}{\end{displaymath}}
\newcommand{\bq}{\begin{equation}}
\newcommand{\eq}{\end{equation}}
\newcommand{\bqn}{\begin{equation*}}
\newcommand{\eqn}{\end{equation*}}

\renewcommand{\L}{{\mathbb L}}

\newcommand{\G}{{\mathcal G}}

\newcommand{\GL}{\mathrm{GL}}

\newcommand{\g}{{\bf \mathfrak g}}

\newcommand{\aL}{{\bf \mathfrak a}}
\newcommand{\nL}{{\bf \mathfrak n}}
\renewcommand{\k}{{\bf \mathfrak k}}

\newcommand{\p}{{\bf \mathfrak p}}

\newcommand{\Hom}{\mathrm{Hom}}
\newcommand{\Ad}{\mathrm{Ad}}

\newcommand{\id}{\mathrm{id}\,}

\renewcommand{\det}{\mathrm{det}\,}

\newcommand{\gam}{\Gamma\backslash}

\definecolor{darkblue}{rgb}{0,0,0.4}

\newcommand{\set}[2]{\left\lbrace #1 \,\middle|\, #2 \right\rbrace}

\newcommand{\transverse}{\pitchfork}
\newcommand{\bbL}{\mathbb{L}}

\newcommand{\map}[4]{\left\lbrace \begin{array}{ccc} #1 & \to & #2 \\ #3 & \mapsto & #4 \end{array} \right.}
\renewcommand{\tilde}{\widetilde}
\newcommand{\cK}{\mathcal{K}}

\newcommand{\cF}{\mathcal{F}}
\newcommand{\cG}{\mathcal{G}}

\newcommand{\cM}{\mathcal{M}}

\newcommand{\fg}{\mathfrak{g}}
\newcommand{\fl}{\mathfrak{l}}
\newcommand{\fm}{\mathfrak{m}}
\newcommand{\fn}{\mathfrak{n}}
\newcommand{\fp}{\mathfrak{p}}

\newcommand{\fz}{\mathfrak{z}}

\newcommand{\fa}{\mathfrak{a}}

\newcommand{\Stab}{\mathrm{Stab}}

\newcommand{\Lag}{\mathrm{Lag}}
\newcommand{\Sp}{\mathrm{Sp}}
\newcommand{\hol}{\mathrm{hol}}

\newcommand{\bbf}{\mathbf{b}}

\newcommand{\Wbb}{\mathbb{W}}

\newcommand{\Fpf}{{\F^\pitchfork}}

\newcommand{\xbf}{\mathbf{x}}
\newcommand{\sbf}{\mathbf{s}}

\newcommand{\Bg}{\mathrm{B}_{\mathfrak{g}}}

\newcommand{\trtimes}{\,{\scriptscriptstyle\stackrel{\scriptscriptstyle\pitchfork}{\scriptstyle\times}}\,}

\newcommand{\Ftr}{\F^+\trtimes\F^-}

\author[B.~Delarue]{Benjamin Delarue}
\email{bdelarue@math.upb.de}
\address{Universit\"at Paderborn, Warburger Str.~100, 33098 Paderborn, Germany}

\author[D.~Monclair]{Daniel Monclair}
\email{daniel.monclair@universite-paris-saclay.fr}
\address{Université Paris-Saclay, CNRS, Laboratoire de mathématiques d’Orsay, 91405, Orsay, France}

\author[A.~Sanders]{Andrew Sanders}
\email{andrew.sanders.2@bc.edu}
\address{Boston College, Visiting Scholar, 02467, Chestnut Hill, Massachusetts, United States}

\title[Locally homogeneous Axiom A flows II]{Locally homogeneous Axiom A flows II:\\ geometric structures for Anosov subgroups}

\begin{document}

\begin{abstract}
Given a non-compact semisimple real Lie group $G$ and an Anosov subgroup $\Gamma$, we utilize the correspondence between $\R$-valued additive characters on Levi subgroups $L$ of $G$ and $\R$-affine homogeneous line bundles over $G/L$ to systematically construct families of non-empty domains of proper discontinuity for the $\Gamma$-action. If $\Gamma$ is torsion-free, the analytic dynamical systems on the quotients are Axiom A, and assemble into a single partially hyperbolic multiflow.  Each Axiom A system admits global analytic stable/unstable foliations with non-wandering set a single basic set on which the flow is conjugate to Sambarino's refraction flow, establishing that all refraction flows arise in this fashion.  Furthermore, the $\R$-valued additive character is regular if and only if the associated Axiom A system admits a compatible pseudo-Riemannian metric and contact structure, which we relate to the Poisson structure on the dual of the Lie algebra of $G$.
\end{abstract}

\maketitle

\section{Introduction}

The concept of a $\Theta$-Anosov subgroup $\Gamma< G$ of a semisimple real Lie group $G$ (i.e., the image of a $\Theta$-Anosov representation with respect to a  subset $\Theta\subset\Delta$ of a simple system) was introduced by Labourie \cite{labourie} and generalized by Guichard-Wienhard \cite{GW12}. It proved to be so rich in examples and interesting general properties that it has become its own field of research. As the name suggests, Labourie's original approach was intimately related to uniformly hyperbolic smooth dynamical systems (a.k.a.\ Anosov flows), but modern definitions \cite{KLP17,GGKW, BPS, KasselPotrie} no longer appeal to flows, instead casting the $\Theta$-Anosov property in terms of growth rates of singular or eigenvalues. 

If $G$ has real rank one, a subgroup $\Gamma< G$ is $\Theta$-Anosov if and only if it is convex cocompact \cite[Thm.\ 5.15]{GW12}, and if $\Gamma$ is torsion-free the rank one Riemannian locally symmetric space is a geometric structure for $\Gamma$. The unit tangent bundle of the latter carries the geodesic flow, featuring a similar uniformly hyperbolic behavior as an Anosov flow (Smale's Axiom A, see Definition \ref{def:axiomA}).

In higher rank, the construction of geometric structures associated to (torsion-free) $\Theta$-Anosov subgroups $\Gamma< G$ is a key challenge in the field \cite{FannyICM,AnnaICM} with many results (e.g.\ \cite{GW12,KLP18,GGKW,CS23,NR24}), yet a systematic construction which always produces non-empty locally homogeneous manifolds has so far remained elusive. In this article, we provide such a construction. The obtained manifolds carry analytic Axiom A flows or partially hyperbolic multiflows, respectively, reaffirming the original link between Anosov subgroups and smooth hyperbolic dynamics.  

In Part I \cite{VolI} of this series, we showed that every torsion-free projective Anosov subgroup $\Gamma<\mathrm{SL}(d, \R)$ is the monodromy group of a locally homogeneous real analytic contact Axiom A dynamical system $(\mathcal{M}, \phi^{t})$ with a unique basic hyperbolic set $\mathcal{K}\subset \mathcal{M}$.   With the modern toolbox of smooth hyperbolic flows now  applicable, we proved exponential mixing for all Gibbs measures (assuming $\Gamma<\mathrm{SL}(d, \R)$ is irreducible) and established meromorphic continuation of Ruelle zeta functions.  The analog of $(\mathcal{M}, \phi^{t})$ when $\Gamma< G$ is a torsion-free $\Theta$-Anosov subgroup is a family of Axiom A systems $(\mathcal{M}_{\mathbf{b}}, \phi_{\mathbf{b}}^{t})$ indexed by additive characters $\mathbf{b}: L_{\Theta}\rightarrow \R$ in an open convex cone in $\mathrm{Hom}(L_{\Theta}, \R)\simeq \R^{\lvert \Theta \rvert}$, where $L_{\Theta}< G$ is the standard Levi subgroup associated to $\Theta\subset \Delta$. 

It is well-known (see e.g.\ \cite[Sec.~9]{GW12}, \cite[Sec.~3]{GGKW}) that every $\Theta$-Anosov subgroup $\Gamma< G$ becomes projective Anosov  after composition with a suitable linear representation $G\rightarrow \mathrm{SL}(d, \R)$, the latter of which are organized by the integer lattice of fundamental weights.  If the differential $d_e\mathbf b$  at the identity lies at suitable points in the weight lattice, the Axiom A system $\mathcal{M}_{\mathbf{b}}$ we will construct also appears as an invariant submanifold in the corresponding locally homogeneous $(\mathrm{SL}(d, \R), \mathbb{L})$-manifold constructed in Part I 
\cite{VolI}.  But, even if $G=\mathbf{G}(\R)$ where $\mathbf{G}$ is a semisimple algebraic group defined over $\mathbb{\Q},$ a generic character $\mathbf{b}\in \mathrm{Hom}(L_{\Theta}, \R)$ is a real analytic, non-algebraic homomorphism and we cannot obtain $\mathcal{M}_{\mathbf{b}}$ from Part I by using linear representations.

Sambarino \cite{SAM24} used additive characters on $L_{\Theta}$ to construct families of so-called refraction flows associated to any $\Theta$-Anosov subgroup, and we will show that every refraction flow is bi-H\"{o}lder conjugate to the restriction of the Axiom A flow $\phi_{\mathbf{b}}^{t}$ to its unique basic hyperbolic set $\mathcal{K}_{\mathbf{b}}\subset \mathcal{M}_{\mathbf{b}}$, in particular arriving at the same open convex cone of admissible parameters $\mathbf b\in \mathrm{Hom}(L_{\Theta}, \R)$ as Sambarino. 

 Our starting point is the observation that the differential $d_e\mathbf b$ canonically extends to a linear form $\beta: \fg\rightarrow \R,$ and the coadjoint $G$-orbit of $\beta$ is generically isomorphic to $G/L_{\Theta}.$  Simultaneously, $\mathbf{b}: L_{\Theta}\rightarrow \R$ defines a homogeneous $\R$-affine line bundle $\mathbb{L}_{\mathbf{b}}$ over this coadjoint orbit, and it is in the total space of $\mathbb{L}_{\mathbf{b}}$ where we shall construct a non-empty domain of proper discontinuity for the $\Gamma$-action.  We now turn to the details of this construction and the statements of our main results.  

Throughout, $G$ denotes a non-compact real semisimple Lie group with Lie algebra $\g$ and $\Theta\subset \Delta$ a non-empty subset of simple restricted roots defining associated opposite proper parabolic subgroups $P_{\Theta}^{+}, P_{\Theta}^{-}< G$ with standard Levi intersection $L_{\Theta}= P_{\Theta}^{+}\cap P_{\Theta}^{-}$.

\subsection{Dynamical \texorpdfstring{$(G,X)$}{(G,X)}-structures}
Here we give a short description of the homogeneous spaces studied in this paper and the attached geometric and dynamical structures. Their construction can be thought of as reverse-engineering the unit tangent bundle $T^1\mathbb X$ of a rank-one symmetric space $\mathbb X$  from the projections $T^1\mathbb X\to\partial_\infty\mathbb X$  sending a unit tangent vector $v\in T^1_x\mathbb{X}$ to the future and past endpoints $v^\pm\in\partial_\infty\mathbb{X}$ of the geodesic initiated by $v$. 
 We replace the visual boundary $\partial_\infty\mathbb{X}$  with a pair of \emph{opposite flag manifolds} $\cF^\pm\simeq G/P^\pm_\Theta$. The image in $\partial_\infty\Xbb\times\partial_\infty\Xbb$ of the  projection $v\mapsto (v^+,v^-)$ will be replaced with the \emph{transverse flag space} $\F^+\trtimes \F^-\subset\F^+\times \F^-$, the unique open $G$-orbit in $\F^+\times \F^-\simeq G/P^+_\Theta\times G/P^+_\Theta$ which can be equivalently described as the homogeneous space $\F^+\trtimes \F^-\simeq G/L_\Theta$.

\subsubsection{Hyperbolic homogeneous flows}
By a homogeneous flow on a homogeneous space $X=G/H$ we mean a flow $\phi^t:X\to X$ defined for all $t\in \R$ that commutes with the $G$-action. 
We will study the case where the orbit space of $\phi^t$ is the homogeneous space $G/L_\Theta$, leading us  to consider subgroups $H_\bbf:=\ker\mathbf b$ for some non-zero additive character $\bbf\in \Hom(L_\Theta,\R)$. From the structure of Levi subgroups (see Section \ref{sec:centerofLevi}), we find a canonical choice of an element $Z_{\bbf}\in\fz(\fl_\Theta)$ in the center of the Lie algebra of $L_\Theta$ such that $d_e\mathbf b(Z_{\bbf})=1$. This allows us to equip the homogeneous space $\L_{\bbf}:=G/H_\bbf$  with the flow
\[ \phi^t_{\mathbf b}:\map{\L_{\mathbf b}}{\L_{\mathbf b}}{gH_\bbf}{g\exp(tZ_{\mathbf b})H_\bbf,}\]
which is a right action of the one-parameter subgroup $A_\bbf:=\exp(\R Z_\bbf)< Z(L_\Theta)$. 

From the embedding of $G/L_\Theta$ into $G/P^+_\Theta\times G/P^-_\Theta$, we obtain a decomposition
\begin{equation}
 \T\L_{\mathbf b}=E^-\oplus E^0\oplus E^+, \label{eq:decomTLbeta}
\end{equation}
where the rank one bundle $E^0$ is tangent to the flow and $E^\pm$ is sent isomorphically onto $T(G/P^\pm_\Theta)$ by the differential of the projection $\L_{\mathbf b}=G/H_\bbf\to G/L_\Theta$ (see Section \ref{sec:tangent bundle flow space}).
\begin{definition}
The pair $(\L_{\mathbf b},\phi_{\mathbf b}^t)$ will be referred to as a \emph{hyperbolic homogeneous flow}.
\end{definition}

\subsubsection{Multiflows}  Hyperbolic homogeneous flows with a given orbit space $G/L_\Theta$ form a family indexed by $\Hom(L_\Theta,\R)$. The Levi subgroup $L_\Theta$ splits as a direct product $L=A_\Theta M_\Theta$, where $A_\Theta< Z(L_\Theta)$ is isomorphic to the abelian Lie group $\R^{|\Theta|}$,  $M_\Theta$ is a reductive subgroup, and the restriction map $\Hom(L_\Theta,\R)\to\Hom(A_\Theta,\R)$ is an isomorphism. Equivalently, these subgroups can be described as factors in the  Langlands decompositions $P^\pm_\Theta=M_\Theta A_\Theta N^\pm_\Theta$. The homogeneous space $\mathbb W_\Theta:=G/M_\Theta$ comes with an action of $A_\Theta\simeq\R^{|\Theta|}$ on the right given by $(g M_\Theta)\cdot h:= gh M_\Theta$ for $h\in A_\Theta$. 
For $\mathbf b\in \Hom(L_\Theta,\R)\setminus\{0\}$, the normal subgroup $M_\Theta\lhd H_\bbf$ satisfies $H_\bbf/M_\Theta\simeq A_\Theta/A_\bbf\simeq \R^{|\Theta|-1}$, and defining the projection $p_{\mathbf b}:\mathbb W_\Theta=G/M_\Theta\to G/H_\bbf=\L_{\mathbf b},$ we find that
\[ p_{\mathbf b}(x\cdot h)=\phi_{\mathbf b}^{\mathbf b(h)}\big(p_{\mathbf b}(x)\big)\,\qquad \forall\, x\in \mathbb W_\Theta, h\in A_\Theta.\]
\begin{definition}
The right $A_\Theta$-action on $\mathbb W_\Theta=G/M_\Theta$  will be referred to as a \emph{multiflow}.
\end{definition}
Dynamically, the fact that $p_{\mathbf b}:\mathbb W_\Theta\to \L_{\mathbf b}$ is a principal  $H_\bbf/M_\Theta\simeq A_\Theta/A_\bbf\simeq \R^{|\Theta|-1}$-bundle means that $\L_{\mathbf b}$ is obtained from $\mathbb W_\Theta$ by  ``quotienting out all but one flow direction.'' 

Perhaps surprisingly, for a $\Theta$-Anosov subgroup $\Gamma< G,$ we will see that there exists a non-empty domain of proper discontinuity $\widetilde{\mathcal{N}}\subset \Wbb_{\Theta}$ such that the $\R^{\lvert \Theta\rvert -1}\simeq A_\Theta/A_{\mathbf{b}}$-action is proper on the $\Gamma$-quotient $\gam \widetilde{\mathcal{N}}$ and the quotient $\gam \widetilde{\mathcal{N}}/\R^{\lvert \Theta\rvert -1}$ is an Axiom A system.  

\subsubsection{\texorpdfstring{$(G,X)$}{(G,X)}-structures for Anosov subgroups} 

Given a $G$-homogeneous space $X=G/H$ and a manifold $\cM$, a $(G,X)$-structure on $\cM$ is a maximal atlas of $X$-valued charts  
on $\cM$ such that the transition maps are $G$-valued and locally constant. A  $(G,X)$-structure on $\cM$ can be equivalently defined by a holonomy homomorphism $\rho_\hol:\pi_1(\cM)\to G$ and a $\rho_\hol$-equivariant local diffeomorphism $\mathrm{Dev}:\widetilde{\cM}\to X$ called the \emph{developing map}, where $\widetilde{\cM}$ is any universal cover of $\M$ realizing $\pi_{1}(\cM)$ as the group of deck transformations. 

Starting with a discrete subgroup $\Gamma< G$, one way of producing a $(G,X)$-manifold is to find a non-empty open $\Gamma$-invariant subset $U\subset X$ on which the $\Gamma$-action is free and properly discontinuous. The quotient manifold $\cM_\Gamma:=\Gamma\backslash U$ carries a $(G,X)$-structure which we call \emph{uniformized} (this is referred to as ``Type $\mathrm{U}$'' in \cite{FannyICM}). 

Any $(G,X)$-manifold inherits all local structures on $X$ preserved by $G$. 
In the case of $X=\L_\bbf$ for some $\bbf\in\Hom(L_\Theta,\R)\setminus\{0\}$, this means that a $(G,\L_\bbf)$-manifold comes with a local flow, i.e. a vector field.  Similarly, a $(G,\Wbb_\Theta)$-structure defines a local $A_\Theta$-action, i.e.\ a family of $|\Theta|$ commuting vector fields.
\begin{definition} A $(G,X)$-structure with $X=\L_\bbf$ or $X=\Wbb_\Theta$ is called \emph{dynamically complete} if the induced vector fields are complete.
\end{definition}
In the uniformized case, this means that the open subset $U\subset\L_\bbf$ (resp.\ $U\subset \Wbb_\Theta$) is $\phi_\bbf^t$-invariant (resp.\ $A_\Theta$-invariant). The history of $(G,X)$-structures associated to Anosov subgroups is rich and well documented, see \cite{FannyICM,AnnaICM,CS23} and references therein.  To date and not including \cite{VolI}, dynamical $(G,X)$-structures for higher rank Anosov subgroups do exist (e.g.\ \cite{GW08}) but are rather rare among known examples.  We now explain that dynamical $(G,X)$-structures with strong hyperbolicity properties exist in abundance.

\subsection{Results} 
We now consider a $\Theta$-Anosov subgroup $\Gamma< G$. The dynamics of the $\Gamma$-action on the flow space $\L_\bbf$ will depend on the character $\bbf\in\Hom(L_\Theta,\R)$, and especially on its relation to the \emph{$\Theta$-limit cone $\mathcal L_\Theta(\Gamma)\subset\aL_\Theta^+$} (the projection of the \emph{Benoist limit cone} $\mathcal L(\Gamma)$ to the Lie algebra $\aL_\Theta$ of $A_\Theta$, see Section  \ref{sec:admissibility}). Consider the dual cone \[\mathcal L_\Theta(\Gamma)^*=\set{\beta\in\aL_\Theta^*}{\beta\vert_{\mathcal L_\Theta(\Gamma)}\geq 0},\] whose interior $\mathrm{Int}(\mathcal L_\Theta(\Gamma)^*)$ is non-empty because $\mathcal L_\Theta(\Gamma)$ is acute, and identify the differential $d_e\bbf\in\fl_\Theta^*$ with its restriction to $\aL_\Theta\subset\fl_\Theta$, where $\fl_\Theta$ is the Lie algebra of $L_\Theta$.

\subsubsection{Uniformized \texorpdfstring{$(G,X)$}{(G,X)}-structures}  Our first result establishes the existence of $(G,X)$-structures and hyperbolic homogeneous flows,   generalizing \cite[Thm.~A(1)]{VolI}.
\begin{introthm}\label{introthm:(G,X)-structure} Let $\Gamma< G$ be $\Theta$-Anosov, and let $\bbf\in\Hom(L_\Theta,\R)$. If $d_e\bbf\in\mathrm{Int}(\mathcal L_\Theta(\Gamma)^*)$,  then there is a non-empty $\Gamma\times\phi_\bbf^t$-invariant open set $\tilde \M_{\bbf}\subset \L_\bbf$ on which $\Gamma$ acts properly discontinuously, yielding a dynamically complete uniformized $(G,\L_\bbf)$-manifold $\mathcal M_\bbf:=\Gamma\backslash \widetilde{\mathcal M_\bbf}$ when $\Gamma$ is torsion-free.
\end{introthm}

It is to be noted that the action of $\Gamma$ on  $\L_\bbf$ is almost never proper, the only notable exception being the case of a group $G$ of real rank one (then $\L_\bbf\simeq T^1\Xbb$). The fact that the multiflow space $\Wbb_\Theta$ fibers over $\L_\bbf$ immediately yields the following consequence:

\begin{introcor}\label{introcor:multiflow domain}
Let $\Gamma< G$ be $\Theta$-Anosov. There is a non-empty $\Gamma\times A_\Theta$-invariant open set $\widetilde{\mathcal N}\subset \Wbb_\Theta$ on which $\Gamma$ acts properly discontinuously, yielding a dynamically complete uniformized $(G,\Wbb_\Theta)$-manifold $\mathcal N:=\Gamma\backslash \widetilde{\mathcal N}$ when $\Gamma$ is torsion-free. 

For each $\bbf\in\Hom(L_\Theta,\R)$ with $d_e\bbf\in\mathrm{Int}(\mathcal L_\Theta(\Gamma)^*)$, the projection $\mathcal{N}\rightarrow \mathcal{M}_{\mathbf{b}}$ induced by the canonical projection $\Wbb_\Theta\to \L_\bbf$ is a right principal $\R^{\lvert \Theta\rvert-1}\simeq A_{\Theta}/A_{\mathbf{b}}$-bundle.
\end{introcor}
The quotient manifolds $\mathcal{M}_{\mathbf{b}}$ are usually non-compact, but we will see that in each case the closure of the set of periodic orbits is compact and equal to the non-wandering set of the flow.  A helpful example is that of quasi-Fuchsian surface groups $\Gamma< \mathrm{PSL}(2, \C),$ where the locally homogeneous manifolds are non-compact but have compact cores.

 In real rank one, or more generally when $|\Theta|=1$, the spaces $\L_\bbf$ and $\Wbb_\Theta$ are identical. In higher rank, the situation when $\Theta$ is the entire set $\Delta$ of simple restricted roots stands out because of the properness of the $G$-action on $\Wbb_\Delta$. In other situations, a $\Theta$-Anosov subgroup $\Gamma$ typically does not act properly on $\Wbb_\Theta$.  The recent preprint \cite{KT24} proving the sharpness conjecture gives a systematic analysis for cocompact actions of general discrete subgroups which includes some examples of Anosov subgroups.
 
The homogeneous space $\Wbb_\Theta= G/M_\Theta$ often appears in the study of Anosov subgroups in the form of a Hopf parametrization  $\Wbb_\Theta\simeq G/L_\Theta\times \fa_{\Theta}$ with a cocycle defining the $G$-action on the right hand side, similarly for $\L_\bbf\simeq G/L_{\Theta}\times \R.$ The novelty in Theorem \ref{introthm:(G,X)-structure} and Corollary \ref{introcor:multiflow domain} is the discovery of the  domains of proper discontinuity $\tilde \M_{\bbf}\subset \L_\bbf$, $\widetilde{\mathcal N}\subset \Wbb_\Theta.$

\subsubsection{Dynamics of quotient flows}
Next, we turn to the dynamics of the flows defined by $(G,\L_\bbf)$-structures obtained in Theorem \ref{introthm:(G,X)-structure}, establishing  Smale's Axiom A. For the definition of the latter and other dynamical terminology used in the following, see Section \ref{sec:background_dynamics}. 
\begin{introthm}\label{introthm:A}
Let $\Gamma< G$ be a torsion-free $\Theta$-Anosov subgroup, and let $\bbf\in\Hom(L_\Theta,\R)$ be such that $d_e\bbf\in\mathrm{Int}(\mathcal L_\Theta(\Gamma)^*)$. Consider the $(G,\L_\bbf)$-manifold  $\M_{\bbf}$ with its flow $\phi_\bbf^t$ obtained from Theorem \ref{introthm:(G,X)-structure}. 
\begin{enumerate}
\item The real analytic flow $\phi^t_\bbf:\M_{\bbf}\to \M_{\bbf}$ is an Axiom A flow. 
\item The non-wandering set of $\phi^t_\bbf$ consists of a unique basic set $\mathcal K_\bbf\subset \M_\bbf$, and the splitting of $T\M_\bbf|_{\mathcal K_\bbf}$ into stable/neutral/unstable subbundles provided by the Axiom A property is given by first descending the $G$-invariant splitting \eqref{eq:decomTLbeta} to the global real analytic $\phi_{\mathbf{b}}^{t}$-invariant splitting $
T\M_\bbf=E_{\Gamma, \mathbf{b}}^{-}\oplus E_{\Gamma, \mathbf{b}}^{0}\oplus E_{\Gamma, \mathbf{b}}^{+}$, 
and then restricting to the basic hyperbolic set $\mathcal{K}_{\mathbf{b}}.$

\item The real analytic $\phi_{\mathbf{b}}^{t}$-invariant $1$-form $\tau_{\mathbf{b}}$ which vanishes on $E_{\Gamma, \mathbf{b}}^{-}\oplus E_{\Gamma, \mathbf{b}}^{+}$ and is equal to $1$ on the generator of $\phi_{\mathbf{b}}^{t}$ is a contact $1$-form if and only if $\mathrm{B}_\g(d_{e}\mathbf{b},w_{\alpha})\neq 0$ for every fundamental weight $w_{\alpha}$ with $\alpha\in \Theta$, where $\mathrm{B}_\g$ is the Killing form.
\end{enumerate}
\end{introthm}

For each character $\bbf\in\Hom(L_\Theta,\R)$ as in Theorem \ref{introthm:A}, Sambarino \cite{SAM24} defined a model of the Gromov geodesic flow of $\Gamma$ given by the so-called \emph{refraction flow}  $\psi_\bbf^t$ (see Section \ref{sec:refractionflows}). Our locally homogeneous Axiom A flows are extensions of refraction flows:

\begin{introthm}\label{introthm:B}In the situation of Theorem \ref{introthm:A}, the restriction of the Axiom A flow $\phi^t_\bbf:\M_\bbf\to\M_\bbf$ to its basic set $\mathcal{K}_\bbf\subset \mathcal{M}_\bbf$ is bi-Hölder-conjugate to the refraction flow $\psi_\bbf^t$. 
\end{introthm}
Theorems \ref{introthm:A} and  \ref{introthm:B} are direct generalizations of \cite[Theorems A, B]{VolI}, which treated projective Anosov subgroups of $\SL(d,\R)$. More examples will be considered in  Section \ref{sec:examples}. One of the differences when moving to this general setting is the condition imposed on the parameter $\bbf$ for the flow to be contact. From the perspective of refraction flows, this condition is superfluous and our results imply that any refraction flow can be realized as the restriction of a contact Axiom A flow to its non-wandering set (see Remark \ref{rem:refraction flows are contact}).

\subsubsection{Ruelle zeta functions}
Let   $\lambda:G\to \aL^+$  denote the Jordan projection with respect to a chosen positive Weyl chamber (see Section \ref{sec:JordanG}). 
Given a $\Theta$-Anosov subgroup $\Gamma< G$ and $\mathbf{b}\in \mathrm{Hom}(L_{\Theta}, \R)$ as in Theorem \ref{introthm:A} with $d_e{\bbf}$ canonically extended to a linear form on $\g$, the non-trivial primitive periodic orbits of the flow $\phi_{\mathbf{b}}^{t}$ are in bijection with primitive hyperbolic conjugacy classes in $\Gamma$  and the period of $[\gamma]$ is equal to $d_e{\bbf}(\lambda(\gamma)).$  Since the flow is real analytic, results of Fried \cite{DF95} imply the meromorphic continuation of Ruelle zeta functions twisted by linear representations, which immediately yields:
\begin{introthm}\label{introthm:C}
Let $k\in \N$ and $\rho: \Gamma\rightarrow \mathrm{GL}(k, \C)$ a homomorphism.  In the situation of Theorem \ref{introthm:A}, the formal Euler product
\[
\zeta_{\bbf, \rho}(s)=\prod_{[\gamma]\in [\Gamma]_{\mathrm{prim}}} \mathrm{det}\left(\mathbf{1}_{k}-e^{-sd_e{\bbf}(\lambda(\gamma))}\rho([\gamma])\right)^{-1}
\]
converges for $\mathrm{Re}(s)\gg 1$ and admits a meromorphic continuation to $s\in\C.$ Here $[\Gamma]_{\mathrm{prim}}$ is the set of conjugacy classes of non-trivial primitive elements in $\Gamma$ and the appearance of $\rho([\gamma])$ is only through its conjugation invariant characteristic polynomial.
\end{introthm}
Setting $k=1$ and $\rho=1$ recovers the usual Ruelle zeta function.  Alternatively, instead of \cite{DF95}, we can also directly apply results of Dyatlov-Guillarmou \cite{DG18} and Borns-Weil-Shen \cite{BWS21} giving the meromorphic continuation of Ruelle zeta functions with arbitrary $C^{\infty}$-weights.  In the setting $k=1$ and $\rho=1$, Theorem \ref{introthm:C} was proved for projective Anosov closed surface subgroups by Pollicott-Sharp \cite{PS24}, and shortly afterwards for all projective Anosov subgroups in \cite{VolI}, where it was deduced analogously from \cite[Theorems A, B]{VolI} using \cite{DG18,BWS21}.  Theorem \ref{introthm:C} gives the full unconditional result for all $\Theta$-Anosov subgroups $\Gamma< G$ and what we call admissible parameters $\bbf\in \Hom(L_\Theta,\R)$.  

When $\Gamma<G$ is Zariski-dense, Chow-Sarkar \cite{ChowSarkar2} recently elaborated the estimates necessary to prove a type of local non-integrability condition (uniformly in the admissible parameter $\bbf$) for Sambarino's refraction flows, leading to a proof that the usual Ruelle zeta functions $\zeta_{\bbf}$ have a zero/pole-free strip to the left of their simple pole at the topological entropy $h_{\mathrm{top}}(\psi_{\bbf}^{t})$, thus establishing exponential mixing for the measure of maximal entropy.  Strictly speaking, this result is only true when $\bbf$ satisfies the hypotheses of Theorem \ref{introthm:A}$.(3),$ though in general one may pass to a smaller subset of roots and obtain the same result (see \cite[Remark $1.3$]{ChowSarkar2}).  

Similar results were previously obtained for irreducible projective Anosov subgroups and arbitrary Gibbs measures in \cite{VolI}, where the local non-integrability condition is geometrically manifested by the existence of a contact structure on the ambient Axiom A system. In general, Theorem \ref{thm:contact flow} shows that under the failure of Theorem \ref{introthm:A}$.(3)$ and therefore the failure of the Axiom A system to be contact, there is a canonical contact reduction (e.g. quotient), which reveals that all refraction flows may be embedded in a contact Axiom A system, and provides a differential-geometric mechanism for the local non-integrability in consonance with the results of \cite{ChowSarkar2}.  Put simply, the success/failure of local non-integrability is exactly captured by whether the ambient Axiom A system is contact, and the fact that the contact property can be recovered via contact reduction (Theorem \ref{thm:contact flow}) is in harmony with \cite[Remark $1.3$]{ChowSarkar2}. Looking towards the future, exponential mixing for arbitrary Gibbs measure and with no Zariski-density/irreducibility hypothesis on $\Gamma<G$ remains an open problem; an analysis in the style of \cite{VolI} utilizing the work of Stoyanov \cite{St11} is one promising avenue of attack.

\subsection{Anosov subgroups vs.~Anosov flows}
Hyperbolic homogeneous flows play an important role in the proof of a famous result of Benoist-Foulon-Labourie \cite{BFL}: if an Anosov flow $\varphi^t:M\to M$ preserves a contact form on  a compact manifold $M$ and has smooth stable and unstable distributions, it is smoothly conjugate to a reparametrisation of the geodesic flow of a closed locally symmetric space of rank one. Their proof can be broken down into the following  steps: first proving that the flow $\varphi^t$ is induced from a $(G,\bbL_{\mathbf b})$-structure for some semi-simple Lie group $G$, then proving that this structure is complete (i.e. $ M$ is covered by $\bbL_{\mathbf b}$),  and finally that $G$ has rank one. In particular, the locally homogeneous flow $(\mathcal M_{\mathbf b},\phi^t_{\mathbf b})$ that we construct from an Anosov subgroup cannot be  an Anosov flow  if $G$ is simple and has higher rank.

\subsection{Further questions and outlook}

As noted above, the contact property of the Axiom A flows obtained in Theorem \ref{introthm:A} is a strong version of the non-integrability condition promoted by Dolgopyat for the study of rates of mixing. Combined with the real analyticity of all dynamical foliations, these tools should allow the construction of very refined anisotropic Banach spaces in which to study transfer operators.  In particular, we have in mind a mixture of the analytic/holomorphic approach of Rugh \cite{RU96} and Fried \cite{DF95}, approaches making key use of a contact structure (e.g.\ \cite{Li04}, \cite{ST25}), and microlocal approaches for non-compact systems as in \cite{DG18}.  One would ultimately like to extend these results to allow for non-smooth Hölder continuous potentials, at least those that appear as reparametrizations in the work of Sambarino (see \cite{SAM24}).

By applying the full technical machinery of microlocal analysis from \cite{DyGuSmale,DG18,Guillarmou-Mazzucchelli-Tzou}, Theorem \ref{introthm:C} can be accompanied by much more comprehensive, and in parts more technical, results on the existence of discrete spectra of Ruelle-Pollicott resonances with associated (co-)resonant states (c.f.\ \cite[Thm.~F]{VolI}). This will be the subject of future works. We are aware of forthcoming independent work by Guedes Bonthonneau-Lefeuvre-Weich \cite{BonthonneauLefeuvreWeich} establishing a spectral theory for the multiflows of Corollary \ref{introcor:multiflow domain}.

Regarding non-Anosov subgroups $\Gamma< G,$ the larger class of $\Theta$-transverse subgroups (see \cite{BCZZ24,CZZ24,KOW24,KOW25} and the survey \cite{CZZ25}) includes all relatively Anosov subgroups and all discrete subgroups in real rank one.  They are essentially defined as the largest class of subgroups for which there exists a pair of $\Gamma$-invariant transverse $\Theta$-limit sets $\Lambda_{\Gamma}^{\pm}\subset G/P_{\Theta}^{\pm}.$  This class of groups was initially identified by Kapovich-Leeb-Porti \cite{KLP17} where they were called $\tau_{mod}$-regular, antipodal subgroups. It would be interesting to know whether every $\Theta$-transverse subgroup $\Gamma< G$ acts properly discontinuously on a non-empty open subset $\widetilde{\mathcal{N}}\subset \Wbb_{\Theta}$ in the $\Theta$-multiflow space $\Wbb_{\Theta}$, in some hyperbolic flow space $\L_\bbf$, or even possibly in some intermediate quotients of rank $1<k<|\Theta|$.

The existence of a $(G,X)$-structure of a given type on a manifold can require specific topological properties, leaving us to wonder if the geometric structures constructed in this article  can impose restrictions on  which hyperbolic groups admit Anosov representations (for a given $G$ and $\Theta\subset\Delta$).  This is a notoriously difficult problem that we will not address in this paper. See \cite{CanaryTsouvalas,Tsouvalas,Dey,DeyGreenbergRiestenberg,KT25} for further information.

\subsection{Structure of the paper and proof strategy}

After a brief Background Section \ref{sec:background}, we turn in Section \ref{sec:Lietheory} to Lie theoretic concepts such as flag manifolds, transverse flag spaces, and the classification of parabolic and Levi subgroups. Crucial preparations for the main proofs are made in Section \ref{sec:linact}, where we consider linear actions on tangent spaces of flag manifolds.  Section \ref{sec:affpermultidens} is devoted to the general concept of affine line bundles and the language of period functions and cocycles. This is then applied in Section \ref{sec:densitiypairings} to density and multidensity bundles, the latter forming a graded analog of the usual notion of $s$-densities. While the flow spaces $\mathbb{L}_{\bbf}$ admit a very simple description in terms of coset spaces as above, our proofs are not algebraic in spirit and do not proceed in terms of this description.  Instead, we construct specific differential geometric models of the space $\mathbb{L}_{\bbf}$ in terms of multidensities,  which is the content of the dedicated Section \ref{sec:flowspaces}. Anosov subgroups enter the picture in Section \ref{sec:Anosovsubgroups}, where we prove our first main Theorem \ref{thm: prop disc general case} and make the connection to Sambarino's refraction flows in Section \ref{sec:refractionflows}. Together with Theorem  \ref{thm - trapped is non wandering is closure of periodic points is K general case} and Lemma \ref{lem:hyperbolicset} from the subsequent Section \ref{sec:dynamics} focusing on the dynamics on the $\Gamma$-quotients, this proves Theorems \ref{introthm:(G,X)-structure}, \ref{introthm:A}(1), \ref{introthm:A}(2) and \ref{introthm:B}. In Section \ref{sec:geometric_structures} we study in detail various geometric properties of our flow spaces $\mathbb{L}_{\bbf}$, with Theorem \ref{thm:contact flow} giving the proof of (a more precise statement than) Theorem \ref{introthm:A}(3). The passage to multiflow spaces is made in Section \ref{sec:towers}, where Corollary \ref{introcor:multiflow domain} is proved as Corollary \ref{cor:domain of proper discontinuity multiflow space}. We take special linear groups as a family of running examples throughout the paper, which is complemented by a final Examples Section \ref{sec:examples} describing our general constructions in concrete situations.

\subsection*{Acknowledgments}  This project has received funding from the Deutsche Forschungsgemeinschaft (DFG, German Research Foundation) through the Priority Program (SPP) 2026 ``Geometry at Infinity'' as well as Project-ID 491392403 – TRR 358. B.D.\ thanks Tobias Weich for countless stimulating discussions. D.M.\ is thankful for the hospitality of IHES where he stayed during part of this work. A.S. thanks Beatrice Pozzetti and Thi Dang for numerous conversations about Anosov subgroups and Weyl chamber flows which led to the consideration of homogeneous affine line bundles over transverse flag manifolds.

\section{Background}\label{sec:background}

This section represents a continuation of the corresponding section \cite[Section 2]{VolI} in Part I of this series. While we do (re-)introduce the required notation and preliminaries used in this paper to make it self-contained, we keep the repetitions minimal.

\subsection{Dynamics}\label{sec:background_dynamics}

Let $\M$ be a metrizable topological space and $\phi^t:\M\to \M$ a continuous flow defined for all $t\in \R$ which has no fixed points. 

\begin{definition} \label{def - non-wandering, periodic, trapped}~
\begin{itemize}
\item The \emph{non-wandering set} ${\mathcal{NW}}(\phi^t)$ of the flow  $\phi^t$ is the set of all  points $x\in \M$ for which there are  sequences $x_N\to x$ in $\M$ and $t_N\to +\infty$ in $\R$ such that $\phi^{t_N}(x_N)\to x$. 
\item The \emph{trapped set} ${\mathcal T}(\phi^t)$ of the flow  $\phi^t$  is the set of all points $x\in \M$ whose $\phi^t$-orbits are relatively compact in $\M$.
\item The set $\mathrm{Per}(\phi^t)$ of \emph{periodic points}  of the flow  $\phi^t$  consists of all points $x\in \M$ for which there exists $T>0$ with $\phi^T(x)=x$.
\end{itemize}
\end{definition}

\begin{definition}Let $\mathcal K\subset \M$ be a compact $\phi^t$-invariant set and $E$ a continuous vector bundle over $\mathcal K$  equipped with a continuous flow $\phi_E^t:E\to E$ lifting $\phi^t$ over $\mathcal K$. Then $\phi_E^t$ is \emph{uniformly contracting} (resp.\ \emph{expanding}) on $E$ if for some (hence any) continuous bundle norm $\norm{\cdot}$ on $E$ there are constants $C,c>0$ such that for all $p\in \mathcal K$ and all $v\in E_p$ one has
\bqn
\norm{\phi^t_E(v)}_{\phi^t(p)}\leq Ce^{-c|t|}\norm{v}_p
\eqn
for all $t\geq 0$ (resp.\ $t\leq 0$).
\end{definition}
For the remainder of Section \ref{sec:background_dynamics}, suppose now that $\M$ is a smooth manifold and $\phi^t$ is a smooth flow with generating vector field $\mathcal X:\M\to T\M$. It is nowhere vanishing since we assume that $\phi^t$ has no fixed points.

\begin{definition} \label{def:hyperbolicset}A compact $\phi^t$-invariant set $\mathcal K\subset \M$ is called \emph{hyperbolic} for the flow $\phi^t$ if the restriction of the tangent bundle $T\M$ to $\mathcal K$ admits a Whitney sum decomposition
\bq
T\M|_{\mathcal K}=E^0\oplus E^{\mathrm{s}}\oplus E^{\mathrm{u}}\label{eq:bundlesplittingstableunstableneutral}
\eq
where $E^0_p=\R \mathcal X(p)$ for all $p\in \mathcal K$ and $E^{\mathrm{s}}$, $E^{\mathrm{u}}$ are $d\phi^t$-invariant continuous subbundles such that $d\phi^t$ is uniformly contracting (resp.\ expanding) on $E^{\mathrm{s}}$ (resp.\ $E^{\mathrm{u}}$).
\end{definition}

\begin{definition}[{c.f.~\cite[§II.5 (5.1)]{smale67}}]\label{def:axiomA}The flow $\phi^t$ is an \emph{Axiom A} flow if the non-wandering set ${\mathcal{NW}}(\phi^t)$ is compact and hyperbolic and coincides with the closure  in $\M$ of the set of periodic points $\mathrm{Per}(\phi^t)$.
\end{definition}
Note that the manifold $\mathcal M$ itself is not assumed to be compact. A hyperbolic set $\mathcal K\subset \M$ for the flow $\phi^t$  is \emph{basic} if it is locally maximal for $\phi^t$ (i.e., there is a neighborhood $\mathcal U\subset \M$ of $\mathcal K$ such that $
\mathcal K=\bigcap_{t\in \R}\phi^t(\mathcal U)$), the flow $\phi^t|_{\mathcal K}$ is topologically transitive (i.e., $\mathcal K$ contains a dense $\phi^t$-orbit), and $\mathcal K$ is the closure in $\M$ of the set of periodic points of $\phi^t|_{\mathcal K}$.

\subsection{Hyperbolic groups}\label{sec:hyperbolicgroups}

Here we repeat only the required notation from \cite[Section 2.2]{VolI} and otherwise we formulate only additional preliminary results that are not contained in that reference. For an overview and  comprehensive background information, we refer the reader to \cite{Gr,Gdh90,kapovichbenakli}. Let $\Gamma$ be a finitely generated group, and fix a  a choice of generators to define the word length function $|\cdot|:\Gamma\to [0,\infty)$. The \emph{stable length} of $\gamma\in \Gamma$ is defined by $|\gamma|_\infty:=\lim_{n\to \infty}\frac{|\gamma^n|}{n}$.  Suppose that $\Gamma$ is hyperbolic (synonyms: \emph{word-hyperbolic}, \emph{Gromov-hyperbolic})  and let $\partial_\infty \Gamma$ be the Gromov boundary of $\Gamma$. Every infinite order element $\gamma\in \Gamma$ has two distinct fixed points $\gamma^+:=\lim_{n\to +\infty}\gamma^n\in\partial_\infty\Gamma$ and $\gamma^-:=\lim_{n\to -\infty}\gamma^n\in\partial_\infty\Gamma$.

The first part of the following lemma generalizes \cite[Lemma 2.10]{VolI} to sequences that do not consist exclusively of infinite order elements.
\begin{lem} \label{lem diverging sequence hyperbolic group}
Consider a sequence $(\gamma_k)$ in $\Gamma$ such that the limits $\gamma_+=\lim_{k\to \infty}\gamma_k\in\partial_\infty\Gamma$ and $\gamma_-=\lim_{k\to \infty}\gamma_k^{-1}\in\partial_\infty\Gamma$ exist and are distinct. Then the following holds:

\begin{enumerate}
\item $\gamma_k$ has infinite order when $k$ is large enough. Furthermore, $\gamma_k^+\to \gamma_+$ and $\gamma_k^-\to\gamma_-$.

\item $|\gamma_k|_\infty\to \infty$.
\end{enumerate} 
\end{lem}
\begin{proof}
(1) Let $U_+,U_-\subset\partial_\infty\Gamma$ be open neighbourhoods of $\gamma_+,\gamma_-$ respectively with disjoint closures, such that $U_-\cup U_+\neq \partial_\infty\Gamma$ (this is only possible if $\partial_\infty\Gamma$ is infinite, but in other cases $\Gamma$ is either finite or virtually isomorphic to $\Z$, the result is straightforward in these cases). From the convergence property of the action on, $\partial_\infty\Gamma$, we have that 
\[ \gamma_k\cdot\big( \partial_\infty\Gamma\setminus U_-\big)\subset U_+ \quad\textrm{ and }\quad  \gamma_k^{-1}\cdot\big( \partial_\infty\Gamma\setminus U_+\big)\subset U_-\]
when $k$ is large enough. Such elements $\gamma_k$ must have infinite order: otherwise, any element $x\notin  U_-$ would satisfy $\gamma_k^{r}\cdot x\in U_+$ where $r$ is the order of $\gamma_k$, thus $x\in U_+$. The fact that
\[ \gamma_k\cdot U_+\subset U_+ \quad\textrm{ and }\quad  \gamma_k^{-1}\cdot U_-\subset U_-\]
also implies that $\gamma_k^+\in U_+$ and $\gamma_k^-\in U_-$. By choosing $U_+$ and $U_-$ arbitrarily small we find that $\gamma_k^+\to \gamma_+$ and $\gamma_k^-\to \gamma_-$.

(2) We know from (1) that $\gamma_k$ has infinite order for large $k$, and that $\gamma_k^+\to \gamma_+$ and $\gamma_k^-\to\gamma_-$. Let $\eta=d_\infty(\gamma_+,\gamma_-)/2$ where $d_\infty$ is a distance on $\partial_\infty\Gamma$ inducing its cone topology. Then $\gamma_k$ is $\eta$-hyperbolic for $k$ large enough (i.e.\ $d_\infty(\gamma_k^+,\gamma_k^-)\geq \eta$), so by \cite[Proposition 2.39]{GMT}, there is $C_1>0$ such that $\vert\gamma_k\vert\leq \ell(\gamma_k)+C_1$ where $\ell(\gamma)$ denotes the translation length of $\gamma$. But according to \cite[Ch.~10, Prop.~6.4]{CDP}, there is a constant $C_2>0$ such that $\ell(\gamma)\leq \vert\gamma\vert_\infty+C_2$ for all $\gamma\in\Gamma$, hence $\vert\gamma_k\vert_\infty\geq \vert\gamma_k\vert - C_1-C_2\to +\infty$.
\end{proof}

\subsection{Semisimple Lie groups and reductive Lie algebras}

Let $G$ be a finite dimensional real Lie group with Lie algebra $\mathfrak{g}.$  Then $G$ admits a canonical real analytic structure such that multiplication, inversion, and the exponential map $\exp: \mathfrak{g}\rightarrow G $ are real analytic maps.  The Killing form $\mathrm{B}_{\mathfrak{g}}: \mathfrak{g}\times \mathfrak{g}\rightarrow \R$ is defined by the trace pairing
\[
\mathrm{B}_{\mathfrak{g}}(X, Y)=\mathrm{tr}(\mathrm{ad}_{X}\circ \mathrm{ad}_{Y}).
\]
The Lie algebra $\mathfrak{g}$ is semisimple if and only if $\mathrm{B}_{\mathfrak{g}}$ is non-degenerate.  In this article, we use the definition that the real Lie group $G$ is \emph{semisimple} if $\mathfrak{g}$ is semisimple and $G$ has finitely many connected components and finite center.

Recall that the real Lie algebra $\mathfrak{g}$ is reductive if it decomposes as a direct sum
\[
\mathfrak{g}=\mathfrak{z}(\mathfrak{g})\oplus \mathfrak{g}_{\mathrm{ss}}
\]
where $\mathfrak{g}_{\mathrm{ss}}$ is semisimple.  In this article, the only Lie groups with non-semisimple reductive Lie algebra that appear are Levi subgroups $L<G$ where $G$ is semisimple.  While there are various definitions of a real reductive Lie group in the literature (see e.g.\ \cite[VII.2]{knapp}), for the aforementioned reason we will not need a general theory of reductive Lie groups.

\section{Parabolic and Levi subgroups, flag manifolds, transverse flag spaces}\label{sec:Lietheory}

Let $G$ be a non-compact real semisimple Lie group. In this paper we will use two common approaches to characterize parabolic and Levi subgroups of $G$: First, via hyperbolic elements and stabilizers (Section \ref{sec:firstapproach}). Second, via restricted roots and Weyl chambers  (Section \ref{sec:classification}). In that section we also make the connection between the two approaches.

\subsection{Parabolic and Levi subgroups as stabilizers}\label{sec:firstapproach}

The Killing form $\mathrm{B}_{\mathfrak{g}}$ on $\mathfrak{g}$  is a non-degenerate indefinite symmetric bilinear form giving a canonical duality $\mathfrak{g}\simeq \mathfrak{g}^{\ast}$ intertwining the adjoint and coadjoint $G$-actions. Via the latter actions, $G$ acts on the Grassmannians $\cG_d(\fg)$, $\cG_d(\fg^\ast)$ for all $0\leq d\leq \dim \g$. An element $X\in \fg$ is called \emph{hyperbolic} if $\ad_X:\fg\to\fg$ is diagonalizable over $\R$. We write $\fg_{\mathrm{hyp}}\subset\fg$ for the subset of hyperbolic elements. For $X\in \fg_{\mathrm{hyp}}$, we consider
  \begin{align}\begin{split}
 \fl_X&:=\ker(\ad_X),\\
 \fp_X&:= \bigoplus_{\lambda\geq 0}\ker(\ad_X-\lambda\id),\\
 L_X&:=\mathrm{Stab}_G(X)=\set{g\in G}{\Ad(g)X=X},\\
 P_X&:=\mathrm{Stab}_G(\p_X)= \set{g\in G}{\Ad(g)\fp_X=\fp_X}.
\label{eq:PLdef}\end{split}
 \end{align}
The following is well-known (see e.g.\ \cite[Ch.~7]{voganorbit2},\cite[VII.7]{knapp}, \cite[Ch.~1.2]{warner72}):
\begin{fact}\label{fact:lXpXLXPXFX}~
\begin{enumerate}
\item $\fl_X,\fp_X\subset\fg$ are Lie subalgebras.
\item $L_X,P_X\subset G$ are closed subgroups whose Lie algebras are $\fl_X$ and $\fp_X$, respectively, and $L_X$ is reductive.
\item $L_X=P_X\cap P_{-X}$. \label{fact:LXPXPminusX}
\end{enumerate}
\end{fact}

\begin{definition}

A \emph{Levi subgroup} of $G$ is a subgroup $L< G$ for which there is $X\in\fg_{\mathrm{hyp}}$ with $L=L_X$.

A \emph{parabolic subgroup} of $G$ is a subgroup $P< G$ for which there is $X\in\fg_{\mathrm{hyp}}$ with $P=P_X$.
 
A \emph{parabolic subalgebra} of $\g$ is a subalgebra $\p\subset \g$ for which there is $X\in\fg_{\mathrm{hyp}}$ with $\p=\p_X$.
\end{definition}

\begin{ex} \label{example1}
In the case of $\fg=\mathfrak{sl}(d,\mathrm{k})$, $\mathrm{k}=\R$ or $\C$, an element $X\in \mathfrak{sl}(d,\R)$ is hyperbolic if and only if it is diagonalizable with real eigenvalues. For some integer $r\in\N$, consider vectors $\mathbf{d}=(d_{1},...,d_{r})\in\N^r$ and $\mathbf x=(x_1,\dots,x_r)\in\R^r$ such that
\[ d_1+\cdots+d_r=d,\qquad d_1x_1+\cdots+d_rx_r=0,\qquad \textrm{and } x_1>\cdots>x_r,\] and the diagonal matrix  
\[ X=\begin{pmatrix}x_1 \mathbf{1}_{d_1} &&\\ &\ddots & \\ && x_r \mathbf{1}_{d_r}  \end{pmatrix}\in\mathfrak{sl}(d,\mathrm{k})\,.\]
Then $L_X$  consists of block-wise diagonal matrices, and $P_X$ (resp. $P_{-X}$) consists of upper (resp. lower) block-wise triangular matrices: 
\[ L_X=\left\lbrace\begin{pmatrix}\ast &&0\\ &\ddots & \\ 0&& \ast  \end{pmatrix}\right\rbrace\,,\quad P_X=\left\lbrace\begin{pmatrix}\ast &&\ast\\ &\ddots & \\ 0&& \ast  \end{pmatrix}\right\rbrace\,,\quad P_{-X}=\left\lbrace\begin{pmatrix}\ast &&0\\ &\ddots & \\ \ast&& \ast  \end{pmatrix}\right\rbrace\,.\]
\end{ex}

\subsubsection{The split part of the center of a Levi subgroup} \label{sec:centerofLevi}
Consider a Levi subgroup $L<G$, and denote by $\fl$ its Lie algebra. By definition there is $X\in \g_{\mathrm{hyp}}$ such that $L=L_X$, but this element is far from unique. Two such elements must however commute, and a good understanding of the center of $L$ will be important later. Its identity component $Z_\circ(L)$  is a connected abelian Lie group, so it splits as a product $Z_\circ(L)=R\times T$ where $R$ is isomorphic to some $\R^n$ and $T$ is a compact torus. While $R$ is not unique, the maximal compact torus $T$ is, as it can be defined as the set of elements $g\in Z_\circ(L)$ such that the closure of $\set{g^n}{n\in\Z}$ is compact. The Killing form of $\g$ and the Lie algebra $\mathfrak t$ of the maximal compact torus $T<Z_\circ(L)$ allow us to define a canonical $\R$-split torus $Z_{\mathrm{split}}(L)$ in the center of $L$.

\begin{definition} \label{def split part center of Levi}
Let $L<G$ be a Levi subgroup. The split part $Z_{\mathrm{split}}(L)$ of its center is the group $Z_{\mathrm{split}}(L)=\exp\left( \mathfrak z_{\mathrm{split}}(\fl)\right)$ where $\fz_{\mathrm{split}}(\fl)\subset \fz(\fl)$ is the orthogonal (with respect to the restriction of the Killing form $\mathrm{B}_\g$ of $\g$) of the Lie algebra $\mathfrak t\subset\fz(\fl)$ of the unique maximal compact torus $T<Z_\circ(L)$.
\end{definition}

Note that for any $X\in\g_{\mathrm{hyp}}$, we have that $X\in \fz_{\mathrm{split}}(\fl_X)$.

\begin{ex} \label{example split part of center}
In the setting of Example \ref{example1}, we find
\[ Z_\mathrm{split}(L_X)=\set{\begin{pmatrix}\lambda_1\mathbf{1}_{d_1} &&0\\ &\ddots & \\ 0&& \lambda_r\mathbf{1}_{d_r}  \end{pmatrix}}{\lambda_1,\dots,\lambda_r\in\R_{>0}\,,~ \lambda_1^{d_1}\cdots \lambda_r^{d_r}=1}\,. \]
\end{ex}

\subsubsection{The structure of parabolic subgroups} \label{sec:structureparabolicsubgroups}

 For $X\in \mathfrak{g}_{\mathrm{hyp}}$, consider the   the Lie subalgebra
 \begin{equation}
 \fn_X:=\bigoplus_{\lambda>0}\ker(\ad_X-\lambda\mathrm{id})\subset \p_X\, , \label{eq:defnX}
\end{equation}  
so that we have further decompositions of  Lie algebras
  \begin{align}\begin{split}
 \fp_X&=\fl_X\oplus\fn_X,\\
 \g&=\fn_{-X}\oplus \fl_X\oplus \fn_X\, .
\label{eq:decompg}\end{split}
 \end{align}
 Then the subgroup
\bq
N_X:=\exp(\fn_X)=\set{g\in G}{\lim_{t\to +\infty} \exp(-tX)g\exp(tX) =e}\label{eq:Ncharact}
\eq
is a simply connected closed Ad-unipotent subgroup of $G$ normalizing $L_X$ and satisfying $N_X\cap L_X=\{e\}$ \cite[Prop.~7.1]{voganorbit2}. 
While the Levi subgroup $L_{X}<P_{X}$ is not uniquely defined as a maximal reductive subgroup of $P_{X},$ the unipotent radical is the unique maximal connected Ad-unipotent normal
subgroup $N_{X}\lhd P_{X},$ and there is a short exact sequence
\[
1\rightarrow N_{X}\rightarrow P_{X}\rightarrow P_{X}/N_{X}\rightarrow 1.
\]
The above short exact sequence is split by $L_{X}\simeq P_{X}/N_{X}$ acting on $N_{X}$ via conjugation and presenting
\bq
P_X=N_X\rtimes L_{X}=\mathrm{Stab}_G(\fn_X)=\set{g\in G}{\Ad(g)\fn_X=\fn_X},\label{eq:Pdef2}
\eq
see \cite[Prop.~7.1]{voganorbit2}, \cite[VII.7,~Prop.~7.82 \& 7.83]{knapp}. Note that we have  $$L_{-X}=L_X,$$ while $$N_{-X}\cap N_X=\{e\}.$$

Given any parabolic subgroup $P$ with Lie algebra $\p$, by definition there is $X \in \mathfrak{g}_{\mathrm{hyp}}$ such that $P=P_X$ and $\p=\p_X$, but as for Levi subgroups this $X$ is far from unique. While the Levi subgroup $L_X$ and its Lie algebra $\fl_X$ will be different for different choices of $X$, the subgroup $N_X$ and (most importantly for us) the subalgebra $\fn_X\subset \p_X$ do not depend on the choice of $X$, as $\fn_X$  is the nilpotent radical of $\p_X$.

So denoting by $\fn_\p$ the nilpotent radical of $\p$, we  arrive at the equalities
\[
P=\mathrm{Stab}_G(\p)=\mathrm{Stab}_G(\nL_\p)
\]
which will be the two main characterizations of parabolic subgroups used throughout this paper. Regarding a parabolic Lie algebra $\p$, we will use the tautological property that $\p$ is the Lie algebra of its own stabilizer.

\begin{ex}
In the setting of Example \ref{example1}, we find that $N_X$ consists of block-wise upper triangular matrices with identity diagonal blocks:
\[  N_{X}=\left\lbrace\begin{pmatrix}\mathbf{1}_{d_1} &&\ast\\ &\ddots & \\ 0&& \mathbf{1}_{d_r}  \end{pmatrix}\right\rbrace~.\]
\end{ex}

\subsection{Flag manifolds and transverse flag spaces} 

We now turn to the homogeneous spaces associated to the groups defined in the previous subsection.
\begin{definition} \label{def: flag manifold}
A \emph{flag manifold} is a $G$-homogeneous space $\cF $ whose point stabilizers are parabolic subgroups.
\end{definition}
The geometric approach to flag manifolds usually consists in interpreting them as orbits in the visual boundary of the Riemannian symmetric space of $G$ (see e.g. \cite[2.17]{eberlein}), but we will use a different description by embedding them in Grassmannian manifolds of the Lie algebra $\g$.  For $X\in\g_{\mathrm{hyp}}$, we consider
\[  \cF_X:= G\cdot \fp_X\subset \cG_{\dim\fp_X}(\fg).\]

By the definition of $P_{X}$, there is an isomorphism $\cF_X\simeq G/P_X$ as analytic $G$-manifolds. Since flag manifolds are compact \cite[VII, Prop. 7.83 (f)]{knapp},  $\cF_X\subset \cG_{\dim\fp_X}(\fg)$ is closed.

\begin{definition} \label{def:standard flag manifold}
A \emph{standard flag manifold} is a subset $\cF\subset \cG_d(\fg)$ (for some $0\leq d\leq \dim\fg$) for which there is $X\in \fg_{\mathrm{hyp}}$ with $\cF=\cF_X$.
\end{definition}
Note that any flag manifold $\cF$ is $G$-equivariantly diffeomorphic to the standard flag manifold $\set{\mathfrak{stab}(\mathbf x)}{\mathbf x\in\F}$ where $\mathfrak{stab}(\mathbf x)$ denotes the Lie algebra of the stabilizer of $\mathbf x$. 

While flag manifolds will play a central role in our arguments, the most important objects will be the open $G$-orbits formed by the transverse pairs in products of opposite flag manifolds.

\begin{definition}Consider two flag manifolds $\cF^+$ and $\cF^-$.  A pair $(\mathbf x^+,\mathbf x^-)\in \cF^+\times\cF^-$ is  \emph{transverse} if there is $X\in \fg_{\mathrm{hyp}}$ such that $\mathrm{Stab}_G(\mathbf x^+)=P_X$ and $\mathrm{Stab}_G(\mathbf x^-)=P_{-X}$. We write $\mathbf x^+\pitchfork\mathbf x^-$ if $(\mathbf x^+,\mathbf x^-)$ is transverse and denote the set of all transverse pairs in $\cF^+\times\cF^-$ by $\Ftr$ (or $\Fpf$ for short, when the factors $\F^+$, $\F^-$ are clear from the context).

The flag  manifolds $\cF^+,\cF^-$ are  \emph{opposite} if $\Ftr\neq\emptyset$. In this case, we call $\Ftr$ a \emph{transverse flag space}.
\end{definition}
In view of \eqref{eq:decompg} one then has for every pair $(\mathbf x^+,\mathbf x^-)\in \F^+\times \F^-$, denoting by $\p^+,\p^-$ the Lie algebras of their stabilizers and $\fn_{\p^+}, \fn_{\p^-}$ their nilpotent radicals,
\bq
\mathbf x^+\transverse \mathbf x^- \iff \fp^+\oplus \fn_{\fp^-}=\fn_{\fp^+}\oplus\fp^-=\g, \label{eq:transverseg}
\eq
which justifies the terminology ``transverse'' and shows that the set $\F^+\trtimes\F^-$, which by definition is a single $G$-orbit, is open and dense in $\F^+\times \F^-$. 

Every transverse flag space $\F^+\trtimes \F^-$ comes with a pair of projections
\bq
\begin{tikzcd}
& \F^+\trtimes \F^- \arrow[dl, swap, "p_{\F^+}"] \arrow[dr, "p_{\F^-}"]
\\
\F^+ & & \F^-
\end{tikzcd}\label{eq:Fpmprojections}
\eq
given by the restriction to  $\F^+\trtimes \F^-$  of the canonical projections $\F^+\times \F^-\to \F^\pm$.

\begin{ex} \label{example flag manifolds}
Consider a  vector $\mathbf{d}=(d_{1},...,d_{r})\in\N^r$ with $d_1+\cdots+d_r=d$. The space $\cF_{\mathbf d}$ of \emph{flags of type} $\mathbf d$  is defined as 
\[ \cF_{\mathbf d}(\mathrm{k}^d)=\set{ V_\bullet\in \prod_{i=0}^{r} \G_{d_1+\cdots+d_i}(\mathrm{k}^d) }{ \{0\}=V_0\subset V_1\subset \cdots \subset V_{r-1}\subset V_r=\mathrm{k}^d }\,.\] 
Just as in Example \ref{example1}, consider a real diagonal matrix $ X=\mathrm{Diag}\left(x_1 \mathbf{1}_{d_1},\dots,x_r \mathbf{1}_{d_r}\right)\in\mathfrak{sl}(d,\mathrm{k})$, with $x_1>\cdots>x_r$. The map
\[ \map{\F_{\mathbf d}(\mathrm{k}^d)}{\F_X}{V_\bullet}{\set{Y\in\mathfrak{sl}(d,\mathrm{k})}{Y\cdot V_i\subset V_i~\forall i}}\]
is an $\SL(d,\mathrm{k})$-equivariant diffeomorphism. The same formula gives an $\SL(d,\mathrm{k})$-equivariant diffeomorphism between $\F_{\iota(\mathbf d)}(\mathrm{k}^d)$ and $\F_{-X}$, where $\iota(\mathbf d)=(d_r,\dots,d_1)$, and the transverse flag space 
\[ \F_{\mathbf d}(\mathrm{k}^d)\trtimes\F_{\iota(\mathbf d)}(\mathrm{k}^d)=\set{\big(V_\bullet,W_\bullet\big)}{V_i\oplus W_{r-i}=\mathrm{k}^d~\forall i}\]
can be identified with the space of \emph{gradings of type} $\mathbf d$
\[ \F_{\mathbf d}^\pitchfork(\mathrm{k}^d) := \set{ E_\bullet\in \prod_{i=1}^r\cG_{d_i}(\mathrm{k}^d)}{ E_1\oplus\cdots\oplus E_{r}=\mathrm{k}^d}\] 
through the map
\[\map{\F_{\mathbf d}(\mathrm{k}^d)\trtimes\F_{\iota(\mathbf d)}(\mathrm{k}^d)}{ \F_{\mathbf d}^\pitchfork(\mathrm{k}^d) }{\big(V_\bullet,W_\bullet\big)}{(V_i\cap W_{r+1-i})_{1\leq i\leq r}}\]
whose inverse maps $E_\bullet\in \F_{\mathbf d}^\pitchfork(\mathrm{k}^d)$ to the pair $(V_\bullet,W_\bullet)$ where $V_i=E_1\oplus\cdots\oplus E_i$ and $W_i=E_r\oplus\cdots\oplus E_{r+1-i}$.

\end{ex}

\subsection{Pseudo-Riemannian structure of transverse flag spaces} \label{sec:pseudoriemannianstructure} 
The tangent space of a transverse flag space $\Fpf=\F^+\trtimes \F^-$ at a point $(\mathbf x^+,\mathbf x^-)$ has a nice description: denoting by $\p^+$ and $\p^-$ the Lie algebras of their respective stabilizers and $\fl=\p^+\cap\p^-$, the projection of $\fn_{\fp^+}\oplus\fn_{\fp^-}\subset\g$ onto $\g/\fl$ is an isomorphism, and the latter can be identified with $T_{(\mathbf x^+,\mathbf x^-)}\Fpf$ via the derivative of the $G$-action. 

This description allows for the definition of a $G$-invariant pseudo-Riemannian metric induced by the  Killing form $\mathrm{B}_\g$ because the subspace $\fn_{\fp^+}\oplus\fn_{\fp^-}\subset\g$ is non-degenerate. More details concerning this pseudo-Riemannian metric and other geometric structures on transverse flag spaces will be discussed in Section \ref{sec:geometric_structures}.

\subsection{Standard parabolic subgroups and flag manifolds}\label{sec:classification}

Let $\theta:\g\to \g$ be a Cartan involution defining a Cartan decomposition $\g=\k\oplus \k^\perp$ into its $+1$ and $-1$ eigenspaces, which are orthogonal with respect to the Killing form $\mathrm{B}_\g$, and let $K<G$ be the maximal compact subgroup with Lie algebra $\k$. It is the fixed point set of the global Cartan involution, which by abuse of notation we again denote by $\theta:G\to G$. The Cartan  decomposition $\g=\k\oplus \k^\perp$ is $\Ad(K)$-invariant.

Choose a maximal abelian subalgebra $\aL$ of $\k^\perp$. Then the Killing form $\mathrm{B}_\g$ restricts to an inner product on $\aL$ and we get the restricted root decomposition $$\fg=\fm\oplus\fa\oplus\bigoplus_{\alpha\in \Sigma}\fg_\alpha,$$ where $\fm=\mathfrak{z}_\k(\aL)$.  Now, we choose a simple system
 $$
\Delta\subset\Sigma,
$$
which defines a closed Weyl chamber $\fa^+=\{Y\in \aL^\ast\,|\,\alpha(Y)\geq 0 \;\forall\;\alpha\in \Delta\}$. The latter is a fundamental domain for the action of the Weyl group $W=N_K(\aL)/Z_K(\aL)$. We also obtain the Iwasawa decomposition $$G=KAN,$$
 where $A=\exp(\aL)$,  $N=\exp(\nL)$ with $\nL:=\bigoplus_{\alpha\in \Sigma^+}\fg_\alpha$ and  $\Sigma^+\subset \Sigma$ consists of all roots that are positive linear combinations of roots in $\Delta$.  Let $w_0\in W$ be the longest   Weyl group element and $\iota_{\mathrm{o}}:\aL\to \aL$, $X\mapsto -w_0(X)$, the opposition involution. The latter has the key property that it maps $\fa^+$ into itself, and it also acts on $\Delta$. For more details on all this, we refer the reader to \cite[Ch.~VII]{knapp}, \cite[Ch.~2]{warner72} and \cite[Ch.~1 \& 2]{wallach}.

\begin{fact}[{\cite[Thm.\ 6.2]{voganorbit2}}]\label{fact:fund}
The inclusions $\aL\hookrightarrow\k^\perp\hookrightarrow \g$ induce bijections between the $W$-orbits in $\aL$, the $\Ad(K)$-orbits in $\k^\perp$, and the hyperbolic $\Ad(G)$-orbits in $\g$. In particular, every hyperbolic adjoint $G$-orbit intersects $\aL^+$ in exactly one element, and one has $\Ad(G)\aL^+=\g_{\mathrm{hyp}}$, $\Ad(K) \aL^+=\k^\perp$.
\end{fact}
Fact \ref{fact:fund} allows us to use the simple system $\Delta$ for classifying parabolic and Levi subgroups as well as the associated flag manifolds and transverse flag spaces. The so-called \emph{standard parabolic subgroups} associated to $\Delta$ can be described in many equivalent ways (see e.g.\ \cite[VII.7]{knapp}, \cite[Ch.~2]{wallach}). For us, the following rather quick description suffices (c.f.\ \cite[VII, Prop.~7.82, 7.83]{knapp}, \cite[Sec.~3.2]{GW12}): Let ${\Theta}\subset \Delta$ be a subset and define 
\bq
\aL_{\Theta}:= \bigcap_{\alpha \in \Delta\setminus {\Theta}}\ker \alpha\subset \aL,\qquad 
 P_{\Theta}:=Z_K(\aL_{\Theta})AN.\label{eq:aln3}
\eq
$P_\Theta$ is the \emph{standard parabolic subgroup} associated to $\Theta$. Its \emph{opposite} is $\overline P_\Theta:=\theta (P_\Theta)$, which is conjugate to $P_{\iota_o(\Theta)}$. The \emph{standard Levi subgroup} $L_\Theta:=P_\Theta\cap \overline P_\Theta$ and its split component are given by  $L_\Theta=Z_G(\aL_\Theta)$,  
\bq
Z_\mathrm{split}(L_\Theta)=\exp(\aL_{\Theta})\subset A,\label{eq:splitcompLeviTheta}
\eq
see \cite[VII, Prop.~7.82]{knapp}. Every parabolic subgroup of $G$ containing the minimal parabolic  subgroup $P_\Delta=Z_K(\aL)AN$ is of the form $P_\Theta$ for a unique subset ${\Theta}\subset \Delta$ \cite[VII, Prop.~7.76]{knapp}.  

Let us now make the connection to Section \ref{sec:firstapproach}: Given $X\in \fa^+$, consider the set
\bq
\Theta_X:=\set{\alpha\in\Delta}{\alpha(X)\neq 0}=\set{\alpha\in\Delta}{\alpha(X)> 0}\subset \Delta.\label{eq:ThetaX}
\eq
Then one finds that $P_X=P_{\Theta_X}$ and $L_X=L_{\Theta_X}$. As for the Lie algebras, one gets
  \begin{align}\begin{split}
 \fp_X&=\fm\oplus\fa\oplus\bigoplus_{\substack{\alpha\in\Sigma\\\alpha(X)\geq 0}}\fg_\alpha,\\
 \fl_X&=\fm\oplus\fa\oplus\bigoplus_{\substack{\alpha\in\Sigma\\\alpha(X)= 0}}\fg_\alpha,\\
 \fn_X&=\bigoplus_{\substack{\alpha\in\Sigma\\\alpha(X)> 0}}\fg_\alpha.
\label{eq:lpXalgebraic}\end{split}
 \end{align}
 It follows that $P_X$ depends only indirectly on $X$ via the set $\Theta_X$. More precisely, it follows from the one-to-one correspondence betweeen standard parabolic subgroups and subsets of $\Delta$  that for all $X,X'\in \fa^+$ we have
\bq
\Theta_X\subset \Theta_{X'} \;\iff\;  P_{X'}< P_{X}.\label{eq:LPincl}
\eq

\begin{definition}\label{def:standardthings}
For  $\Theta\subset\Delta$, the \emph{standard flag manifold  associated to $\Theta$} is $\cF_\Theta:=G\cdot\fp_\Theta$.
\end{definition}

Note that $\F_\Theta$ is in fact a standard flag manifold according to Definition \ref{def:standard flag manifold}, since $\F_\Theta=\F_X$ for any $X\in \aL^+$ such that $\Theta_X=\Theta$. From \eqref{eq:lpXalgebraic} and Fact \ref{fact:fund}, we get the following:
\begin{prop}\label{prop:parabolic subgroup conjugate to standard} Every parabolic subgroup is conjugate to a standard parabolic subgroup $P_\Theta$ for a unique subset $\Theta\subset\Delta$,  and every  flag manifold is $G$-equivariantly diffeomorphic to a unique standard flag manifold.
\end{prop}

Choosing a different root space decomposition provides standard parabolic subgroups that are conjugate to each other. However, the standard flag manifolds actually equal, and not just isomorphic to each other.

\begin{prop} \label{prop:flagTheta}
 For every standard flag manifold $\cF$, there is a unique subset $\Theta\subset\Delta$ such that $\cF=\cF_\Theta$.
\end{prop}
\begin{proof}
By Fact \ref{fact:fund} every $X\in\fg_{\mathrm{hyp}}$ is transported by $\Ad(G)$ to a unique element of $X'\in\fa^+$, and $\F_X=\F_{X'}=\F_{\Theta_{X'}}$ by \eqref{eq:lpXalgebraic}. The uniqueness follows from Proposition \ref{prop:parabolic subgroup conjugate to standard}.
\end{proof}

\begin{ex} \label{example standard parabolics}
In the study of $\SL(d,\mathrm{k})$, $\mathrm{k}=\R$ or $\C$, the usual description (see \cite[p.\ 509]{knapp}) leads to choosing for $\aL\subset \mathfrak{sl}(d,\mathrm{k})$ the space of real diagonal traceless matrices, and choosing as a closed Weyl chamber
 \[  \aL^+=\set{\mathrm{Diag}(x_1,\ldots,x_d)\in\aL}{x_1\geq \cdots\geq x_d }\,,\]
corresponding to simple restricted roots $\Delta=\{\alpha_1,\dots,\alpha_{d-1}\}$ where 
\[ \alpha_j\big(\mathrm{Diag}(x_1,\ldots,x_d)\big):=x_j-x_{j+1}\,.\]
For a matrix $X=\mathrm{Diag}\left(x_1 \mathbf{1}_{d_1},\dots,x_r \mathbf{1}_{d_r}\right)\in\aL^+$ as in Example \ref{example1}, we find \[\Theta_X=\{\alpha_{d_1},\alpha_{d_1+d_2},\dots,\alpha_{d_1+\cdots+d_{r-1}}\}\,.\]
The corresponding standard parabolic subgroups are the block-wise upper triangular groups described in Example \ref{example1}, and the opposition involution acts on $\Delta$ as $\iota_o(\alpha_j)=\alpha_{d-j}$.
\end{ex}

\subsection{The hierarchy of flag manifolds}\label{sec:Delta} The correspondence between standard flag manifolds and subsets of simple restricted roots endows the set of standard flag manifolds with a partial ordering, corresponding to inclusion of subsets of $\Delta$. We shall see in this section how to describe this partial ordering without fixing a restricted root decomposition. 

\begin{lem} \label{lem uniqueness of equivariant maps between dominated flag manifolds} Let $\F,\F'$ be two flag manifolds. If there exists a $G$-equivariant map $\F\to\F'$, then it is unique and it is a fiber bundle projection.
\end{lem}
\begin{proof}
Consider two $G$-equivariant maps $\varphi,\psi:\F\to\F'$. Fix a restricted root space decomposition, as well as an element $\mathbf x_0\in\F$ whose stabilizer $\mathrm{Stab}_G(\mathbf x_0)$ is a standard parabolic subgroup $P_\Theta$ as defined in section \ref{sec:classification}. Then the stabilizers of $\varphi(\mathbf x_0)$ and $\psi(\mathbf x_0)$ are both parabolic subgroups of $G$ containing $P_\Theta$, so according to \cite[VII, Prop.~7.76]{knapp}, they are also standard parabolic subgroups. But they are conjugate to each other, so $\varphi(\mathbf x_0)$ and $\psi(\mathbf x_0)$ have the same stabilizer $P_{\Theta'}$ for some $\Theta'\subset \Theta$, thus $\varphi(\mathbf x_0)=\psi(\mathbf x_0)$ (because parabolic subgroups are their own normalizers), and the equality $\varphi=\psi$ follows from the $G$-equivariance. In the coset description $\varphi$ corresponds to the canonical projection $G/P_{\Theta}\to G/P_{\Theta'}$.
\end{proof}

In the case of isomorphic (i.e.\ $G$-equivariantly diffeomorphic) flag manifolds, Lemma \ref{lem uniqueness of equivariant maps between dominated flag manifolds} has a strong  implication: the classification up to isomorphisms of Proposition \ref{prop:parabolic subgroup conjugate to standard} is actually a classification up to \emph{unique}  isomorphism. Practically, this means that choosing between working with abstract flag manifolds or standard flag manifolds is merely a matter of taste. Lemma \ref{lem uniqueness of equivariant maps between dominated flag manifolds} is also a motivation for the following partial ordering of flag manifolds:

\begin{definition}\label{def:domination flag manifolds}\,
 Let $\F,\F'$ be two flag manifolds. We say that $\F$ \emph{dominates} $\F'$, and denote it by $\F'\prec\F$, if there is a $G$-equivariant map $\Pi^{\mathcal F}_{\mathcal F'}:\F\to \F'$.
\end{definition}

The hierarchy defined by this partial ordering corresponds to inclusion for subsets of simple restricted roots.

\begin{prop} \label{prop domination standard flag manifolds}
Fix a restricted root space decomposition of $\g$. For any two subsets $\Theta, \Theta'\subset \Delta$ we have $\cF_\Theta\prec \cF_{\Theta'}$ iff $\Theta\subset \Theta'$.
\end{prop}
\begin{proof}
It follows from \eqref{eq:LPincl}.
\end{proof}

So far there is an important  operation on simple restricted roots that we have not described in terms of simple flag manifolds: the opposition involution. It happens to coincide with the notion of opposite flag manifolds.
\begin{prop}\label{prop:uniqueopp}
Every standard flag manifold has a unique opposite standard flag manifold. Given a restricted root space decomposition and a subset $\Theta\subset \Delta$, the standard flag manifold opposite to  $\cF_\Theta$ is  $\cF_{\iota_\mathrm{o}(\Theta)}$. 
\end{prop}
\begin{proof}
Let $\F$ be a standard flag manifold. Proposition \ref{prop:flagTheta} tells us that for a fixed choice of restricted root space decomposition with closed Weyl chamber $\aL^+$ there is $X\in \aL^+$ such that $\F=\F_X=\F_{\Theta_X}$ in the notation of Definition \ref{def:standardthings}, and the set ${\Theta_X}$ determines $\F$ uniquely. Now $\F_{-X}$ is opposite to $\F$, and by Fact \ref{fact:fund} the adjoint $G$-orbit of $-X$ intersects $\fa^+$ in a unique element $X'$, so that $\F_{-X}=\F_{\Theta_{X'}}$ is uniquely determined by $\Theta_{X'}$. Concretely $X'$ is given by $X'=\iota_\mathrm{o}(X)$ and we get $\Theta_{X'}=\iota_\mathrm{o}(\Theta_{X})$. In particular, $\Theta_{X'}$ is uniquely determined by $\Theta_X$ . If  $Y\in \aL^+$ is another element such that $\F=\F_Y$, then by Proposition \ref{prop:flagTheta} we have $\Theta_Y=\Theta_X$ and thus $\F_{-Y}=\F_{\iota_\mathrm{o}(\Theta_{X})}=\F_{-X}$. Thus, $\F_{\iota_\mathrm{o}(\Theta_{X})}$ is the unique opposite of $\F_{\Theta_X}$ and the proof is finished.
\end{proof}
 Given a standard flag manifold $\F$, we write $\overline{\F}$ for its opposite standard flag manifold. The notions of ``opposition'' for standard flag manifolds and parabolic subgroups are compatible in the sense that the opposite of $\F_\Theta=G\cdot \p_\Theta$ is $\overline \F_\Theta=\F_{\iota_o(\Theta)}=G\cdot \overline \p_\Theta$, where $\overline\p_\Theta=\theta(\p_\Theta)$ is the image under the Cartan involution of the Lie algebra of  $P_\Theta$.

\begin{cor} \label{cor domination and opposition}
Consider two transverse flag spaces $\F_1^\pitchfork=\F^+_1\trtimes\F_1^-$ and $\F_2^\pitchfork=\F^+_2\trtimes\F_2^-$. Then $\F_1^+\prec \F_2^+$ if and only if $\F_1^-\prec \F_2^-$. Furthermore, if these dominations occur, then
\[\Big(\Pi^{\F^+_2}_{\F_1^+}(\mathbf x^+),\Pi^{\F_2^-}_{\F_1^-}(\mathbf x^-)\Big)\in\F_1^\pitchfork\qquad \forall (\mathbf x^+,\mathbf x^-)\in \F_2^\pitchfork \,. \]
\end{cor}
\begin{proof}
The equivalence follows from Proposition \ref{prop:flagTheta} and Proposition \ref{prop domination standard flag manifolds}. For the second statement, we may fix a restricted root space decomposition and assume that $\mathrm{Stab}_G(\mathbf x^+)=P_{\Theta_2}$ and $\mathrm{Stab}_G(\mathbf x^-)=\overline P_{\Theta_2}$ for some subset $\Theta_2\subset\Delta$. Simplify the notations to
\[ \pi^+:=\Pi^{\F^+_2}_{\F_1^+}~,\qquad \pi^-:=\Pi^{\F_2^-}_{\F_1^-}\,. \]
 Now $\mathrm{Stab}_G(\pi^+(\mathbf x^+))$ is a parabolic subgroup of $G$ containing $P_{\Theta_2}$, so by \cite[VII, Prop.~7.76]{knapp} there is a subset $\Theta_1\subset\Delta$ such that $\mathrm{Stab}_G(\pi^+(\mathbf x^+))=P_{\Theta_1}$. Similarly there is a subset $\Theta'_1\subset\Delta$ such that $\mathrm{Stab}_G(\pi^-(\mathbf x^-))=\overline P_{\Theta'_1}$. But the fact that the flag manifolds $\F_1^+=\F_{\Theta_1}$ and $\F_1^-=\F_{\iota_o(\Theta'_1)}$ are opposite implies that $\Theta'_1=\Theta_1$, therefore $(\pi^+(\mathbf x^+),\pi^-(\mathbf x^-))\in\F_1^\pitchfork$.
\end{proof}

\subsection{Simple flag manifolds} \label{sec:simpleflagmanifolds}
 The partial ordering introduced in Definition \ref{def:domination flag manifolds} has a  minimal element $\{\fg\}\in\cG_{\dim\fg}(\fg)$ given by a one point space (associated to the empty set $\emptyset\subset\Delta$ in a restricted root space decomposition). We shall call a  flag manifold  \emph{non-trivial} if it does not consist of a single point. 
\begin{definition} A  \emph{simple flag manifold} is a non-trivial standard flag manifold that does not dominate any other non-trivial flag manifold. 
\end{definition}

Once a restricted root space decomposition has been fixed, Proposition \ref{prop domination standard flag manifolds} implies that the simple flag manifolds are exactly the $\F_{\alpha}:=\F_{\{\alpha\}}$ for $\alpha\in\Delta$, thus providing a bijection
\begin{align}\label{eq:bijalphas}\begin{split}
 \Delta&\simeq\{\text{simple flag manifolds}\}\\
 \alpha&\mapsto \F_\alpha.\end{split}
\end{align}

The  opposition involution $\iota_{\mathrm o}$ on $\Delta$ can be transported to the set of simple flag manifolds by the bijection \eqref{eq:bijalphas}, but thanks to Proposition \ref{prop:uniqueopp} it has an intrinsic meaning: it consists of the map sending a flag manifold to its unique opposite simple flag manifold.

In many occurences, all that will be used from a root space decomposition is the possibility of indexing flag manifolds by subsets of $\Delta$ and the fact that indexing opposite flag manifolds corresponds to the opposition involution. In this case, \eqref{eq:bijalphas} allows us to forget about root space decompositions and just consider $\Delta$ as any finite set indexing simple flag manifolds (e.g. the set of simple flag manifolds itself), and define the opposition involution by the relation $\F_{\iota_o(\alpha)}=\overline\F_\alpha$ for $\alpha\in \Delta$. Then any flag manifold $\F$ corresponds to a subset  $\Theta_\F\subset \Delta$ which is the set of simple flag manifolds dominated by $\F$:
\bq
\begin{split}
\Theta_{\F}&:=\set{\alpha\in \Delta}{\F_\alpha\prec \F}.\label{eq:nFpf}
\end{split}
\eq
Concretely, this means that we can either start our discussion with a subset $\Theta\subset \Delta$ (in a fixed restricted root space decomposition) and consider the associated standard flag manifold $\F_\Theta$, or start with a flag manifold $\F$ and consider the set $\Theta_\F$ from \eqref{eq:nFpf}. By using the notations $\F$ for the flag manifold and $\Theta\subset \Delta$ for the subset, both point of views lead to the same subsequent discussion.

A good example of a property that can be expressed with either point of view is the definition of a $G$-equivariant fiber bundle projection 
\bq
\Pi^\F_\alpha:\F\to\F_\alpha \label{eq:FFalphaprojection}
\eq 
for each $\alpha\in \Theta$. From the point of view of a fixed restricted root decomposition, it comes from the inclusion $P_\Theta<P_{\{\alpha\}}$ and the identifications $\F\simeq G/P_\Theta$ and $\F_\alpha\simeq G/P_{\{\alpha\}}$. From the intrinsic point of view on simple flag manifolds, the existence of a $G$-equivariant map $\F\to\F_\alpha$ for $\alpha\in \Theta$ is the definition of $\Theta$, and Lemma \ref{lem uniqueness of equivariant maps between dominated flag manifolds} guarantees its uniqueness. It also has a simple meaning: $\Pi_\alpha^\F(\mathbf x)$ is the unique element of  $\F_\alpha\subset\mathcal G_{d_\alpha}(\g)$ containing the Lie algebra of the stabilizer of $\mathbf x\in\F$.

In the following sections, a restricted root space decomposition will only be introduced when further algebraic considerations are necessary.

\begin{ex} \label{example simple flag manifolds}
Carrying on with the notations of Example \ref{example standard parabolics}, given a dimension vector $\mathbf d=(d_1,\dots,d_r)\in\N^r$ with $d_1+\cdots+d_r=d$, the space $\F_{\mathbf d}(\mathrm{k}^d)$ of flags of type $\mathbf d$ introduced in Example \ref{example flag manifolds} is $\SL(d,\mathrm{k})$-equivariantly diffeomorphic to the standard flag manifold corresponding to the subset
\[\Theta_{\mathbf d}:=\{\alpha_{d_1},\alpha_{d_1+d_2},\dots,\alpha_{d_1+\cdots+d_{r-1}}\}\,.\]
Simple flag manifolds correspond to Grassmannians $\G_n(\mathrm{k}^d)=\cF_{\{n,d-n\}}(\mathrm{k}^d)$, and the equivariant projections $\Pi^{\mathbf d}_j:\F_{\mathbf d}(\mathrm{k}^d)\to \G_{d_1+\cdots+d_j}(\mathrm{k}^d)$ are defined by
\[ \Pi^{\mathbf d}_j\big( V_\bullet\big) = V_j\,.\]
\end{ex}

\subsection{Flag manifolds and the Jordan decomposition}\label{sec:defjordanclassfcn}
 The \emph{Jordan decomposition}  tells us that every group element $g\in G$ can be uniquely written as a product $g=g_eg_hg_u$, where $g_e$, $g_h$, and $g_u$ are elliptic, hyperbolic, and unipotent, respectively and commute with each other (see  \cite[Prop.~2.1]{kostant}). Here $g_e$ being elliptic means that $\Ad(g_e)$ is diagonalizable over $\C$ with eigenvalues of absolute value $1$, $g_h$ being hyperbolic means that $g_h=\exp(X_h)$ with $X_h\in \g$ hyperbolic (i.e., $\ad_X$ is diagonalizable over $\R$, as in Section \ref{sec:firstapproach}) and $g_u$ being unipotent means that $g_u=\exp(X_n)$ with $X_n\in \g$ nilpotent.

\subsubsection{Jordan projections in \texorpdfstring{$G$}{G}}\label{sec:JordanG}

Let us fix a restricted root space decomposition as in Section \ref{sec:classification}, in particular we fix an Iwasawa decomposition $G=KAN$ which, in contrast to the Jordan decomposition, involves making choices. Then a group element being elliptic (resp.\ hyperbolic, unipotent) means that it is conjugate to an element in $K$ (resp.\ $A$, $N$). 

\begin{definition}Let $g\in G$. The \emph{Jordan projection}\footnote{The Jordan projection is sometimes also called \emph{Lyapunov projection}.} of $g$ with respect to the  Weyl chamber $\fa^+$ is the element $\lambda(g)\in \fa^+$ such that $g_h$ is conjugate to $\exp(\lambda(g))$.
\end{definition}
Fact \ref{fact:fund} tells us that the Jordan projection is well-defined and has the property
\bq
\lambda(g^{-1})=\iota_o(\lambda(g))\qquad\forall\; g\in G.\label{eq:Jordanginverse}
\eq

\subsubsection{Fundamental weights}
Fixing a  restricted root space decomposition provides a  vector $X_{\alpha}\in\g_{\mathrm{hyp}}$ such that $\F_\alpha=\F_{X_{\alpha}}$ for each  $\alpha\in \Delta$. Indeed, consider the fundamental weight $w_\alpha\in\fa^*$ defined by 
\bq
\Bg(w_\alpha,\alpha)=\frac{1}{2}\Bg(\alpha,\alpha),\qquad \Bg(w_\alpha,\alpha')=0\quad \forall\;\alpha'\in \Delta\setminus \{\alpha\}.\label{eq:def_fundamentalweights}
\eq
Let $X_{\alpha}\in \aL$ be the $\mathrm{B}_\g$-dual of $w_\alpha$ (i.e.\ the co-root of $\alpha$). Then $X_{\alpha}\in \aL^+\cap \aL_\Theta$ and using the notation of \eqref{eq:ThetaX} we find $\Theta_{X_{\alpha}}=\{\alpha\}$, therefore
\begin{equation}
\F_\alpha=\F_{X_{\alpha}}~. \label{eq:Falpha fundamental weight}
\end{equation}
Given a subset $\Theta\subset\Delta$, the vectors $X_{\alpha}\in \aL^+$ for $\alpha\in\Theta$ form a basis of $\aL_\Theta=\cap_{\alpha\in\Delta\setminus\Theta}\ker\alpha$, and from the definition of the fundamental weights $w_\alpha$ one easily infers that their restricttions to $\aL_\Theta$ form a basis of the dual space $\aL_\Theta^*$:

\begin{lem}\label{lem:walphaaTheta}For all $\alpha\in \Theta$ and $X\in \aL$  one has $w_\alpha(X)=w_\alpha(p_\Theta(X))$, where $p_\Theta:\aL\to\aL_\Theta$ is the orthogonal projection with respect to the Killing form $\mathrm{B}_\g$.\qed
\end{lem}

\subsubsection{Class functions on \texorpdfstring{$G$}{G}}

A restricted root space decomposition is often used in order to define certain class functions (i.e.\ functions on $G$ that are invariant under conjugation) by evaluating some element of $\aL^\ast$ on the Jordan projection $\lambda(g)\in\aL$ of an element $g\in G$. We will see that they can be defined using a flag manifold as the starting point. The first example consists in applying a simple restricted root.

\begin{definition}\label{def:spectralradiusfcn}
Given $\alpha\in \Delta$,  the $\alpha$-proximality class function is
\[ \mathcal{P}_{\alpha}:\map{G}{\R_{\geq 0}}{g}{\alpha(\lambda(g)).}\]
\end{definition}

\begin{lem}\label{lem:rootclassfunctionindependence} The class function $\mathcal P_{\alpha}$ only depends on the simple flag manifold $\F_\alpha$, i.e.\ it is independent of  the restricted root space decomposition used to define it.
\end{lem}
\begin{proof}
As a consequence of Fact \ref{fact:fund}, any two possible choices of Weyl chambers $\aL_1^+,\mathfrak{a}_2^+\subset \g$ are related by $\mathfrak{a}_2^+=\Ad(g_0)\aL_1^+$ for some $g_0\in G$. Now, the corresponding Jordan projections $\lambda_i(g)$ and restricted roots $\alpha_i\in \aL_i^*$  defined using  $\aL^+_i$, $i=1,2$, differ by
\[
\lambda_2(g)=\Ad(g_0)\lambda_1(g)~;\qquad \alpha_2=\alpha_1\circ \Ad(g_0^{-1})~;
\]
so that $\alpha_1(\lambda_1(g))=\alpha_2(\lambda_2(g))$ as claimed.
\end{proof}
This ultimately means that $\mathcal P_{\alpha}$ has a geometric meaning, this will be made more precise in Lemma \ref{lem:rootclassfunction is spectral radius of derivative}.

Next, we consider an arbitrary  flag manifold $\F$ and the associated subset $\Theta\subset\Delta$. Let $Y_\Theta\in\fa_\Theta$ be the unique vector such that $\alpha(Y_\Theta)=1$ for all $\alpha\in\Theta$ and 
 define
 \[\Sigma^+_\Theta:=\set{\alpha\in\Sigma}{\alpha(Y_\Theta)>0 }, \]
 so that $$\fn_{\fp_\Theta}=\bigoplus_{\alpha\in\Sigma^+_\Theta}\g_\alpha.$$
 
 \begin{definition}\label{def:JacbobF} For a flag manifold $\F$ associated to a subset $\Theta\subset\Delta$, the \emph{Jacobian class function of $\F$} is
 \bq
\mathcal J_\F:\map{G}{\R_{\geq 0}}{g}{2\varrho_{\Theta}\big(p_\Theta(\lambda(g))\big),}\label{eq:lambdaF}
\eq
where $p_\Theta:\fa\to\fa_\Theta$ is the Killing-orthogonal projection and
\bq
\varrho_{\Theta}:=\frac{1}{2}\sum_{\alpha\in \Sigma^+_\Theta}\dim \fg_\alpha \alpha\quad \in \aL^\ast.\label{eq:wF}
\eq
\end{definition}
Here again the geometric terminology will be justified later, in Lemma \ref{lem: determinant proximal element on a  flag manifold}.

\begin{lem}\label{lem:indep} The class function $\mathcal J_{\F}$ only depends on the  flag manifold $\F$, i.e.\ it is independent of the restricted root space decomposition used to define it.
\end{lem}
\begin{proof}The same proof as for Lemma \ref{lem:rootclassfunctionindependence} applies here.
\end{proof}

The two class functions introduced above satisfy a simple symmetry when passing to opposite flag manifolds.

\begin{lem}\label{lem:Jordanfcnopp} The proximality class functions of $\alpha\in\Delta$ and its opposite $\iota_o(\alpha)$ are related by
\[
\mathcal P_{\alpha}(g)=\mathcal P_{\iota_o(\alpha)}(g^{-1})\qquad \forall\; g\in G.
\]
The Jacobian class functions of a pair of opposite flag manifolds $\F^+,\F^-$ are related by
\[
\mathcal J_{\F^+}(g)=\mathcal J_{\F^-}(g^{-1})\qquad \forall\; g\in G.
\]
\end{lem}
\begin{proof}
Both follow from the simple observation that
\[ \iota_o(\alpha)\big( \lambda(g)\big) = \alpha\big(\iota_o(\lambda(g))\big)=\alpha\big(\lambda(g^{-1})\big),\]
where we used \eqref{eq:Jordanginverse} in the last step.
\end{proof}

For a simple flag manifold $\F_\alpha$ the Jacobian  and proximality class functions $\mathcal J_{\F_\alpha}$ and $\mathcal P_\alpha$ are in general not multiples of each other because  the element $\varrho_{\{\alpha\}}$ turns out to be a multiple of the  fundamental  weight $w_\alpha$ rather than the root $\alpha$:
\begin{lem}\label{lem:wmalpha}
For each $\alpha\in\Delta$ there is a positive integer $m_\alpha\in \N=\{1,2,\ldots\}$ such that
\bq
\varrho_{\{\alpha\}}=\frac{m_\alpha}{2} w_\alpha.\label{eq:defmalpha}
\eq
\end{lem}
\begin{proof}First we rewrite the index set $\Sigma^+_{\{\alpha\}}$ in the sum defining $\varrho_{\{\alpha\}}$. Note that the $\mathrm{B}_\g$-dual $X_{\alpha}$ of $w_\alpha$ lies in $\aL^+$. For all $\beta\in \Sigma$ we have $\beta(X_{\alpha})=\mathrm{B}_\g(\beta,w_\alpha)$ hence
\bq
2\varrho_{\{\alpha\}}=\sum_{\substack{\beta\in \Sigma:\\ \mathrm{B}_\g(\beta,w_\alpha)> 0}}\dim \fg_\beta \beta.\label{eq:wXalpha}
\eq
Now, given any two simple roots $\alpha',\alpha''\in \Delta$, the root reflection $r_{\alpha'}:\aL^\ast\to\aL^\ast$ is an orthogonal transformation with respect to $\mathrm{B}_\g$  acting on the weight $w_{\alpha''}$ according to
\bq
r_{\alpha'}(w_{\alpha''})=w_{\alpha''}-\delta_{\alpha',\alpha''}{\alpha''},\label{eq:salphawalpha}
\eq
where $\delta_{\alpha',\alpha''}=1$ if $\alpha'=\alpha''$ and $\delta_{\alpha',\alpha''}=0$ otherwise. In particular, when $\alpha'\neq \alpha$, every $\beta\in \Sigma$ satisfies
\[
\mathrm{B}_\g(\beta,w_\alpha)>0\iff \mathrm{B}_\g(r_{\alpha'}(\beta),w_\alpha)>0.
\]
Thus \eqref{eq:wXalpha} shows that $\varrho_{\F_\alpha}$ is invariant under all root reflections besides that of $\alpha$:
\bq
r_{\alpha'}(\varrho_{\{\alpha\}})=\varrho_{\{\alpha\}}\qquad \forall\; \alpha'\in \Delta\setminus \{\alpha\}.\label{eq:salphawXalpha}
\eq
This implies that $2\varrho_{\{\alpha\}}$ is an integer multiple of $w_\alpha$. Indeed, since the fundamental weights form a basis of $\aL^\ast$ in which each root has integer coefficients, and $2\varrho_{\{\alpha\}}$ is by definition a linear combination of roots weighted by integers, we can write 
$$
2\varrho_{\{\alpha\}}=\sum_{\alpha'\in \Delta}m_{\alpha'} w_{\alpha'}
$$
with unique coefficients $m_{\alpha'}\in \Z$. If $\alpha'\in \Delta\setminus \{\alpha\}$, then  \eqref{eq:salphawXalpha} and \eqref{eq:salphawalpha} imply 
\[
m_{\alpha'} = m_{\alpha'}(1-C_{\alpha',\alpha'}),
\]
where $C_{\alpha',\alpha'}$ is the coefficient in front of $w_{\alpha'}$ when writing $\alpha'$ as a linear combination of the fundamental weights. Recalling the definition of $w_{\alpha'}$ one finds that $C_{\alpha',\alpha'}=2$, so $m_{\alpha'}=0$. 

It remains to show that $m_{\alpha}>0$. To this end, we apply $\mathrm{B}_\g$ to  $2\varrho_{\{\alpha\}}$ and $w_\alpha$:
\bqn
\mathrm{B}_\g(2\varrho_{\{\alpha\}},w_\alpha)= m_{\alpha}\underbrace{\mathrm{B}_\g(w_\alpha,w_\alpha)}_{>0} =\sum_{\substack{\beta\in \Sigma:\\ \mathrm{B}_\g(\beta,w_\alpha)> 0}}\underbrace{\mathrm{B}_\g(\beta,w_\alpha)}_{>0}\underbrace{\dim \fg_\beta}_{>0}.
\eqn
The sum's index set is non-empty, it contains $\alpha$. Thus $m_\alpha>0$ and the proof is finished.
\end{proof}
For general flag manifolds we have the following fact about Jacobian class functions:
\begin{lem} \label{lem:decomposition Jacobian class function}
Let $\F$ be a flag manifold associated to a subset $\Theta\subset\Delta$. There is a collection of positive rational numbers $\mathbf r=(r_\alpha)_{\alpha\in\Theta}\in\Q_{>0}^\Theta$ such that $\mathcal J_{\F}=\sum_{\alpha\in\Theta}r_\alpha \mathcal J_{\F_{\alpha}}$.
\end{lem}
\begin{proof}
The linear form $\varrho_\Theta\circ p_\Theta\in \fa^\ast$ vanishes on $\ker p_\Theta=\aL_\Theta^\perp=\cap_{\alpha\in\Theta}\ker w_\alpha$, so from Lemma \ref{lem:wmalpha} and simple linear algebra we get the existence of a tuple of real numbers $\mathbf r=(r_\alpha)_{\alpha\in\Theta}\in\R^\Theta$ such that $\mathcal J_{\F}=\sum_{\alpha\in\Theta}r_\alpha \mathcal J_{\F_{\alpha}}$. In fact these numbers must be rational because weights and roots are rational linear combinations of each other and the ratio obtained in \eqref{eq:defmalpha} is rational. Their  positivity comes from the fact that $X_{w_\beta}\in\aL^+\setminus\{0\}$ for any simple restricted root $\beta$.
\end{proof}

\begin{ex} \label{example fundamental weights}
Still using the notations of Example \ref{example standard parabolics} for $\SL(d,\mathrm{k})$, the fundamental weight $w_{\alpha_j}$ corresponding to the simple restricted root $\alpha_j$ is given by
\[ w_{\alpha_j}\Big(\mathrm{Diag}(x_1,\dots,x_d)\Big)=x_1+\cdots+x_j\,,\]
and we find that $m_{\alpha_j}=d$ for $\mathrm{k}=\R$ and $m_{\alpha_j}=2d$ for $\mathrm{k}=\C$.
\end{ex}

\subsubsection{Jordan projections within a  Levi subgroup} \label{sec Jordan decomposition within Levi}
The Jordan decomposition also holds for reductive groups \cite[Prop.~2.10]{voganorbit}, \cite[Ch.~9]{Milne2017}. Considering a  Levi subgroup $L<G$, this means that the Jordan decomposition $g=g_eg_hg_u$ in $G$ of $g\in L$ produces $g_e,g_h,g_u\in L$. 

Since we still consider a fixed restricted root space decomposition, let us consider the standard Levi subgroup $L_\Theta$ associated to a set $\Theta\subset\Delta$ of simple restricted roots, and investigate the difference between the Jordan projection $\lambda_{L_\Theta}$ within the reductive group $L_\Theta$ and the restriction to $L_\Theta$ of the Jordan projection $\lambda:G\to\aL^+$ of $G$. The first is that they do not have the same target space. Indeed, $\aL\subset\fl_\Theta$ is still a maximal abelian subspace of $\fl_\Theta\cap\mathfrak k^\perp$, but $\aL^+$ is not a fundamental domain of the corresponding Weyl group $W_\Theta$. Instead we consider
\begin{equation} \aL^+_{L_\Theta}:=\aL_\Theta\oplus\aL_{\Delta\setminus\Theta}^+=\set{X\in\aL}{\forall\alpha\in\Delta\setminus\Theta~\alpha(X)\geq0}\subset\aL~, \label{eq:aLTheta} 
\end{equation}
 where $\aL_\Theta=\cap_{\alpha\in\Delta\setminus\Theta}\ker\alpha$ as before, and $\aL_{\Delta\setminus\Theta}^+:=\aL_{\Delta\setminus\Theta}\cap\aL^+$. The Jordan projection $$\lambda_{L_\Theta}:L_\Theta\to \aL_{L_\Theta}^+$$ is  the map sending $g\in L_\Theta$ to the unique $X\in \aL_{L_\Theta}^+$ such that $\exp(X)$ is conjugate within $L_\Theta$ to the hyperbolic component $g_h\in L_\Theta$ of $g$. Note that $\aL^+_{L_\Theta}$ contains the full linear subspace $\aL_\Theta$. Since this space lies in the center of $\fl_\Theta$, we cannot conjugate its elements into $\aL_\Theta^+=\aL_\Theta\cap\aL^+$ within $L_\Theta$. In particular, for $X\in \aL_{L_\Theta}^+$, the  vectors $X=\lambda_{L_\Theta}(\exp(X))$ and $\lambda(\exp(X))\in \aL^+$  are usually  distinct, and we have the equivalence
\begin{equation}
\lambda(\exp(X))=X \iff X\in \aL^+ ~. \label{eq: equality case between Jordan projections}
\end{equation}

The Jordan projection $\lambda_{L_\Theta}:L_\Theta\to \aL_{L_\Theta}^+$ can be composed with the Killing-orthogonal projection $p_\Theta:\fa\to\fa_\Theta$, thus defining a projection into the split part $\fa_\Theta=\fz_{\mathrm{split}}(\fl_\Theta)$ of the center of the Lie algebra $\fl_\Theta$ of $L_\Theta$.
\begin{definition}  The \emph{split central projection} of $L_\Theta$ is the  map
\begin{equation}
 \lambda_{L_\Theta}^Z=p_\Theta\circ\lambda_{L_\Theta}:L_\Theta\to \aL_\Theta=\fz_{\mathrm{split}}(\fl_\Theta)\,. \label{eq:centralprojectionstandardlevi}
 \end{equation}
\end{definition}
 A noticeable difference with Jordan projections considered so far is that $\lambda_{L_\Theta}^Z$ is a group homomorphism onto the additive group $\aL_\Theta$. Indeed, the Langlands decomposition \cite[VII.7, Prop.~7.82]{knapp} tells us there is a reductive subgroup $M_\Theta<L_\Theta$ with compact center such that $L_\Theta$ decomposes into the direct product
\begin{equation}
L_\Theta= Z_\mathrm{split}(L_\Theta) M_\Theta\,, \label{eq:direct product levi}
\end{equation}
and $\exp\circ \lambda_{L_\Theta}^Z$ is the projection onto the first factor in this direct product. The translation of Definition \ref{def:JacbobF} in the framing of $L_\Theta$ therefore gives not just a class function, but an additive character of $L_\Theta$.

\begin{definition}Let $L_\Theta<G$ be a standard Levi subgroup. The \emph{Jacobian character of $L_\Theta$} is the function
\[ \mathbf J_{L_\Theta}:\map{L_\Theta}{\R}{g}{2\varrho_\Theta\big(\lambda^Z_{L_\Theta}(g)\big).}\]
\end{definition}

Since any element $X\in \g_{\mathrm{hyp}}$ is contained in a Weyl chamber $\aL^+$ for an appropriate root space decomposition, any Levi subgroup $L<G$ can be considered as a standard Levi subgroup, thus defining a split central projection $\lambda^Z_L\in\Hom(L,\fz_{\mathrm{split}}(\fl))$ and a  Jacobian character $\mathbf J_L\in\Hom(L,\R)$. 

\begin{lem} The split central projection $\lambda^Z_L\in\Hom(L,\fz_{\mathrm{split}}(\fl))$ and the Jacobian character $\mathbf J_L\in\Hom(L,\R)$ of a Levi subgroup $L<G$ does not depend on the restricted root space decomposition used to define it.
\end{lem}
\begin{proof}
Repeat the proof of Lemma \ref{lem:rootclassfunctionindependence} once again.
\end{proof}

\subsection{Linear actions on tangent spaces of flag manifolds}\label{sec:linact}

Consider a pair of opposite flag spaces $\F^+,\F^-$ and the transverse flag space $\Fpf=\F^+\trtimes \F^-$, and let $\mathbf{x}=(\mathbf x^+,\mathbf x^-)\in \Fpf$. Write $\p^\pm$ for the Lie algebras of the stabilizers of $\mathbf x^\pm$, and $\fn_{\p^\pm}$ for their nilpotent radicals. Recall from \eqref{eq:transverseg} that we have the two direct sum  decompositions
\bq
\fp^+\oplus \fn_{\fp^-}=\fn_{\fp^+}\oplus\fp^-=\g.\label{eq:decomp3905}
\eq
Every element $X\in \fg$ induces fundamental vector fields on $\F^\pm$ and $\Fpf$, which we shall all denote by $\overline X$, via the formula $\overline{X}(x)= \left.\frac{d\,}{dt}\right\vert_{t=0} \exp(tX)\cdot x$. The kernel of the evaluation map
\bq\begin{split}
 \g&\to T_{\mathbf x^\pm}\cF^\pm\\
 X&\mapsto\overline{X}(\mathbf x^\pm)\label{eq:evalmap}\end{split}
 \eq
is the Lie algebra of the stabilizer of $\mathbf x^\pm$, i.e.\ $\fp^\pm$. Therefore, the  decomposition \eqref{eq:decomp3905} provides us with linear isomorphisms
\bq
T_{\mathbf x^+}\F^+ \simeq\fn_{\fp^-},\qquad 
T_{\mathbf x^-}\F^- \simeq\fn_{\fp^+},\qquad 
T_{\mathbf{x}}\Fpf \simeq\fn_{\fp^-}\oplus \fn_{\fp^+}.\label{eq:TfFpfident}
\eq
\begin{rem}\label{rem:signswap}Note that these isomorphisms swap the $+$ and $-$ signs in the upper index, which will be important to keep track of when we apply them.
\end{rem}
Moreover, consider the parabolic subgroups $P_{\mathbf x^\pm}:=\mathrm{Stab}_G(\mathbf x^\pm)=\mathrm{Stab}_G(\nL_{\fp^\pm})$.  Then for every $g\in P_{\mathbf x^+}$ the fact that $g$ fixes $\nL_{\fp^+}$ allows us to identify both $T_{\mathbf x^-}\F^-$ and $T_{g\cdot\mathbf x^-}\F^-$ with $\nL_{\fp^+}$ via the evaluation map \eqref{eq:evalmap}, and analogously if $g\in P_{\mathbf x^-}$. In both cases, the derivative $dg:T_{\mathbf x^\mp}\F^\mp\to T_{g\cdot\mathbf x^\mp}\F^\mp$ of the $g$-action then becomes a linear map $\nL_{\fp^\pm}\to \nL_{\fp^\pm}$ given by the restricted adjoint action
 \bq\begin{split}
\fn_{\fp^\pm}&\to \fn_{\fp^\pm}\\ \label{eq:tangentTFpm}
X&\mapsto \Ad(g)X.
\end{split}
\eq

This allows us to give a geometric interpretation of the Jacobian characters of Levi subgroups:

\begin{lem} \label{lem:LeviJacobianClassFunctionIsDeterminant}
Let $\F^+,\F^-$ be a pair of opposite flag manifolds. Consider $\mathbf x=(\mathbf x^+,\mathbf x^-)\in\F^+\trtimes\F^-$, let $L=\mathrm{Stab}_G(\mathbf x)$ denote its stabilizer and $\fn$ the nilpotent radical of the Lie algebra of the stabilizer of $\mathbf x^+\in \F^+$. Then for any $g\in L$, we have
\[e^{\mathbf J_L(g)}= \vert \det( d_{\mathbf x^-}g) \vert = |\det( \Ad(g)\vert_{\fn})|,\] 
where $d_{\mathbf x^-}g\in \GL\left(T_{\mathbf x^-}\F^-\right)$ denotes the differential of the action of $g$ at its fixed point $\mathbf x^-\in\F^-$.
\end{lem}
\begin{proof}
The second equality comes from the discussion above. Considering an element $Y\in\g_{\mathrm{hyp}}$ such that $\mathrm{Stab}_G(\mathbf x^+)=P_Y$ and $\mathrm{Stab}_G(\mathbf x^-)=P_{-Y}$, and a restricted root space decomposition such that $Y\in\aL^+$,   we may write $L=L_\Theta$ for some subset $\Theta\subset\Delta$. Now consider the Jordan decomposition $g=g_eg_hg_u$ in the reductive Lie group $L_\Theta$. Since the elements $g_e,g_h,g_u\in L_\Theta< P_\Theta$ fix $\nL$, we immediately have
\[ |\det( \Ad(g_e)\vert_{\fn})|= \det( \Ad(g_u)\vert_{\fn}) = 1~,\]
thus
\bq
 |\det( \Ad(g)\vert_{\fn})|=|\det( \Ad(g_h)\vert_{\fn})|.\label{eq:gghreduction}
\eq
Since we also have $\mathbf J_L(g)=\mathbf J_L(g_h)$ by definition of the Jordan projection $\lambda_{L_\Theta}$, we may now assume that $g$ is hyperbolic, i.e. $g=g_h$. Furthermore conjugating in $L_\Theta$, we may assume that $g=\exp(X)$ for some $X\in \aL^+_{L_\Theta}$, where $\aL^+_{L_\Theta}$ was defined in \eqref{eq:aLTheta}.

 It now comes from the very definition of $\varrho_\Theta=\frac{1}{2}\sum_{\alpha\in \Sigma^+_\Theta}\dim \fg_\alpha \alpha$  that the determinant of the action of $\Ad(\exp(X))=\exp(\ad_X)$ on $\fn=\bigoplus_{\alpha\in \Sigma_\Theta^+}\fg_\alpha$ is 
\[ \det\left( \Ad(\exp(X))\vert_{\fn}\right)=e^{2\varrho_\Theta\circ p_\Theta(X)}~, \]
and $2\varrho_{\Theta}\circ p_\Theta(X)$ is the definition of $\mathbf J_{L_\Theta}(\exp(X))$.
\end{proof}

In order to give a similar geometric interpretation of the class functions $\mathcal P_\alpha$ and $\mathcal J_\F$ defined on the whole group $G$, we must first guarantee the existence of fixed points in flag manifolds, then select a preferred fixed point as there can be several. This is done through the notion of proximality.

\begin{definition}\label{def:proximal}
Let $\F$  be a flag manifold, $\mathbf x\in \F$, and $g\in G$. We say that $\mathbf x$ is an \emph{attracting fixed point} of $g$ if  $g\in\mathrm{Stab}_G(\mathbf x)$ and the linear action of $g$ on $T_{\mathbf x}\F$ given by the derivative of the $G$-action has spectral radius $<1$.

An element $g\in G$ is called \emph{$\F$-proximal} if $g$ has an attracting fixed point in $\F$.
\end{definition}
Numerous useful properties of $\F$-proximal elements are proved in \cite[Lemma 2.26, Prop.~3.3~(c)]{GGKW}. For instance, it is shown there that every $\F$-proximal element $g\in G$ has a \emph{unique} attracting fixed point in $\F$, and that
\begin{align}\begin{split}
g \text{ is  $\F$-proximal} &\iff g^{-1} \text{ is  $\overline \F$-proximal}\\
&\iff \mathcal P_\alpha(g)>0~\forall \alpha\in \Theta\, . \label{eq:FproxFbarprox}\end{split}
\end{align}
\begin{rem} Consider a pair of opposite flag manifolds $\F^+,\F^-$, and an $\cF^+$-proximal element $g\in G$. Denoting by $g^+\in \cF^+$ its attracting fixed point  and by $g^-\in\cF^-$ its \emph{repelling fixed point}, i.e.  the attracting fixed point of $g^{-1}$, we have that $g^+\pitchfork g^-$.
\end{rem}

With this terminology, we can provide a geometric interpretation of the Jacobian class functions of flag manifolds:

\begin{lem} \label{lem: determinant proximal element on a  flag manifold}Consider a pair of opposite flag manifolds $\F^+,\F^-$ and the transverse flag space $\Fpf=\F^+\trtimes \F^-$. Let $g\in G$ be $\F^+$-proximal with attracting/repelling fixed points $\mathbf x=(\mathbf x^+,\mathbf x^-)\in \Fpf$, and denote by $\fn$ the nilpotent radical of the Lie algebra of the stabilizer of $\mathbf x^+\in \F^+$.  Then
\bq
e^{\mathcal J_{\F^+}(g)}=\big|\det\big( \left.\Ad(g)\right\vert_{\fn} \big)\big|. \label{eq:238950328}
 \eq
\end{lem}
\begin{proof}
Let $L=\mathrm{Stab}_G(\mathbf x)$, so that according to Lemma \ref{lem:LeviJacobianClassFunctionIsDeterminant} we only need to show that $\mathcal J_{\F^+}(g)=\mathbf J_{L}(g)$. Repeating the process used in the proof of Lemma \ref{lem:LeviJacobianClassFunctionIsDeterminant}, we can consider a fixed restricted root space decomposition such that $L=L_\Theta$ is a standard Levi subgroup and reduce the study to the case of $g=\exp(X)$ with $X\in \aL_{L_\Theta^+}$.

Assuming that $\mathbf x^-$ is the repelling fixed point of $g$ means that  all the eigenvalues of the action of $\exp(X)$ on the tangent space $T_{\mathbf x^-}\cF^-$  are greater than one (in absolute value). But the fact that $\exp(X)\in L_\Theta$ means that this derivative is conjugate to the action of $\Ad(\exp(X))=\exp(\ad_X)$ on $\fn$. The eigenvalues of the action of $\ad_X$ on $\fn=\bigoplus_{\alpha\in \Sigma_\Theta^+}\fg_\alpha$ are therefore positive, hence $\alpha(X)>0$ for any $X\in \Theta$. But we already had that $X\in \aL_{L_\Theta}^+$, i.e. that $\alpha(X)\geq 0$ for any $\alpha\in \Delta\setminus\Theta$, thus $X\in\aL^+$ and $\lambda(\exp(X))=X=\lambda_{L_\Theta}(\exp(X))$ by \eqref{eq: equality case between Jordan projections}. It follows that $\mathcal J_{\F^+}(g)=\mathbf J_{L}(g)$.
\end{proof}

\begin{rem} Note that from Lemma \ref{lem:LeviJacobianClassFunctionIsDeterminant} we also obtain  $e^{\mathcal J_{\F^+}(g)}=\big|\det\big(d_{\mathbf x^-}g\big)\big|$. This geometric interpretation of the Jacobian class function does not involve the attracting fixed point, but the proof relies on its existence and transversality to the repelling fixed point. 
\end{rem}

For later use, we record the following elementary naturality property of proximality.
\begin{lem}\label{lem:proxdomination}
Let $\F^+,\F^-$ be a pair of opposite flag manifolds corresponding to subsets $\Theta\subset\Delta$ and $\iota_o(\Theta)$ respectively. Denote the transverse flag spaces by $\Fpf=\F^+\trtimes \F^-$ and $\F_\alpha^\pitchfork=\F_\alpha\trtimes \F_{\iota_o(\alpha)}$ for $\alpha\in\Theta$. Then any $\F^+$-proximal element  $g\in G$  with attracting/repelling fixed points  $(\mathbf x^+,\mathbf x^-)\in \Fpf$ is also $\F_\alpha$-proximal with attracting/repelling fixed points $(\mathbf x^+_\alpha,\mathbf x^-_\alpha)\in \F^\pitchfork_\alpha$ for all $\alpha\in\Theta$, where $\mathbf x^+_\alpha=\Pi^{\F^+}_\alpha(\mathbf x^+)$ and $\mathbf x^-_\alpha=\Pi^{\F^-}_{\iota_o(\alpha)}(\mathbf x^-)$ are the images under the unique projections introduced in \eqref{eq:FFalphaprojection}.
\end{lem}
\begin{proof}
The fact that  $(\mathbf x^+_\alpha,\mathbf x^-_\alpha)\in \F^\pitchfork_\alpha$ comes from Corollary \ref{cor domination and opposition}, and the contraction/dilation of the derivatives at these fixed points follows from the same property on $\F^\pm$ by simple linear algebra.
\end{proof}

Let us conclude this section with the geometric interpretation of the proximality class functions.

\begin{lem} \label{lem:rootclassfunction is spectral radius of derivative}
Let $\cF_\alpha$ be a simple flag manifold and $\F^\pitchfork_\alpha=\F_\alpha\trtimes\F_{\iota_o(\alpha)}$ the corresponding transverse flag space. For any $\cF_\alpha$-proximal element $g\in G$, we have
\[ \mathcal P_\alpha(g)= -\lambda_1\big(d_{\mathbf x^-}g^{-1}\big), \]
where $d_{\mathbf x^-}g^{-1}\in \GL\left(T_{\mathbf x^-}\F_{\iota_o(\alpha)}\right)$ denotes the differential of the action of $g^{-1}$ at the repelling fixed point $\mathbf x^-\in\F_{\iota_o(\alpha)}$ of $g$ and $\lambda_1(d_{\mathbf x^-}g^{-1})$ is the  logarithm of its spectral radius.
\end{lem}
\begin{rem} This means that if we rank the logarithms of moduli of eigenvalues
\[ \lambda_1(d_{\mathbf x^-}g)\geq\cdots\geq \lambda_n(d_{\mathbf x^-}g)~,\]
then $\mathcal P_\alpha(g)= \lambda_n(d_{\mathbf x^-}g)$, i.e.\ $\mathcal P_\alpha(g)$ measures the dilation of $g$ near its repelling fixed point.
\end{rem}
\begin{proof}
Just as in the proof of Lemma \ref{lem: determinant proximal element on a  flag manifold}, $\cF_\alpha$-proximality means that we can work in the context of a fixed restricted root space decomposition and assume that $g=\exp(X)$ for some $X\in\aL^+$, in which case $\mathcal P_\alpha(g)=\alpha(X)$ and the derivative $d_{\mathbf x^-}g$ is conjugate to the action of $\Ad(\exp(X))=\exp(\ad_X)$ on the nilpotent radical $\fn=\bigoplus_{\beta\in\Sigma_{\{\alpha\}}^+}\g_\beta$, so we find
\[\lambda_n(d_{\mathbf x^-}g)=\min\set{\beta(X)}{\beta\in \Sigma_{\{\alpha\}}^+}\,.\]
Now every $\beta\in \Sigma_{\{\alpha\}}^+\subset\Sigma^+$ can be expressed as $\beta=\alpha+\beta'$ where $\beta'$ is a combination of elements of $\Delta$ with non-negative coefficients, thus $\lambda_n(d_{\mathbf x^-}g)=\alpha(X)$.
\end{proof}

\section{Affine line bundles and period functions}\label{sec:affpermultidens}

Given a transverse flag space $\Fpf=\F^+\trtimes\F^-$, our goal is to find homogeneous $G$-manifolds on which we can define a smooth flow in such a way that we can interpret the two transverse flags forming a point $\mathbf x=(\mathbf x^+,\mathbf x^-)\in \Fpf$ as the backward and forward time limits of a unique flow line, thus mimicking the geodesic flow of a rank one symmetric space. Such a flow over $\Fpf$ should have trivial dynamics, so that all the dynamical properties of a quotient flow by a discrete subgroup $\Gamma<G$ actually describe the dynamics in the subgroup $\Gamma$ itself. This leads us to the study of homogeneous affine line bundles.

\subsection{\texorpdfstring{$G$}{G}-equivariant affine line bundles}

We will apply the following general definition with $\Fbb$ either a flag manifold or a transverse flag space:

\begin{definition}\label{def:affinelinebundles}Let $\Fbb$  be a homogeneous $G$-manifold.  \begin{itemize}\item A \emph{$G$-equivariant affine line bundle} $\L\to \Fbb$ is a principal $\R$-bundle over $\Fbb$ such that:
\begin{itemize}
\item $\L$ is a $G$-manifold.
\item The projection $\L\to \Fbb$ is $G$-equivariant.
\item The $G$-action and the $\R$-action on $\L$ commute.
\end{itemize}
\item A \emph{morphism  $\L\to \L'$ of $G$-equivariant affine line bundles}  $\L,\L'$ over $\Fbb$   is a $G$-equivariant map $\L\to \L'$ that forms a morphism of principal $\R$-bundles over $\Fbb$. We write $\L\simeq \L'$ for an isomorphism of $G$-equivariant affine line bundles.

\item A $G$-equivariant affine line bundle $\L\to \Fbb$ is \emph{trivial} if $\L\simeq \Fbb\times \R$, where on the right $G$ acts by the product of the $G$-action on $\Fbb$ and the trivial action on $\R$. 

\item The \emph{time-reversal} $\overline\L\to \Fbb$ of a $G$-equivariant affine line bundle $\L\to \Fbb$ is obtained from $\L$ by flipping the sign in the $\R$-action, i.e., composing with the map $\R\owns t\mapsto -t\in \R$.

\item A $G$-equivariant affine line bundle $\L$ over $\Fbb$ is \emph{homogeneous} if the $G$-manifold $\L$ is homogeneous, i.e., the $G$-action on $\L$ is transitive. In this case,  $\L\to \Fbb$ is called a \emph{homogeneous affine line bundle} over $\Fbb$.
\end{itemize}
\end{definition}
Note that a trivial $G$-equivariant affine line bundle is never homogeneous. Conversely, any non-trivial $G$-equivariant affine line bundle $\L\to\Fbb$ is homogeneous.

\subsection{Affine line bundles and oriented vector line bundles}\label{sec:constraff} Principal $\R$-bundles correspond to oriented vector line bundles\footnote{We use the term \emph{vector line bundle} for what is usually just called \emph{line bundle}, i.e.\ a vector bundle of rank one, in order to draw a clear distinction from an \emph{affine line bundle}.}, since the group $\R$ is isomorphic to the identity component $\GL^+(1,\R)$. Note that an orientation on a vector line bundle $\mathcal L\to \Fbb$ is just a smooth choice of a connected component $\mathcal L_{>0}$ of the complement of the zero section in each fiber. Then $\mathcal L_{>0}\to\Fbb$ can be given the structure of a principal $\R$-bundle by defining the action of $t\in\R$ on $v\in \mathcal L_{>0}$ to be $v\cdot t =e^tv$.

 Starting with a principal $\R$-bundle $\L\to\Fbb$, the usual construction of an associated bundle $$\mathcal L=(\L\times\R) / \R\to\Fbb,$$ where the action of $t\in\R$ on $(\hat{\mathbf x},s)\in\L\times\R$ is $(\hat{\mathbf x},s)\cdot t=(\hat{\mathbf x}\cdot t,e^{-t}s)$, has a natural orientation defined as $$\mathcal L_{>0}:=(\L\times\R_{>0})/\R.$$  

In the case of $G$-equivariant bundles over a homogeneous $G$-manifold $\Fbb$, both directions of this correspondence are compatible with $G$-actions. Note that the time reversal $\overline\L$ corresponds to switching the orientation of the vector line bundle $\mathcal L$ from $\mathcal L_{>0}$ to $\mathcal L_{<0}$.

While affine line bundles are more fitted to our dynamical intentions, since the $\R$-action can be interpreted as a flow, vector line bundles are more practical when it comes to algebraic constructions such as tensor products. Most of our focus will be on  the affine point of view (including the current section), but the constructions in Section \ref{sec:densitiygeneral} will be achieved using oriented vector line bundles.

\subsection{Homogeneous affine line bundles and additive characters}\label{sec:homalg}

Consider the case of a $G$-homogeneous space $\Fbb$, and fix a base point 
$\mathbf x\in \Fbb$ so that we can identify $\Fbb=G/H$ where $H<G$ is the stabilizer of $\mathbf x$. The following basic correspondence is well-known.

\begin{lem} \label{lem correspondence homogeneous affine line bundles and additive characters}
There is a bijective correspondence between  Lie group homomorphisms $b:H\to \R$ (considered up to precomposition by an isomorphism) and (equivalence classes of) $G$-equivariant affine line bundles over $G/H$.\qed
\end{lem}

A classification of homogeneous affine line bundles over a given homogeneous space $G/H$ therefore reduces to a classification of Lie group homomorphisms $H\to \R$. In the case of a transverse flag space,  the isotropy group is a Levi subgroup $L<G$. Recall from \eqref{eq:direct product levi} that $L$ splits into a direct product $L= Z_\mathrm{split}(L)M_L$ where $M_L$ is a reductive group with compact center, so  the  map
\begin{equation}
\map{\fz_{\mathrm{split}}(\fl)^*}{\Hom(L,\R)}{\beta}{ \beta\circ \lambda^Z_L}  \label{eq:isomorphism additive characters dual split center}
\end{equation}
is an isomorphism, where $\lambda^Z_L:L\to \fz_{\mathrm{split}}(\fl)$ is the split central projection.

For a standard Levi subgroup $L_\Theta$, a natural basis of the dual space of $\fz_{\mathrm{split}}(\fl_\Theta)=\aL_\Theta$ is given by the fundamental weights $w_\alpha$, $\alpha\in\Theta$.

\begin{cor} \label{cor correspondence homogeneous affine line bundles and combinations of fundamental weights}
Consider a restricted root space decomposition, a subset $\Theta\subset\Delta$  and the corresponding standard Levi subgroup $L_\Theta$. There is a bijective correspondence between isomorphism classes of $G$-equivariant affine line bundles over the transverse flag space $\F_\Theta\trtimes\F_{\iota_o(\Theta)}\simeq G/L_\Theta$  and  linear combinations of the fundamental weights $w_\alpha$, $\alpha\in\Theta$.
\end{cor}

\subsection{The period function of a \texorpdfstring{$G$}{G}-equivariant affine line bundle} \label{sec:period function}We now give a second description of $G$-equivariant affine line bundles over a homogeneous space $\Fbb$, where in contrast to Section \ref{sec:homalg} we do not need to fix a base point. To this end, we consider a real additive character
\bqn
\ell_\L(\mathbf x,-)\in\Hom(\Stab_G(\xbf),\R)
\eqn
for each $\xbf\in\Fbb$, defined through the action
\[ g\cdot\hat \xbf=\hat \xbf\cdot \ell_\L(\xbf,g)\,,\quad \forall \hat\xbf\in \L_\xbf\, \forall g\in\Stab_G(\xbf)\,.\]
If we consider for some $h\in G$ the point $h\cdot\mathbf x\in \Fbb$, then we get $\mathrm{Stab}_G(h\cdot\mathbf x)=h\mathrm{Stab}_{G}(\mathbf x)h^{-1}$ and for $g\in \mathrm{Stab}_{G}(\mathbf x)$ we find
\bqn
\ell_\L(h\cdot\mathbf x,h g h^{-1})=\ell_\L(\mathbf x,g)
\eqn
because the $G$- and $\R$-actions on $\L$ commute. We equip the \emph{stabilizer bundle}
 $$
 \mathrm{Stab}_\Fbb:=\{(\mathbf x,g)\in \Fbb\times G\,|\, g\in \mathrm{Stab}_{G}(\mathbf x)\}
 $$
with the $G$-action $h\cdot (\mathbf x,g):=(h\cdot\mathbf x,hgh^{-1})$. The projection onto the first factor makes $\mathrm{Stab}_{\Fbb}$ a $G$-equivariant fiber bundle over $\Fbb$. \begin{definition}\label{def:periods}
The $G$-invariant function 
\[ 
\ell_\L:\map{\mathrm{Stab}_{\Fbb}}{\R}{(\mathbf x,g)}{\ell_\L(\mathbf x,g)} 
\]
is called the \emph{period function} of the $G$-equivariant affine line bundle $\L$.
\end{definition}
Period functions classify $G$-equivariant affine line bundles up to isomorphism:
\begin{lem}\label{lem:isolambda}For any two $G$-equivariant affine line bundles  $\L$, $\L'$  over a $G$-homogeneous space $\Fbb$, the following statements are equivalent:
\begin{enumerate}
\item $\L\simeq \L'$.
\item $\ell_{\L}=\ell_{\L'}$.
\item There is a point $\mathbf x\in \Fbb$ such that $\ell_{\L}(\mathbf x,-)=\ell_{\L'}(\mathbf x,-):\mathrm{Stab}_{G}(\mathbf x)\to \R$.
\end{enumerate}
\end{lem}
\begin{proof}
The equivalence between \emph{(1)} and \emph{(3)} is Lemma \ref{lem correspondence homogeneous affine line bundles and additive characters}, and it is straightforward to see that \emph{(1)}$\Rightarrow$\emph{(2)}$\Rightarrow$\emph{(3)}.
\end{proof}

A base-point free version of Lemma \ref{lem correspondence homogeneous affine line bundles and additive characters} therefore goes through the following definition:
\begin{definition}\label{def:space of period functions}
The space $\mathcal P(\Fbb)$ of \emph{period functions} of  $\Fbb$ is the space of  $G$-invariant functions $\Stab_\Fbb\to\R$ that restrict to a Lie homomorphism on each fiber.
\end{definition}
Note that $\mathcal P(\Fbb)$ is a finite dimensional real vector space.

\begin{cor}\label{cor:isocor} Let $\Fbb$ be a homogeneous $G$-manifold. The map $\L\mapsto \ell_\L$ establishes a bijective correspondence between isomorphism classes of $G$-equivariant affine line bundles over $\Fbb$ and $\mathcal P(\Fbb)$.\qed
\end{cor}
In particular, the vector space structure on the set of isomorphism classes of $G$-equi\-variant affine line bundles over $\Fbb$ given by identifying it with additive characters in Lemma \ref{lem correspondence homogeneous affine line bundles and additive characters} does not depend on the choice of a base point. 

Consider opposite flag manifolds $\F^+,\F^-$ and the transverse flag space $\Fpf=\F^+\trtimes\F^-$. The Jacobian characters of Levi subgroups can be combined into the period function
\[ \map{\mathrm{Stab}_\Fpf}{\R}{(\mathbf x,g)}{\mathbf J_{\mathrm{Stab}_G(\mathbf x)}(g).}\]
If we now consider the corresponding subsets $\Theta,\iota_o(\Theta)\subset\Delta$, then for each $\alpha\in \Theta$ there is a unique $G$-equivariant projection $\pi_\alpha:\Fpf\to \F^\pitchfork_\alpha=\F_\alpha\trtimes\F_{\iota_o(\alpha)}$ given by Corollary \ref{cor domination and opposition}. We can use it to define 
\[ \mathbf J_{\Fpf}^{\alpha}:\map{\mathrm{Stab}_\Fpf}{\R}{(\mathbf x,g)}{\mathbf J_{\Stab_G(\pi_\alpha(\xbf))}(g).}\]

\begin{prop} \label{prop basis of Jacobian class functions}
Consider opposite flag manifolds $\F^+,\F^-$ and their corresponding subsets  $\Theta,\iota_o(\Theta)\subset\Delta$, and let $\Fpf=\F^+\trtimes \F^-$. The functions $(\mathbf J_{\Fpf}^{\alpha})_{\alpha\in\Theta}$ form a basis of $\mathcal P(\Fpf)$.
\end{prop}

\begin{proof}
It follows from Lemma \ref{lem:wmalpha} ($\varrho_{\{\alpha\}}$ and $w_\alpha$ are proportional) and Corollary \ref{cor correspondence homogeneous affine line bundles and combinations of fundamental weights}.
\end{proof}

\subsection{Sections and cocycles} An principal $\R$-bundle over any manifold is trivializable. However, a homogeneous affine line bundle cannot be equivariantly trivializable because the trivial bundle is not homogeneous. The failure of equivariance of a given trivialization (i.e.\ a section) is encoded in a cocycle.

\begin{definition} \label{def cocycle from a section}
Let $\sigma$ be a section of a $G$-equivariant  affine line bundle $\L\to\Fbb$. The cocycle associated to $\sigma$ is the function $\mathcal C_\sigma:\Fbb\times G\to \R$ defined by the relation
\[ g\cdot \sigma(\mathbf x)=\sigma(g\cdot\mathbf x)\cdot \mathcal C_\sigma(\mathbf x,g)\quad \forall (\mathbf x,g)\in\Fbb\times G~.\]
\end{definition}

Note that if $g\cdot \mathbf x=\mathbf x$, the value of $\mathcal C_\sigma(\mathbf x,\nu)$ does not depend on $\sigma$ since $\mathcal C_\sigma(\mathbf x,\nu)=\ell_\L(\mathbf x,g)$ in this case. So one can see $\mathcal C_\sigma$ as an extension of the period function $\ell_\L$ to $\Fbb\times G$ that  satisfies the cocycle relation
\[ \mathcal C_\sigma(\mathbf x, g_2g_1)=\mathcal C_\sigma(g_1\cdot\mathbf x,g_2)+\mathcal C_\sigma(\mathbf x,g_1)\quad \forall (\mathbf x, g_1,g_2)\in \Fbb\times G\times G~. \]

Although the cocycle depends on the choice of a section, its asymptotic behavior when $g\to\infty$ in $G$ will not depend on this choice provided that $\Fbb$ is compact (which is the case for flag manifolds, but not for transverse flag spaces).

\begin{lem} \label{lem bound difference cocycle over compact space}
Consider two sections $\sigma,\sigma'$ of a $G$-equivariant affine line bundle $\L\to\Fbb$. If the space $\Fbb$ is compact, then the difference between the two cocycles
\[ \mathcal C_\sigma-\mathcal{C}_{\sigma'}:\Fbb\times G\to \R\] is uniformly bounded.
\end{lem}

\begin{proof}
Consider the function $\tau:\Fbb\to\R$ such that $\sigma'(\mathbf x)=\sigma(\mathbf x)\cdot \tau(\mathbf x)$ for all $\mathbf x\in\Fbb$. The difference between the two cocycles is $\mathcal C_\sigma(\mathbf x, g)-\mathcal{C}_{\sigma'}(\mathbf x, g)=\tau(\mathbf x)-\tau(g\cdot \mathbf x)$.
\end{proof}

\section{Densities and multidensities}\label{sec:densitiypairings}

We will apply the construction from Section \ref{sec:constraff} to density and multidensity bundles, which are natural families of oriented vector line bundles on any smooth manifold. We begin by quickly recalling their general definitions and relevant properties.

\subsection{General definitions and properties}\label{sec:densitiygeneral}

Let $E$ be a finite dimensional real vector space of dimension $d>0.$  Let $\mathcal{B}_{E}$ be the set of ordered bases of $E$, which is the unique open $\mathrm{GL}(E)$-orbit in the product of $d$ copies of $E$ . 

\begin{definition}
For $s\in \R,$ an \emph{$s$-density} on $E$ is a map $\nu: \mathcal{B}_{E}\rightarrow \R$ such that for every $g\in \mathrm{GL}(E):$
$$
\nu(g\cdot e_{1},\ldots ,g\cdot e_{d})=\lvert \mathrm{det}(g)\rvert^{s} \nu(e_{1}, \ldots ,e_{d})
$$
for all $\mathbf{e}=(e_{1},\ldots,e_{d})\in \mathcal{B}_{E}.$  

We write $\mathscr D^s(E)$ for the real vector space of $s$-densities on $E$, and $\mathscr D^s_{>0}(E)$ for the subset of \emph{positive} $\R_{>0}$-valued $s$-densities. 
\end{definition}

 Note that $\mathscr D^s(E)$ is a $1$-dimensional space (because $\mathcal B_E$ is a $\GL(E)$-torsor), oriented by $\mathscr D^s_{>0}(E)$.

\begin{ex} 
 Any volume form $\omega\in \Lambda^{d}E^{*}$ produces an $s$-density defined by $\lvert \omega \rvert^{s}$, which is a positively oriented basis of $\mathscr D^s(E)$.
\end{ex}
\begin{ex}\label{example pseudo-riemannian density} A non-degenerate symmetric bilinear form $\varphi:E\times E\to \R$ produces a positive $s$-density $\mathrm{vol}_\varphi^s\in\mathscr D^s_{>0}(E)$ defined by $\mathrm{vol}_\varphi^s(\mathbf e)=\vert \det [\varphi]_{\mathbf e}\vert^{s/2}$, where the $(d\times d)$-Gram matrix $[\varphi]_{\mathbf e}$ has coefficients $\varphi(e_i,e_j)$ for $\mathbf e=(e_1,\dots,e_d)$.
 
\end{ex}

Next, consider a grading $E_{\bullet}$ of $E$, that is a  collection $E_\bullet=(E_i)_{i\in I}\in \prod_{i\in I} \mathcal G_{d_i}(E)$ of vector subspaces of dimensions $d_i>0$ such that
$$
E=\bigoplus_{i\in I} E_{i}~.
$$
  Denote by $\mathrm{GL}(E_{\bullet})<\mathrm{GL}(E)$ the subgroup of grading-preserving transformations (i.e. preserving each subspace $E_i$). We now adapt the notion of density to account for the grading.
\begin{definition}\label{def:multidensities}
For a tuple $\mathbf{s}=(s_i)_{i\in I}\in \R^I$,  which we will occasionally call a \emph{weight vector}, an \emph{$\mathbf{s}$-multidensity} on $E$ with respect to the grading $E_{\bullet}$ is a map
$$
\nu: \prod_{i\in I}\mathcal{B}_{E_i}\rightarrow \R
$$
such that for every $g\in \mathrm{GL}(E_{\bullet}):$
$$
\nu\big(g\cdot (\mathbf{e}_{i})_{i\in I}\big)=\prod_{i\in I} \lvert\mathrm{det}(g_{i})\rvert^{s_{i}}\nu\big((\mathbf e_{i})_{i\in I}\big)
$$
for every ordered basis $\mathbf{e}_{i}\in \mathcal{B}_{E_{i}}$ and every $i\in I.$  Above, the linear map $g_{i}\in \mathrm{GL}(E_{i})$ is the restriction of $g\in \mathrm{GL}(E_{\bullet})$ to $E_{i}$, and $g\cdot (\mathbf{e}_{i})_{i\in I}=(g_i\cdot \mathbf{e}_{i})_{i\in I}$. We write $\mathscr{D}^{\mathbf{s}}(E_{\bullet})$ for the  real vector space of $\mathbf{s}$-multidensities on $E_{\bullet}$, and $\mathscr{D}^{\mathbf{s}}_{>0}(E_{\bullet})$ for the subset of positive $\mathbf s$-multidensities.
\end{definition}
Again appealing to the free and transitive action of $\mathrm{GL}(E_\bullet)$ on $\prod_{i\in I}\mathcal{B}_{E_i}$, the space $\mathscr{D}^{\mathbf{s}}(E_{\bullet})$  is $1$-dimensional and oriented. 

\begin{ex} \label{example product of densities gives a multidensity}
 A collection $(\nu_i)_{i\in I}$ of $s_i$-densities on each $E_i$ yields an $\mathbf s$-multidensity $\prod_{i\in I}\nu_i\in\mathscr D^{\mathbf s}(E_\bullet)$ defined by 
\begin{equation}
\left(\prod_{i\in I}\nu_i\right)\Big((\mathbf e_{i})_{i\in I}\Big) := \prod_{i\in I}\nu_i(\mathbf e_i)~. \label{eq: product of density gives a multi-density}
\end{equation}
defining an isomorphism $\bigotimes_{i\in I} \mathscr{D}^{s_{i}}(E_{i})\rightarrow \mathscr{D}^{\mathbf{s}}(E_{\bullet})$ preserving positivity.  
\end{ex}
\begin{ex} \label{example pseudo-Riemannian and volume multidensities}
 This product operation can be combined with any concrete construction of $s_i$-densities on each $E_i$, such as starting with a collection of volume forms $\omega_{i}\in \Lambda^{d_{i}}E_{i}$ for each $i\in I$ or a symmetric bilinear form $\varphi:E\times E\to\R$ whose restriction to each subspace $E_i$ is non-degenerate.
\end{ex}

Given a vector bundle $E\rightarrow M$ over a smooth manifold $M$, by a grading of $E\to M$ we mean a collection $E_\bullet=(E_i)_{i\in I}$ of vector subbundles of $E$ such that $E=\bigoplus_{i\in I}E_i$.  Given such a grading, we call the pair $(E,E_\bullet)$ a graded vector bundle (often written just $E_\bullet$). The fiberwise application sending a graded vector space to the space of $\mathbf{s}$-multidensities defines the real vector line bundle $\mathscr{D}^{\mathbf{s}}(E_{\bullet})\rightarrow M$ whose sections are multidensities in the graded fibers of $E$.   

Next, consider  $\mathbb{F}$ a homogeneous $G$-manifold, and $(E,E_{\bullet})$  a $G$-equivariant graded vector bundle over $\Fbb$ (i.e., $E$ and each of its subbundles $E_i$ in the grading are $G$-equivariant). Then $\mathscr{D}^{\mathbf{s}}(E_{\bullet})$ is a $G$-equivariant real vector line bundle over $\mathbb{F}.$  A section $\nu$ of $\mathscr{D}^{\mathbf{s}}(E_{\bullet})$ is called an \emph{$\mathbf{s}$-multidensity on $E_\bullet$ over $\mathbb{F}$.}  If $E=T\mathbb{F}$ and $\mathbf{s}=s$ is just a scalar parameter (identified with a $1$-tuple), it is convention to refer to sections of $\mathscr{D}^{s}(T\mathbb{F})$ as simply the \emph{$s$-densities on $\mathbb{F}$.}

\subsection{(Multi-)density bundles on  flag manifolds}\label{sec:multitrans} 
 
We now apply the constructions from Sections  \ref{sec:densitiygeneral} and \ref{sec:constraff} to concrete (graded) vector bundles over flag manifolds. 

\subsubsection{Density bundles}

Let  $\Fbb$ be a homogeneous $G$-manifold and $s\in \R$. We write 
\bq
\mathscr{D}^s \Fbb:=\mathscr{D}^s T\Fbb\label{eq:defDs}
\eq
for the oriented line bundle of $s$-densities on $\Fbb$. The total space of this bundle carries a $G$-action obtained from the $G$-action on $T\Fbb$ through multilinear extension, making $\mathscr{D}^s \Fbb$ a $G$-equivariant oriented vector line bundle. Section \ref{sec:constraff} provides the $G$-equivariant affine line bundle 
$$
\mathscr{D}^s_{>0}{\Fbb}:=\mathscr{D}^s_{>0} T\Fbb
$$
of positive $s$-densities on $\Fbb$.  When $\Fbb$ carries a $G$-invariant pseudo-Riemannian metric (such is the case of a transverse flag space $\Fbb=\F^\pitchfork$, as  seen in Section \ref{sec:pseudoriemannianstructure}), the $G$-equivariant affine line bundle $\mathscr{D}^s_{>0}{\Fbb}$ has a $G$-equivariant section, so it is $G$-equivariantly trivializable. 

When $\Fbb=\F$ is a non-trivial flag manifold and $s\neq 0$, then $\mathscr{D}^s_{>0}{\F}$ is a homogeneous affine line bundle. Given a pair of opposite flag manifolds $\F^+,\F^-$, the period functions of the density bundles $\mathscr{D}^s_{>0}\F^\pm$  are related to Jacobian characters of Levi subgroups:

\begin{cor}\label{cor:periodfunctionJordandecomp} 
Consider a pair of opposite flag manifolds $\F^+,\F^-$, suppose that $g\in G$ has a pair of transverse fixed points $(\mathbf x^+,\mathbf x^-)\in \F^+\trtimes\F^-$ and write $L=\mathrm{Stab}_G(\mathbf x)$ for the corresponding Levi subgroup.  Then one has
\begin{align*}
\ell_{\mathscr{D}^s_{>0}\F^+}(\mathbf x^+,g)&=s \, \mathbf J_L(g^{-1}),\\
\ell_{\mathscr{D}^s_{>0}\F^-}(\mathbf x^-,g)&=s \, \mathbf J_{L}(g).
\end{align*}
\end{cor}
\begin{proof}Apply Lemma \ref{lem:LeviJacobianClassFunctionIsDeterminant}. 
\end{proof}
Here the presence of $g^{-1}$ in the first equation reflects the sign change pointed out in Remark \ref{rem:signswap}. From Lemma \ref{lem: determinant proximal element on a  flag manifold}, we also get a way of computing the cocycle associated to a section of $\mathscr D^s_{>0}\F$ (Definition \ref{def cocycle from a section}).

\begin{cor}\label{cor:cocycle density bundle} Consider a pair of opposite flag manifolds $\F^+,\F^-$, let $\sigma$ be a section of $\mathscr D^s_{>0}\F^-$ for some $s\in\R$,   $g\in G$  an $\F^+$-proximal element with attracting fixed point ${\mathbf x}^+\in\F^+$ and fix a basis $(X_1,\dots,X_n)$ of the nilpotent radical $\fn$ of the Lie algebra of its stabilizer. Then for any $\mathbf x\in\F^-$ transverse to ${\mathbf x}^+$, one has 
\[\mathcal C_\sigma(\mathbf x,g)+s\mathcal J_{\F^+}(g)=\log \frac{\sigma_{\mathbf x}\left(\overline{X_1}(\mathbf x),\dots,\overline{X_n}(\mathbf x)\right)}{\sigma_{g\cdot \mathbf x}\left(\overline{X_1}(g\cdot \mathbf x),\dots,\overline{X_n}(g\cdot \mathbf x)\right)}\,.\]
\end{cor}
\begin{proof}
The formula $d_\mathbf{x}g \cdot\overline X(\mathbf x)= \overline{\Ad(g)X}(g\cdot\mathbf x)$ and the basis $(d_{\mathbf x}g\cdot\overline{X_1}(\mathbf x),\dots,d_{\mathbf x}g\cdot\overline{X_n}(\mathbf x))$  of $T_{g\cdot\mathbf x}\F^-$  allow us  to express
\[ \mathcal C_\sigma(g,\mathbf x)= \log \frac{\sigma_{\mathbf x}\left(\overline{X_1}(\mathbf x),\dots,\overline{X_n}(\mathbf x)\right)}{\sigma_{g\cdot \mathbf x}\left(d_{\mathbf x}g\cdot\overline{X_1}(\mathbf x),\dots,d_{\mathbf x}g\cdot\overline{X_n}(\mathbf x)\right)}\,,\]
while Lemma \ref{lem: determinant proximal element on a  flag manifold} and the definition of a density yield
\begin{align*} s\mathcal J_{\F^+}(g) &=s\log \left\vert \det\big( \Ad(g)\vert_{\fn}\big)\right\vert\\
&=\log \frac{\sigma_{g\cdot \mathbf x}\left(\overline{\Ad(g)X_1}(g\cdot\mathbf x),\dots,\overline{\Ad(g)X_n}(g\cdot\mathbf x)\right)}{\sigma_{g\cdot \mathbf x}\left(\overline{X_1}(g\cdot \mathbf x),\dots,\overline{X_n}(g\cdot \mathbf x)\right)}\\
 &= \log \frac{\sigma_{g\cdot \mathbf x}\left(d_{\mathbf x}g\cdot\overline{X_1}(\mathbf x),\dots,d_{\mathbf x}g\cdot\overline{X_n}(\mathbf x)\right)}{\sigma_{g\cdot \mathbf x}\left(\overline{X_1}(g\cdot \mathbf x),\dots,\overline{X_n}(g\cdot \mathbf x)\right)}  \,.
\end{align*}
\end{proof}

\subsubsection{Multidensity bundles}

The graded vector bundles required to introduce multidensity bundles on  flag manifolds will be defined using the hierarchy introduced in Section \ref{sec:simpleflagmanifolds}. Consider a flag manifold $\F$ and the associated subset $\Theta\subset\Delta$.

Recall from \eqref{eq:FFalphaprojection} and Lemma \ref{lem uniqueness of equivariant maps between dominated flag manifolds}  that for each $\alpha\in\Theta$ there is a unique $G$-equivariant projection $\Pi_\alpha^\F:\F\to\F_\alpha$. We can use these projections to define the vector bundle
\[ E^\F:=\bigoplus_{\alpha\in\Theta}\big(\Pi^\F_\alpha\big)^*\big(T\F_\alpha\big) \]
over $\F$ with the obvious grading $E^\F_\bullet=\big(E^\F_{\alpha}\big)_{\alpha\in\Theta}$ defined over $\mathbf x\in \F$ by
\[ \big(E^\F_{\alpha}\big)_{\mathbf x}:= \set{ (v_{\alpha'})_{\alpha'\in\Theta}\in \bigoplus_{\alpha'\in\Theta}T_{\Pi^\F_{\alpha'}(\mathbf x)}\F_{\alpha'}}{ v_{\alpha'}=0~\forall \alpha'\neq \alpha} \simeq T_{\Pi^\F_\alpha(\mathbf x)}\F_\alpha.\]
The following diagram depicts the maps involved in the construction of the grading $E^\F_\bullet$. 
\begin{center}
\begin{tikzcd}
T\F_{\alpha_1} \arrow[d]  &  \cdots &  T\F_{\alpha_k} \arrow[d] \\
\F_{\alpha_1}  & \cdots &  \F_{\alpha_k}  \\
 & \arrow[lu, "\Pi_{\alpha_1}^\F"]\F\arrow[ru, "\Pi^\F_{\alpha_k}"']  &  
\end{tikzcd}
\end{center}

\begin{definition} \label{def multi density bundle over flag manifold}
Let $\F$ be a flag manifold corresponding to a subset $\Theta\subset\Delta$, and $\mathbf{s}=(s_\alpha)_{\alpha\in \Theta}\in \R^{\Theta}$. As in Section \ref{sec:densitiygeneral}, we define the $\mathbf{s}$-multidensity bundle over $\F$ to be
\[\mathscr D^{\mathbf s}\F_\bullet:=\mathscr D^{\mathbf s}E^\F_\bullet\]
and  corresponding positive $\mathbf{s}$-multidensity bundle 
\[\mathscr D^{\mathbf s}_{>0}\F_\bullet:=\mathscr D^{\mathbf s}_{>0}E^\F_\bullet~.\]
\end{definition}

The positive $\mathbf{s}$-multidensity bundle $\mathscr D^{\mathbf s}_{>0}\F_\bullet$ is a $G$-equivariant affine line bundle over $\F$.

\begin{lem}\label{lem:lambdasbf} Let $\F^+$ and $\F^-$ be opposite flag manifolds corresponding respectively to $\Theta\subset\Delta$ and $\iota_o(\Theta)$, and denote by $\pi_\alpha:\F^+\trtimes\F^-\to\F_\alpha\trtimes\F_{\iota_o(\alpha)}$ the $G$-equivariant projections. Let $g\in G$ possess  a pair of transverse fixed points $\mathbf x=(\mathbf x^+,\mathbf x^-)\in \F^+\trtimes\F^-$ and write $L_\alpha:=\mathrm{Stab}_G(\pi_\alpha(\xbf))$. Then, for all $\mathbf s\in \R^{\Theta}$, we have
\begin{align*}
\ell_{\mathscr{D}^{\mathbf{s}}_{>0}{\F^+_\bullet}}(\mathbf x^+,g)&=\sum_{\alpha \in \Theta}s_\alpha\mathbf J_{L_\alpha}(g^{-1}),\\
\ell_{\mathscr{D}^{\mathbf{s}}_{>0}{\F^-_\bullet}}(\mathbf x^-,g)&=\sum_{\alpha \in \Theta}s_\alpha\mathbf J_{L_\alpha}(g).
\end{align*}
\end{lem}
\begin{proof}
Write $\mathbf x^+_\alpha=\Pi^{\F^+}_\alpha(\mathbf x^+)$ and $\mathbf x^-_\alpha=\Pi^{\F^-}_\alpha(\mathbf x^-)$, so that we have
 \begin{align*}
\ell_{\mathscr{D}^{\mathbf{s}}_{>0}{\F^+_\bullet}}(\mathbf x^+,g)&=\sum_{\alpha \in \Theta}\ell_{\mathscr{D}^{s_\alpha}_{>0}{\F^+_\alpha}}(\mathbf x^+_\alpha,g),\\
\ell_{\mathscr{D}^{\mathbf{s}}_{>0}{\F^-_\bullet}}(\mathbf x^-,g)&=\sum_{\alpha \in \Theta}\ell_{\mathscr{D}^{s_\alpha}_{>0}{\F^+_\alpha}}(\mathbf x^+_\alpha,g).
\end{align*}
 By Lemma \ref{lem:proxdomination}, each pair $(\mathbf x^+_\alpha,\mathbf x^-_\alpha )$ is transverse, so the claim follows from Corollary \ref{cor:periodfunctionJordandecomp}. Note that the change of signs present in Corollary \ref{cor:periodfunctionJordandecomp}, which results from Remark \ref{rem:signswap}, is reflected here in the appearance of $g^{-1}$ in the $+$ case and $g$ in the $-$ case.
\end{proof}

\begin{rem}It is important to note that, while the ordinary density bundles were defined in \eqref{eq:defDs} using the tangent bundle as the underlying vector bundle, the multidensity bundle $\mathscr D^{\mathbf s}\F_\bullet$ is  not constructed from the tangent bundle of $\F$ (unless $\F$ is a simple flag manifold). 
\end{rem}

\section{Flow spaces}\label{sec:flowspaces}

With the preparations from Sections \ref{sec:affpermultidens} and \ref{sec:densitiypairings} at hand, it is now time to introduce the desired flow spaces. The idea is to create a multidensity analog of the hyperbola branch $\{(x^+,x^-)\in \R_{>0}^2 \,|\, x^+\cdot  x^-=1\}$. To this end, we need to pullback both  $\mathbf{s}$-multidensity bundles $\mathscr{D}^\mathbf{s}\F^+_\bullet$ and $\mathscr{D}^\mathbf{s}\F^-_\bullet$ to the transverse flag space $\F^+\trtimes\F^-$ where a pairing can be defined thanks to the pseudo-Riemannian structure.

\subsection{Pairings of multidensities on a transverse flag space}\label{sec pairing multidensities}
Consider a pair of opposite flag manifolds $\F^+,\F^-$ with  corresponding subsets $\Theta\subset\Delta$ and $\iota_o(\Theta)$ respectively, and a weight vector $\mathbf s\in\R^\Theta$. By using the projections
\[
\begin{tikzcd}
& \F^+\trtimes \F^- \arrow[dl, swap, "p_{\F^+}"] \arrow[dr, "p_{\F^-}"]
\\
\F^+ & & \F^-
\end{tikzcd}
\]
we can consider the real rank-two vector bundle
\begin{equation} \mathscr D^{\mathbf s}\cF^+_\bullet \boxplus \mathscr D^{\mathbf s}\cF^-_\bullet := p_{\F^+}^*\big(\mathscr D^{\mathbf s}\cF^+_\bullet\big) \oplus p_{\F^-}^*\big(\mathscr D^{\mathbf s}\cF^-_\bullet\big) \label{eq boxplus bundle}
\end{equation}
over $\F^+\trtimes \F^- $. Recall from the definitions of multidensities and the bundles $\mathscr D^{\mathbf s}\F_\bullet$ (Sections \ref{sec:densitiygeneral} and \ref{sec:multitrans} respectively) that an element $(\mu^+,\mu^-)$ in the fiber over $\mathbf x=(\mathbf x^+,\mathbf x^-)\in \F^+\trtimes\F^-$ consists of maps
\[ \mu^+: \prod_{\alpha\in\Theta} \mathcal B_{T_{\mathbf x^+_\alpha}\F^\alpha} \to \R~,\quad \mu^-: \prod_{\alpha\in\Theta} \mathcal B_{T_{\mathbf x^-_\alpha}\F_{\iota_o(\alpha)}} \to \R~ \]
obeying a given transformation rule under grading preserving linear transformations, where $\mathcal B_E$ stands for the set of ordered bases of a real vector space $E$, and we use the notations $\mathbf x^+_\alpha=\Pi^{\F^+}_\alpha(\mathbf x^+)\in\F_\alpha$, $\mathbf x^-_\alpha=\Pi^{\F^-}_{\iota_o(\alpha)}(\mathbf x^-)\in\F_{\iota_o(\alpha)}$ for the projections to simple flag manifolds.

An ordered basis $\mathbf e_\alpha^+=(e_1^+,\dots,e_n^+)$ of the tangent space $T_{\mathbf x^+_\alpha}\F_\alpha$ together with an ordered basis $\mathbf e_\alpha^-=(e_1^-,\dots,e_n^-)$ of the tangent space $T_{\mathbf x^-_\alpha}\F_{\iota_o(\alpha)}$ can be concatenated in order to form an ordered basis $(\mathbf e^+_\alpha,\mathbf e^-_\alpha):=\big( (e_1^+,0),\dots,(e^+_n,0),(0,e^-_1),\dots,(0,e_n^-)\big)$ of $T_{\mathbf x^+_\alpha}\F_\alpha\oplus T_{\mathbf x^-_\alpha}\F_{\iota_o(\alpha)}= T_{(\mathbf x^+_\alpha,\mathbf x^-_\alpha)}\F_\alpha\trtimes \F_{\iota_o(\alpha)}$. So by using the pseudo-Riemannian metric on the transverse flag manifold $\F_\alpha\trtimes \F_{\iota_o(\alpha)}$ defined in Section \ref{sec:pseudoriemannianstructure} and the construction of an associated $s_\alpha$-density $\mathrm{vol}^{s_\alpha}_{\F_\alpha^\pitchfork}$ from Example \ref{example pseudo-riemannian density} we can define the number
\[ \mathrm{vol}^{s_\alpha}_{\F_\alpha^\pitchfork} (\mathbf e^+_\alpha,\mathbf e^-_\alpha)>0~.\]
We now see from the transformation rule in the definition of a multidensity that the ratio
\begin{equation} \mu^+\cdot\mu^- := \frac{\mu^+\big( (\mathbf e^+_\alpha)_{\alpha\in \Theta} \big) \mu^-\big( (\mathbf e^-_\alpha)_{\alpha\in \Theta} \big)}{\prod_{\alpha\in \Theta}\mathrm{vol}^{s_\alpha}_{\F_\alpha^\pitchfork} (\mathbf e^+_\alpha,\mathbf e^-_\alpha)} \label{eq pairing multidensities}
\end{equation}
does not depend on the choice of a family of ordered bases $(\mathbf e^+_\alpha)_{\alpha\in \Theta}\in \prod_{\alpha\in\Theta} \mathcal B_{T_{\mathbf x^+_\alpha}\F^\alpha}$ and $(\mathbf e^-_\alpha)_{\alpha\in \Theta}\in \prod_{\alpha\in\Theta} \mathcal B_{T_{\mathbf x^-_\alpha}\F_{\iota_o(\alpha)}}$, thus defining an analytic $G$-invariant function
\begin{align*}\begin{split}
\mathscr D^{\mathbf s}\cF^+_\bullet \boxplus \mathscr D^{\mathbf s}\cF^-_\bullet&\rightarrow \R \\
(\mu^+,\mu^-) &\mapsto \mu^+\cdot\mu^-\end{split}
\end{align*}
that satisfies the homogeneity relation
\begin{equation} \big( t^+\mu^+\big)\cdot\big(t^-\mu^-\big)=t^+t^-\mu^+\cdot\mu^-\qquad \forall\; (t^+,t^-)\in\R^2~,\label{eq scaling of pairing}
\end{equation}
and that is positive on the subset $\mathscr D^{\mathbf s}_{>0}\cF^+_\bullet \boxplus \mathscr D^{\mathbf s}_{>0}\cF^-_\bullet\subset \mathscr D^{\mathbf s}\cF^+_\bullet \boxplus \mathscr D^{\mathbf s}\cF^-_\bullet$.

\subsection{The multidensity flow}
\begin{definition}\label{def:LO}Consider a pair of opposite flag manifolds $\F^+,\F^-$ with  corresponding subsets $\Theta\subset\Delta$ and $\iota_o(\Theta)$ respectively, and a weight vector $\mathbf s\in\R^\Theta$.
  
The \emph{flow space} $\bbL^{\mathbf s}_{\Theta}$ of $\Fpf=\F^+\trtimes\F^-$ with parameter $\mathbf s$ is the $G$-equivariant affine line bundle over $\Fpf$ given by
\[ 
\L^{\mathbf s}_{\Theta}:=\set{\left(\mu^+,\mu^-\right)\in \mathscr D^{\mathbf s}_{>0}\cF^+_\bullet \boxplus \mathscr D^{\mathbf s}_{>0}\cF^-_\bullet }{\mu^+\cdot \mu^-= 1},
\]
on which we let $t\in \R$ act by the flow $\phi^t$ defined as 
\[  
\phi^t\left(\mu^+,\mu^-\right) := \left(e^{-t}\mu^+,e^{t}\mu^-\right).
\]
\end{definition}
Note that this flow is well-defined because of the homogeneity of the pairing \eqref{eq scaling of pairing}.

More important than the flow spaces themselves are their isomorphism classes, where, as before,  ``isomorphism'' means ``isomorphism of $G$-equivariant affine line bundles'',  (Definition \ref{def:affinelinebundles}). Projecting onto the first and second factors of the flow space $\L^{\mathbf s}_{\Theta}\subset \mathscr D^{\mathbf s}_{>0}\cF^+_\bullet \boxplus \mathscr D^{\mathbf s}_{>0}\cF^-_\bullet$ reveals that $\L^{\mathbf s}_{\Theta}$ is nothing but a particularly useful  (or ``symmetric'') representative of the isomorphism class 
\bq
[\L^{\mathbf s}_{\Theta}]=\big[p_{\F^-}^\ast\big(\mathscr D^{\mathbf s}_{>0}\cF^-_\bullet\big)\big]=\big[p_{\F^+}^\ast\big(\overline{\mathscr D^{\mathbf s}_{>0}\cF^+_\bullet}\big)\big].\label{eq:isosLs}
\eq
From \eqref{eq:isosLs} and Lemma \ref{lem:lambdasbf}  we learn that the period function of $\L^{\mathbf s}_\Theta$ is given by
\bq
 \ell_{\L^{\mathbf s}_\Theta}=\sum_{\alpha\in\Theta} s_\alpha\mathbf J_{\Fpf}^{\alpha}\,.\label{eq:lLs}
\eq

\begin{cor}\label{cor:classificationLs}Every $G$-equivariant affine line bundle $\L\to\Fpf$ is isomorphic to $\L^{\mathbf{s}}_{\Fpf}$ for a unique $\mathbf s\in\R^\Theta$.
\end{cor}
\begin{proof}It follows from \eqref{eq:lLs}, Corollary \ref{cor:isocor} and Proposition \ref{prop basis of Jacobian class functions}.
\end{proof}

\section{Anosov subgroups}\label{sec:Anosovsubgroups}

Let $\F^+,\F^-$ be a pair of opposite flag manifolds associated to subsets $\Theta,\iota_o(\Theta)\subset\Delta$ respectively, and  let $\Fpf=\F^+\trtimes \F^-$ be their transverse flag space. Recall that if $\Gamma$ is a hyperbolic group, then $\partial_\infty\Gamma$ denotes its Gromov boundary.

\begin{definition} \label{def:ThetaAnosov}
A hyperbolic discrete subgroup $\Gamma<G$ is called $\Theta$-\emph{Anosov} if
\begin{enumerate}[label=(A\arabic*)]
\item There are continuous $\Gamma$-equivariant maps $\xi^\pm:\partial_\infty\Gamma\to \cF^\pm$, called \emph{boundary maps}.\label{property:P2}

\item The boundary maps are \emph{transverse} in the sense that\label{property:P3}
\[ 
\forall (x,y)\in\partial_\infty\Gamma^{(2)}\quad  (\xi^+(x),\xi^-(y))\in \Fpf.
\] 
\item The boundary maps preserve the convergence properties of the $\Gamma$-action on $\partial_\infty \Gamma$: For any unbounded sequence $\gamma_k\in \Gamma$ with boundary limit points $\gamma_+=\lim_{k\to +\infty}\gamma_k\in\partial_\infty\Gamma$ and  $\gamma_-=\lim_{k\to +\infty}\gamma_k^{-1}\in\partial_\infty\Gamma$,  the $\Gamma$-actions on $\F^+$ and $\F^-$ obey the following dynamics as $k\to +\infty$:
\begin{enumerate}
\item $\gamma_{k}\cdot \mathbf x^+\rightarrow \xi^+(\gamma_{+})$ for all $\mathbf x^+\in\F^+$ with $\mathbf x^+\pitchfork \xi^-(\gamma_{-})$;
\item $\gamma_{k}^{-1}\cdot \mathbf x^+\rightarrow \xi^+(\gamma_{-})$ for all $\mathbf x^+\in\F^+$ with $\mathbf x^+\pitchfork \xi^-(\gamma_{+})$;
\item $\gamma_{k}\cdot \mathbf x^-\rightarrow \xi^-(\gamma_{+})$ for all $\mathbf x^-\in\F^-$ with $\xi^+(\gamma_{-})\pitchfork \mathbf x^-$;
\item $\gamma_{k}^{-1}\cdot \mathbf x^-\rightarrow \xi^-(\gamma_{-})$ for all $\mathbf x^-\in\F^-$ with $\xi^+(\gamma_{+})\pitchfork \mathbf x^-$;
\end{enumerate}
and all these convergences are uniform on compact subsets. \label{property:P4}
\end{enumerate}
\end{definition}

\begin{rem} This is not the original definition of $\Theta$-Anosov subgroups as given in \cite{labourie} or \cite{GW12}, but an equivalent characterisation proved in \cite[Theorem 1.1]{KLP17}, where it is called \emph{asymptotically embedded}. Limit maps satisfying \ref{property:P2}, \ref{property:P3} and \ref{property:P4} are unique.
\end{rem}

\begin{lem} \label{lem Anosov implies growth Jordan projections} Let $\Gamma<G$ be $\Theta$-Anosov, and consider a sequence $(\gamma_k)$ in $\Gamma$ possessing  boundary limits $\gamma_+=\lim_{k\to\infty}\gamma_k, \gamma_-=\lim_{k\to\infty}\gamma_k^{-1}\in\partial_\infty\Gamma$. If $\gamma_+\neq\gamma_-$, then
\[\lim_{k\to\infty} \mathcal P_\alpha(\gamma_k)=+\infty~, \quad \forall \alpha\in\Theta\,.\]
\end{lem}
\begin{proof}
By Lemma \ref{lem diverging sequence hyperbolic group}, we have that $|\gamma_k|_\infty\to+\infty$, and \cite[Theorem 1.7]{GGKW} guarantees that $\mathcal P_\alpha(\gamma)\to+\infty$ as $|\gamma|_\infty\to+\infty$.
\end{proof}

For a hyperbolic group $\Gamma$, we write 
\[
\partial_\infty\Gamma^{(2)}:=\{(z^+,z^-)\in \partial_\infty\Gamma\times \partial_\infty\Gamma\,|\,z^+\neq z^-\}.
\] 
Let $\Gamma< G$ be a $\Theta$-Anosov subgroup, and combine the boundary maps $\xi^+,\xi^-$ to the map
\bq
\xi^\pitchfork:\map{\partial_\infty\Gamma^{(2)}}{\Fpf}{(z^+,z^-)}{(\xi^+(z^+),\xi^-(z^-)).}\label{eq:sigmamapnew}
\eq
The images of $\xi^\pm$ and $\xi^\pitchfork$ are the \emph{limit sets} in $\F^\pm$ and $\Fpf$, respectively, denoted by
\[
\Lambda^\pm_\Gamma:=\xi^\pm(\partial_\infty\Gamma)\subset\F^\pm,\qquad\Lambda^\transverse_\Gamma:=\xi^\pitchfork(\partial_\infty\Gamma^{(2)})\subset\Fpf.  
\] 
For a given word-hyperbolic subgroup $\Gamma< G$ one has (see \cite[Lemma 3.18]{GW12})
\bq
\Gamma\text{ is }\Theta\text{-Anosov}\iff \Gamma\text{ is }\iota_o(\Theta)\text{-Anosov}. \label{eq:FpfAnosovbarFpfAnosov}
\eq

If $\Theta= \emptyset$, then only finite subgroups of $G$ are $\Theta$-Anosov and the claims are trivially true. Thus, in the following we assume that $\Theta\neq \emptyset$.

\subsection{Invariant subsets of the flow space}

As before, let $\F^+,\F^-$ be a pair of opposite flag manifolds associated to subsets $\Theta,\iota_o(\Theta)\subset\Delta$ respectively, with transverse flag space  $\Fpf=\F^+\trtimes \F^-$, and let $\Gamma<G$ be $\Theta$-Anosov. Consider a flow space $\L^{\mathbf s}_\Theta$ as defined in Section \ref{sec:flowspaces}. Then it turns out that $\Gamma$ possesses a non-empty flow-invariant domain of proper discontinuity in $\L^{\mathbf s}_\Theta$ for suitable parameters $\mathbf s\in\R^\Theta$. The heart of the construction is the following
\begin{definition} \label{def prop disc domain}
Consider the following subset of $\Fpf$:
\bqn
\Omega_{\Gamma}:=\set{(\mathbf x^+,\mathbf x^-)\in \Fpf}{\forall\, z\in \partial_\infty\Gamma: \mathbf x^+\transverse \xi^-(z) \textrm{ or } \xi^+(z)\transverse \mathbf x^-},
\eqn
as well as the subsets of $\bbL^{\mathbf s}_\Theta$
\begin{align*}
\tilde{\cM}_{\Gamma,\Theta}^{\mathbf s}&:= p^{-1}(\Omega_{\Gamma})=\set{\mu\in \bbL^{\mathbf s}_\Theta}{p(\mu)\in \Omega_{\Gamma}},\\
\tilde{\cK}_{\Gamma,\Theta}^{\mathbf s}&:=p^{-1}(\Lambda_{\Gamma}^\pitchfork)= \set{\mu\in \bbL^{\mathbf s}_\Theta}{p(\mu)\in\Lambda^\pitchfork_\Gamma},
\end{align*}
where $p:\L_\Theta^{\mathbf s} \to\Fpf$ is the bundle projection.
\end{definition}
Note that $\tilde{\M}_{\Gamma,\Theta}^{\mathbf s}$ (resp.\ $\tilde{\cK}_{\Gamma,\Theta}^{\mathbf s}$)  is the restriction of the bundle $\bbL^{\mathbf s}_\Theta$ to the subset $\Omega_{\Gamma}$ (resp.\ $\Lambda^\pitchfork_\Gamma$) of its base space $\Fpf$. In particular, since the flow $\phi^t$ preserves the fibers of $\bbL^{\mathbf s}_\Theta$, we see that $\tilde{\M}_{\Gamma,\Theta}^{\mathbf s}$ and $\tilde{\cK}_{\Gamma,\Theta}^{\mathbf s}$ are $\phi^t$-invariant subsets of $\bbL^{\mathbf s}_\Theta$. Further elementary properties can be shown exactly as in the projective Anosov case, see \cite[Lemmas~3.4,~3.7]{VolI}:
\begin{lem}
 The sets $\Omega_{\Gamma}\subset  \Fpf$ and $\tilde{\M}_{\Gamma,\Theta}^{\mathbf s}\subset \bbL^{\mathbf s}_\Theta$ are open and $\Gamma$-invariant.  The set $\tilde{\cK}_{\Gamma,\Theta}^{\mathbf s}\subset \bbL^{\mathbf s}_\Theta$ is closed, $\Gamma$-invariant, and contained in $\tilde{\M}_{\Gamma,\Theta}^{\mathbf s}$. 
\end{lem}
In particular, since $\Lambda^\pitchfork_\Gamma$, and consequently $\tilde{\cK}_{\Gamma,\Theta}^{\mathbf s}$, is non-empty unless $\Gamma$ is finite, in which case $\Omega_{\Gamma}$ is all of $\Fpf$, we conclude that the sets $\Omega_{\Gamma}$ and $\tilde{\M}_{\Gamma,\Theta}^{\mathbf s}$ are always non-empty.

\subsection{Admissible parameters}\label{sec:admissibility} The dynamics of the action of $\Gamma$ on $\bbL_\Theta^{\mathbf s}$ depends on the parameter $\mathbf s\in\Theta$, in order to understand which parameters give a sufficiently well behaved action we need to associate them to linear forms on the Benoist limit cone. Consider a restricted root decomposition of $\g$, and recall the Jordan projection $\lambda:G\to\aL^+$ of $G$ defined in Section \ref{sec:defjordanclassfcn}. 

\begin{definition}
The \emph{Benoist limit cone} $\mathcal L(\Gamma)\subset \aL^+$ is the smallest closed cone in $\aL^+$ containing all Jordan projections $\lambda(\gamma)$, where $\gamma\in \Gamma$.
\end{definition}

For $\mathbf s\in \R^{\Theta}$, consider the element
\bq
w^{\mathbf s}_\Theta:=\sum_{\alpha \in \Theta}s_\alpha m_\alpha w_\alpha \in \aL^\ast,\label{eq:wmathbfs}
\eq
where the $w_\alpha$ are the fundamental weights defined in \eqref{eq:def_fundamentalweights} and the positive integers $m_\alpha\in \N$ were defined in \eqref{eq:defmalpha}.
\begin{definition}\label{def:Benoistpositive}An element $\mathbf s\in \R^{\Theta}$ is \emph{$(\Gamma,\Theta)$-admissible} if $w^{\mathbf s}_\Theta(X)>0\;\forall\, X\in \mathcal L(\Gamma)\setminus \{0\}$. 
\end{definition}

Even though this definition involves a root space decomposition, we can also consider the associated class function
\begin{equation}
\mathcal J_{\Theta}^{\mathbf s}:=w^{\mathbf s}_\Theta\circ \lambda=\sum_{\alpha\in\Theta}s_\alpha\mathcal J_{\F_\alpha},\label{eq:JTheta}
\end{equation}
with $\mathcal J_{\F_\alpha}$  the Jacobian class function (Definition \ref{def:JacbobF}) of the simple flag manifold $\F_\alpha$.

\begin{rem} \label{rem:admissible parameters involution}
Let $\iota:\R^{\Theta}\to \R^{\Theta}$ be the involution defined by $\iota(\mathbf{s})=(s_{\iota_o(\alpha)})_{\alpha\in \Theta}$. Then it is clear from the definitions that for every $\mathbf s\in \R^{\Theta}$ we have: 
\bq
\mathbf s \text{ is } (\Gamma,\Theta)\text{-admissible} \iff \iota(\mathbf s) \text{ is } (\Gamma,\iota_o(\Theta))\text{-admissible}.\label{eq:regularbarregular}
\eq
\end{rem}

\begin{rem} \label{rem:admissible parameters dual limit cone}
Admissible parameters can be interpreted as linear forms in the interior of some cone. Given a restricted root space decomposition, we consider $\aL_\Theta=\bigcap_{\alpha\in\Delta\setminus\Theta}\ker\alpha$ and   the orthogonal projection $p_\Theta:\aL\to\aL_\Theta$ with respect to the Killing form, and define the $\Theta$-limit cone $\mathcal L_\Theta(\Gamma)=p_\Theta\big(\mathcal L(\Gamma)\big)\subset\aL_\Theta$ as well as its dual cone
\[ \mathcal L_\Theta(\Gamma)^\ast:=\set{\varphi\in\aL_\Theta^*}{\varphi(X)\geq 0~\forall X\in \mathcal L_\Theta(\Gamma)}\,.\]
Then $\mathbf s\in\R^\Theta$ is admissible if and only if the restriction $w_\Theta^{\mathbf s}\vert_{\aL_\Theta}$ lies in the interior of $\mathcal L_\Theta(\Gamma)^\ast$.
\end{rem}

\begin{rem} \label{rem:existence regular parameters}
Note that the set of $(\Gamma,\Theta)$-admissible parameters $\mathbf s\in\R^{\Theta}$ is always non-empty: it contains $\R_{>0}^{\Theta}$.  
\end{rem}

\subsection{Refraction flows and cocompactness}\label{sec:refractionflows} Recall from Definition \ref{def cocycle from a section} that fixing a real analytic section $\sigma: \F^{\pitchfork}\rightarrow \mathbb{L}_\Theta^{\mathbf s}$  induces a real analytic cocycle
\[
\mathcal{C}_{\sigma}: \F^{\pitchfork}\times G\rightarrow \R\,.
\]
It can be pre-composed with the transverse limit map $\xi^{\pitchfork}: \partial_{\infty}\Gamma^{(2)}\rightarrow \F^{\pitchfork}$ of the $\Theta$-Anosov subgroup $\Gamma<G$, resulting in the Hölder cocycle
\[ \mathcal C_{\Gamma,\sigma}:\map{\partial_{\infty}\Gamma^{(2)}\times\Gamma}{\R}{(\mathbf x,\gamma)}{\mathcal C_\sigma(\xi^\pitchfork(\mathbf x),\gamma)}\,,\]
which induces a $\Gamma$-action on $\partial_{\infty}\Gamma^{(2)}\times\R$ via the formula
\begin{equation} \gamma\cdot \big(\mathbf x,s\big):=\big(\gamma\cdot\mathbf x,s-\mathcal C_{\Gamma,\sigma}(\mathbf x,\gamma)\big)\,. \label{eq:refraction action}
\end{equation}

Sambarino proved \cite[Corollary 5.3.3]{SAM24} that if the parameter $\mathbf s\in\R^{\Theta}$ is $(\Gamma,\Theta)$-admissible, then this action  is properly discontinuous and cocompact.

\begin{definition}
The \emph{refraction flow space} $\mathfrak{R}_{\sigma}$ associated to the $(\Gamma,\Theta)$-admissible parameter $\mathbf s\in\R^{\Theta}$ and the section  $\sigma: \F^{\pitchfork}\rightarrow \L_\Theta^{\mathbf s}$ is the Hausdorff quotient of $\partial_{\infty}\Gamma^{(2)}\times \R$ by the proper $\Gamma$-action described above.  The flow
\[
\psi_\sigma^{t}: \mathfrak{R}_{\sigma}\rightarrow \mathfrak{R}_{\sigma}
\]
induced by the affine translation action of $\R$ is called the \emph{refraction flow}.  
\end{definition}
The name refraction flow is due to Sambarino \cite{SAM24}. Although our use of an arbitrary section of $\sigma: \F^{\pitchfork}\rightarrow \L_\Theta^{\mathbf s}$ to define cocycles and $G$-actions may look more general than the explicit Hopf-Busemann-Iwasawa cocycles used in \cite{SAM24}, the computation of the period function \eqref{eq:lLs} shows that \cite[Lemma 5.3.2]{SAM24}, and therefore \cite[Corollary 5.3.3]{SAM24} or \cite[Theorem 3.2]{SAM14}, still applies (Remark \ref{rem:admissible parameters dual limit cone} offers a translation between our notations and those used in  \cite{SAM24}).  The refraction flow does depends on the section $\sigma,$ but for any two sections $\sigma$ and $\sigma^{\prime},$ the resulting refraction flows $\mathfrak{R}_{\sigma}$ and $\mathfrak{R}_{\sigma^{\prime}}$ have the same periods, and are therefore homeomorphically conjugate to each other \cite[Rem. 3.1]{SAM14}.

\begin{prop} \label{prop:cocompactness}
If $\Gamma<G$ is $\Theta$-Anosov and $\mathbf s\in\R^{\Theta}$ is $(\Gamma,\Theta)$-admissible, then the action of $\Gamma$ on $\widetilde{\mathcal K}_{\Gamma,\Theta}^{\mathbf s}$ is properly discontinuous and cocompact.
\end{prop}

\begin{proof}
Consider any section $\sigma: \F^{\pitchfork}\rightarrow \mathbb{L}_\Theta^{\mathbf s}$. The map
\begin{equation} \widetilde\Phi_\sigma:\map{\partial_{\infty}\Gamma^{(2)}\times\R}{\L_\Theta^{\mathbf s}}{(\mathbf x,s)}{\phi^s\circ\sigma\circ\xi^\pitchfork(\mathbf x)} \label{eq:embedding of refraction flow}
\end{equation}
is a $\Gamma$-equivariant Hölder-continuous topological embedding with image $\widetilde{\mathcal K}_{\Gamma,\Theta}^{\mathbf s}$, so the result follows from \cite[Corollary 5.3.3]{SAM24}.
\end{proof}
Note that the map $\widetilde{\Phi}_{\sigma}$ of \eqref{eq:embedding of refraction flow} satisfies $\widetilde{\Phi}_{\sigma}(\mathbf{x}, s+t)=\phi^{t}\circ \widetilde{\Phi}_{\sigma}(\mathbf{x}, s)$, so it induces a bi-Hölder homeomorphism
\begin{equation}
\Phi_{\sigma}:\mathfrak{R}_{\sigma}\to \mathcal K_{\Gamma,\Theta}^{\mathbf s} \label{eq:isomorphism refraction flow}
\end{equation}
conjugating the refraction flow with the quotient flow of $\phi^t$ on  $\mathcal K_{\Gamma,\Theta}^{\mathbf s}:=\Gamma\backslash\widetilde{\mathcal K}_{\Gamma,\Theta}^{\mathbf s}$.

\subsection{Proof of proper discontinuity} \label{sec:proofofThm3}

With these preparations we can finally formulate the following theorem, which represents the main result of this section.

\begin{thm} \label{thm: prop disc general case}
Let  $\Gamma< G$ be $\Theta$-Anosov and consider a $(\Gamma,\Theta)$-admissible parameter $\mathbf s\in\R^{\Theta}$. Then  $\Gamma$ acts properly discontinuously on $\tilde{\M}_{\Gamma,\Theta}^{\mathbf s}$ and cocompactly on  $\tilde{\cK}_{\Gamma,\Theta}^{\mathbf s}$. If $\Gamma$ is torsion-free, then the $\Gamma$-action on $\tilde{\M}_{\Gamma,\Theta}^{\mathbf s}$ is free.
\end{thm}

\begin{rem}\label{rem:lemmatimeinv} Theorem \ref{thm: prop disc general case} is a statement about $\tilde{\M}_{\Gamma,\Theta}^{\mathbf s}$ and $\tilde{\cK}_{\Gamma,\Theta}^{\mathbf s}$ as $\Gamma$-spaces, it does not involve the flow.  In particular, \eqref{eq:isosLs} implies that 
\bq
\tilde{\M}_{\Gamma,\Theta}^{\mathbf s}\simeq \tilde{\M}_{\Gamma,\iota_o(\Theta)}^{\iota(\mathbf s)}\quad \text{as $\Gamma$-manifolds.}\label{eq:MMbarasGspaces}
\eq
But, they are isomorphic as $\Gamma$-flow spaces if and only if $\iota(\mathbf{s})=\mathbf{s}$ and $\iota_{0}(\Theta)=\Theta,$ and therefore the flows are typically irreversible.   
\end{rem}
Our proof will crucially rely on the following property of Anosov subgroups, where we use the terminology from \eqref{eq:JTheta}  and Definition \ref{def:Benoistpositive}.
\begin{prop}\label{prop:Anosovescape}
Let $\Gamma< G$ be $\Theta$-Anosov, and $\mathbf s\in\R^{\Theta}\setminus\{0\}$ be $(\Gamma,\Theta)$-admissible. Then any sequence $(\gamma_k)$ in $\Gamma$ possessing distinct boundary limits $\gamma_+=\lim_{k\to\infty}\gamma_k\neq\gamma_-=\lim_{k\to\infty}\gamma_k^{-1}\in\partial_\infty\Gamma$ satisfies 
\[ \lim_{k\to\infty} \mathcal J_\Theta^\mathbf{s}(\gamma_k)=+\infty \quad \textrm{ and } \quad \lim_{k\to\infty} \mathcal J_{\iota_o(\Theta)}^\mathbf{s}(\gamma_k)=+\infty ~.\]
\end{prop}
\begin{proof}By Lemma \ref{lem:Jordanfcnopp}, \eqref{eq:FpfAnosovbarFpfAnosov},   and \eqref{eq:regularbarregular} it suffices to prove the statement on $\mathcal J_\Theta^\mathbf{s}$. Consider a restricted root decomposition and set $\beta=w_\Theta^\mathbf{s}$. Writing
$$
\beta(\lambda(\gamma))=\norm{\lambda(\gamma)}\beta\bigg(\frac{\lambda(\gamma)}{\norm{\lambda(\gamma)}}\bigg),
$$ 
for any $\gamma\in\Gamma$ with $\lambda(\gamma)\neq 0$ and any norm on $\aL$, the compactness of $\set{X\in\mathcal L(\Gamma)}{\norm{X}=1}$ and the $(\Gamma,\Theta)$-admissibility of $\mathbf s$  mean that it suffices to show that $\norm{\lambda(\gamma_k)}\to \infty$, i.e.,  $\lambda(\gamma_k)\to \infty$ in $\aL$, for any sequence $(\gamma_k)$ such as in the statement. This follows from Lemma \ref{lem Anosov implies growth Jordan projections}.
\end{proof}

The following lemma uses cocycles associated to a section (Definition \ref{def cocycle from a section}) and Jacobian class functions (Definition \ref{def:JacbobF}).

\begin{lem} \label{lem: properness criterion preparation}Suppose that $\Gamma$  is $\Theta$-Anosov. Let $\gamma_k\in\Gamma$ be a sequence which possesses distinct boundary limits $\gamma_+=\lim \gamma_k\neq\lim \gamma_k^{-1}=\gamma_-\in \partial_\infty\Gamma$. Let $s\in \R$ and consider a section $\sigma$ of $\mathscr D^s_{>0}\F^-_\bullet$. If $(\mathbf x_k)_{k\geq 0}$ is a sequence in $\F^-$  converging to an element $\mathbf x\in\cF^-$  transverse to $\xi^+(\gamma_-)\in\F^+$, then the sequence $\mathcal C_\sigma(\mathbf x_k,\gamma_k)+s\mathcal J_{\F^+}(\gamma_k)$ is bounded. 
\end{lem}

\begin{proof}   Choose a basis $(X_{1,k},\dots,X_{n,k})$ of each nilpotent radical $\fn_{\xi^+(\gamma_k^-)}$ converging to a basis $(X_1,\dots,X_n)$ of $\fn_{\xi^+(\gamma_-)}$ as $k\to +\infty$. This is possible because we have
\bq
(\xi^+(\gamma_k^-),\xi^-(\gamma_k^+))\stackrel{k\to \infty}{\longrightarrow}(\xi^+(\gamma_-),\xi^-(\gamma_+))\label{eq:convxi}
\eq 
by Lemma \ref{lem diverging sequence hyperbolic group} and the continuity of the  boundary maps demanded in \ref{property:P2}, and the fact that the spaces $\fn_{\mathbf x^+}$ over $(\mathbf x^+,\mathbf x^-)\in\cF^\pitchfork$ form a continuous (even analytic) vector bundle. Since $\mathbf x_k\to\mathbf x$ as $k\to \infty$ and transversality is an open condition, the fundamental vector fields of $X_{1,k},\dots,X_{n,k}$ at $\mathbf x_k$ form a basis of $T_{\mathbf x_k}\F^-$ when $k$ is large enough.

It follows from Corollary \ref{cor:cocycle density bundle} that
\begin{equation} \mathcal C_\sigma(\mathbf x_k,\gamma_k)+s\mathcal J_{\F^+}(\gamma_k) = \log \frac{\sigma_{\mathbf x_k}\left(\overline{X_{1,k}}(\mathbf x_k),\dots,\overline{X_{n,k}}(\mathbf x_k) \right)}{\sigma_{\gamma_k\cdot\mathbf x_k}\left(\overline{X_{1,k}}(\gamma_k\cdot \mathbf x_k),\dots,\overline{X_{n,k}}(\gamma_k\cdot \mathbf x_k) \right)}~. \label{eq cocycle vs Jordan class function} 
\end{equation}
Note that by condition \ref{property:P4}(d), we have that $\gamma_k\cdot \mathbf x_k\to \xi^-(\gamma_+)$. Since both $\mathbf x$ and $\xi^-(\gamma_+)$ are transverse to $\xi^+(\gamma_-)$, the right hand side of \eqref{eq cocycle vs Jordan class function} converges to
\[  \log \frac{\sigma_{\mathbf x}\left(\overline{X_{1}}(\mathbf x),\dots,\overline{X_{n}}(\mathbf x) \right)}{\sigma_{\xi^-(\gamma_+)}\left(\overline{X_{1}}(\xi^-(\gamma_+)),\dots,\overline{X_{n}}(\xi^-(\gamma_+)) \right)}\in\R~, \] in particular it is bounded.
\end{proof}

\begin{cor} \label{cor: properness criterion multi-preparation}Suppose that $\Gamma$  is $\Theta$-Anosov. Let $\gamma_k\in\Gamma$ be a sequence which possesses distinct boundary limits $\gamma_+=\lim \gamma_k\neq\lim \gamma_k^{-1}=\gamma_-\in \partial_\infty\Gamma$. Let $\mathbf s\in \R^\Theta$ and consider a section $\sigma$ of $\mathscr D^{\mathbf s}_{>0}\F^-_\bullet$. If $(\mathbf x_k)_{k\geq 0}$ is a sequence in $\F^-$  converging to an element $\mathbf x\in\cF^-$  transverse to $\xi^+(\gamma_-)\in\F^+$, then the sequence $\mathcal C_\sigma(\mathbf x_k,\gamma_k)+\mathcal J^{\mathbf s}_\Theta(\gamma_k)$  is bounded.
\end{cor}

\begin{proof}
Consider a section $\sigma^\alpha$ of  $\mathscr D^{s_\alpha}_{>0}\F^-_\alpha$ for each $\alpha\in \Theta$, and let $\sigma'=\prod_{\alpha\in\Theta} (\Pi^{\F^-}_\alpha)^\ast \sigma^\alpha\in \mathscr D^{\mathbf s}_{>0}\F^-_\bullet$. According to Lemma \ref{lem bound difference cocycle over compact space}, it is sufficient to show that the sequence $\mathcal C_{\sigma'}(\mathbf x_k,\gamma_k)+\mathcal J^{\mathbf s}_\Theta(\gamma_k)$ is bounded. The claim follows from Lemma \ref{lem: properness criterion preparation}, the fact that projections onto simple flag manifolds are continuous and preserve transversality, and the equality 
\[ \mathcal C_{\sigma'}(\mathbf y,g)=\sum_{\alpha\in\Theta}\mathcal C_{\sigma^\alpha}(\Pi^{\F^-}_\alpha(\mathbf y),g) \quad \forall\, (\mathbf y,g)\in \F^-\times G.\vspace*{-1em}\]
\end{proof}

The following Lemma is the second crucial ingredient to our proof of Theorem \ref{thm: prop disc general case} besides Proposition \ref{prop:Anosovescape}.

\begin{lem} \label{lem: properness criterion general case}Suppose that $\Gamma$  is $\Theta$-Anosov. Let $\gamma_k\in\Gamma$ be a sequence which possesses distinct boundary limits $\gamma_+=\lim \gamma_k\neq\lim \gamma_k^{-1}=\gamma_-\in \partial_\infty\Gamma$. Let $\mathbf s\in\R^{\Theta}$ and consider a  converging sequence $\mathbf x_k\to\mathbf x\in\F^-$ with $\xi^+(\gamma_-)\transverse\mathbf x$ as well as a converging sequence of lifts $\nu_k\to\nu\in {\mathscr D}^{\mathbf s}_{>0}\cF^-_\bullet$ in the corresponding fibers.

 Then, if $\lim_{k\to\infty} \mathcal J_\Theta^\mathbf{s}(\gamma_k)=+\infty$, we have $\gamma_k\cdot \nu_k\to 0$ in ${\mathscr D}^{\mathbf s}\cF^-_\bullet$ as $k\to \infty$.
\end{lem}
\begin{proof}
Let us fix a section $\sigma$ of ${\mathscr D}^{\mathbf s}_{>0}\cF^-_\bullet$, and write $\nu_k=e^{t_k}\sigma_{\mathbf x_k}$ and $\nu=e^t\sigma_{\mathbf x}$, so that the convergence $\nu_k\to \nu\in{\mathscr D}^{\mathbf s}_{>0}\cF^-_\bullet$ is encoded in $\mathbf x_k\to\mathbf x\in\F^-$ and $t_k\to t\in\R$. We compute
\begin{align*}
\gamma_k\cdot \nu_k &= e^{t_k} \gamma_k\cdot \sigma_{\mathbf x_k}\\
&= e^{t_k+\mathcal C_{\sigma}(\mathbf x_k,\gamma_k)}\sigma_{\gamma_k\cdot\mathbf x_k}~,
\end{align*}
but $\sigma_{\gamma_k\cdot\mathbf x_k}\to \sigma_{\xi^-(\gamma_+)}$ because $\xi^+(\gamma_-)\pitchfork \mathbf x$, and $t_k+\mathcal C_{\sigma}(\mathbf x_k,\gamma_k)\to -\infty$ because of Corollary \ref{cor: properness criterion multi-preparation}, so $\gamma_k\cdot\nu_k\to 0\in{\mathscr D}^{\mathbf s}\cF^-_\bullet$.
\end{proof}
\begin{rem}The conclusion of Lemma \ref{lem: properness criterion general case} implies that  $\gamma_k\cdot \nu_k\to \infty$ in ${\mathscr D}^{\mathbf s}_{>0}\cF^-_\bullet$ as $k\to \infty$, in the sense that it leaves all compact subsets of the positive multidensity bundle ${\mathscr D}^{\mathbf s}_{>0}\cF^-_\bullet$.
\end{rem}

Finally we are in the position to prove Theorem \ref{thm: prop disc general case}.

\begin{proof}[Proof of Theorem \ref{thm: prop disc general case}]
The cocompactness of the action on $\widetilde{\mathcal K}_{\Gamma,\Theta}^{\mathbf s}$ comes from Proposition \ref{prop:cocompactness}, and any properly discontinuous action of a torsion-free group is free, so all that is left to prove is the proper discontinuity of the $\Gamma$-action on $\tilde{\M}_{\Gamma,\Theta}^{\mathbf s}$.

By way of contradiction, suppose there exists  $(\nu^+_k,\nu^-_k)\in \tilde{\M}_{\Gamma,\Theta}^{\mathbf s}$ such that
$(\nu^+_k,\nu^-_k)\to (\nu^+,\nu^-) \in \tilde{\M}_{\Gamma,\Theta}^{\mathbf s}$ and $\delta_{k}\rightarrow \infty$ in $\Gamma$ with $\delta_{k}\cdot \nu_k\to (m^+,m^-)\in\tilde{\M}_{\Gamma,\Theta}^{\mathbf s}$.  By \cite[Lemma 2.11]{VolI}, there exists $r>0$,  $f\in \Gamma$ such that by passing to a   subsequence we get $d_\infty((f\delta_k)^+,(f\delta_k)^-)\geq r$. Set $\gamma_k:=f\delta_k$ and $\mu^\pm:=f\cdot m^\pm$, so that  $d_\infty(\gamma_k^+,\gamma_k^-)\geq r$ and $\gamma_k\cdot(\nu^+_k,\nu^-_k)\to (\mu^+,\mu^-)\in\tilde{\M}_{\Gamma,\Theta}^{\mathbf s}$. 

Now, extract a further subsequence so that $\gamma_{k}$  converges to $\gamma_+\in \partial_{\infty}\Gamma$ and $\gamma_{k}^{-1}$  converges to $\gamma_-\in \partial_{\infty}\Gamma$ (note that $\gamma_+\neq \gamma_-$), and denote by $(\mathbf x^+,\mathbf x^-)\in\Fpf$ the base point of $ (\nu^+,\nu^-)$.  By virtue of $ (\nu^+,\nu^-) \in \tilde{\M}_{\Gamma,\Theta}^{\mathbf s}$, either $ \xi^+(\gamma_-)\pitchfork \mathbf x^-$ or $ \mathbf x^+\pitchfork\xi^-(\gamma_{-})$.  The first case $ \xi^+(\gamma_-)\pitchfork \mathbf x^-$ implies via Lemma \ref{lem: properness criterion general case} and Proposition \ref{prop:Anosovescape} that $\gamma_{k}\cdot \nu^-_{k}\rightarrow \infty$  in ${\mathscr D}^{\mathbf s}_{>0}\F^-_\bullet$, which is absurd because of the assumption $\gamma_{k}\cdot(\nu^+_k,\nu^-_k)\rightarrow (\mu^+,\mu^-)\in \tilde{\M}_{\Gamma,\Theta}^{\mathbf s}$.
In the second case $ \mathbf x^+\pitchfork\xi^-(\gamma_{-})$ we can repeat the same argument with $(\Theta,\mathbf s)$ replaced by $(\iota_o(\Theta),\iota(\mathbf s))$, as is justified by \eqref{eq:FpfAnosovbarFpfAnosov}, \eqref{eq:regularbarregular}, and \eqref{eq:MMbarasGspaces}. This leads to the analogous contradiction, and therefore the $\Gamma$-action on $\tilde{\M}_{\Gamma,\Theta}^{\mathbf s}$ is properly discontinuous. 
\end{proof}
\begin{rem}\label{rem:beyondAnosov}
We used properties of $\Theta$-Anosov subgroups in two different ways when proving Theorem \ref{thm: prop disc general case}: to give a definition of the proper discontinuity domain (Definition \ref{def prop disc domain} involves the limit maps) and to prove proper discontinuity. We  expect that Theorem \ref{thm: prop disc general case} is still valid for a larger class of discrete subgroups, provided they have a good notion of limit set in $\cF^\pm$ and growth properties of Jordan projections such as in Lemma \ref{lem Anosov implies growth Jordan projections}, as is the case for $\Theta$-transverse subgroups discussed in the introduction.
\end{rem}

\section{Dynamics on the quotient manifold}\label{sec:dynamics}

Consider opposite flag manifolds $\F^+,\F^-$ associated to subsets $\Theta,\iota_o(\Theta)\subset\Delta$ respectively, $\Fpf=\F^+\trtimes\F^-$ the transverse flag space, $\Gamma< G$ a torsion-free $\Theta$-Anosov subgroup, and  $\mathbf s\in\R^{\Theta}$ a $(\Gamma,\Theta)$-admissible parameter. Then we know from Theorem \ref{thm: prop disc general case} that $\Gamma$ acts properly discontinuously and freely on the $\phi^t$-invariant open set $\tilde{\M}_{\Gamma,\Theta}^{\mathbf s}\subset \bbL^{\mathbf s}_\Theta$ and cocompactly on the $\phi^t$-invariant closed set $\tilde{\cK}_{\Gamma,\Theta}^{\mathbf s}\subset \tilde{\M}_{\Gamma,\Theta}^{\mathbf s}$. Since the $\Gamma$-actions on these sets commute with the flow $\phi^t$, the quotient manifold
\[
\M_{\Gamma,\Theta}^{\mathbf s}:=\Gamma\backslash \tilde{\M}_{\Gamma,\Theta}^{\mathbf s}
\]
inherits a complete analytic flow, which we shall again denote by $\phi^t$. It preserves the compact set
\[
\cK_{\Gamma,\Theta}^{\mathbf s}:=\Gamma\backslash\tilde{\cK}_{\Gamma,\Theta}^{\mathbf s}\subset \M_{\Gamma,\Theta}^{\mathbf s},
\]
which we call the \emph{basic set}. This terminology will be justified shortly.

\subsection{Simplified notations}

As before, elements of the flow space $\bbL^{\mathbf s}_\Theta$ will be denoted by
\[ (\mu^+,\mu^-)\in \bbL^{\mathbf s}_\Theta\subset \mathscr D^{\mathbf s}_{>0}\cF_\bullet^+ \boxplus \mathscr D^{\mathbf s}_{>0}\cF_\bullet^-\,,\]
and their projections on the base space $\Fpf$ will be written as
\[ \big([\mu^+],[\mu^-]\big)\in \Fpf=\F^+\trtimes\F^-\subset \F^+\times\F^-\,.\]
 To further ease the notation, we write from now on in this section
\begin{align}\begin{split}
\L:=\L_{\Theta}^{\mathbf s},\qquad 
\tilde{\M}:=\tilde{\M}_{\Gamma,\Theta}^{\mathbf s},\qquad 
\tilde{\cK}:=\tilde{\cK}_{\Gamma,\Theta}^{\mathbf s},\qquad
\M:=\M_{\Gamma,\Theta}^{\mathbf s},\qquad
\cK:=\cK_{\Gamma,\Theta}^{\mathbf s}.\label{eq:simplifiednotation1}\end{split}
\end{align}

\subsection{Topological dynamics}
Let us now describe the topological dynamics of the flow $\phi^t$ on $\M$. The following result and its proof are analogous to the projective Anosov case considered in \cite[Sec.~3.5]{VolI}. There is some difference in the notation, so we repeat the arguments here for convenience. 

\begin{thm}[{C.f.\ \cite[Thm.~2]{VolI}}] \label{thm - trapped is non wandering is closure of periodic points is K general case}
The following subsets of $\M$ coincide:
\begin{itemize}
\item The  basic set $\cK$,
\item The non-wandering set $\mathcal{NW}(\phi^t)$ of the flow $\phi^t$,
\item The closure of the set $\mathrm{Per}(\phi^t)$ of periodic points of the flow $\phi^t$,
\item The trapped set ${\mathcal T}(\phi^t)$ of the flow $\phi^t$,
\item The set $\overline{\bigcap_{t\in\R}\phi^t(U)}$ for any relatively compact open subset $U\subset \M$ containing $\cK$.
\end{itemize}
Moreover, the restriction of the flow $\phi^t$ to $\mathcal{K}$ is topologically transitive.
\end{thm}
The proof of Theorem \ref{thm - trapped is non wandering is closure of periodic points is K general case} is given below after preparing some auxiliary lemmas closely resembling those in \cite[Sec.~3.5]{VolI}.

\begin{lem}[{C.f.\ \cite[Lemma 3.10]{VolI}}] \label{lem topological dynamics general case} Consider sequences $x_k\in \M$ and $t_k\to+\infty$ and points $x,y\in\M$ such that $x_k\to x\in \M$ and $\phi^{t_k}(x_k)\to y\in \M$. Let $(\nu^+,\nu^-)\in\tilde{\cM}$ (resp. $(\mu^+,\mu^-)\in\tilde{\cM}$) be a lift of $x$ (resp.\ of $y\in\M$). Then $[\nu^+]\in \Lambda^+_\Gamma$ and $[\mu^-]\in \Lambda^-_\Gamma$.
\end{lem}
\begin{proof}
Consider lifts $(\nu^+_k,\nu^-_k)\in\tilde{\cM}$ of $x_k$  such that $(\nu^+_k,\nu^-_k)\to (\nu^+,\nu^-)$. There is a sequence $\delta_k\in \Gamma$ such that $\delta_k\cdot(e^{-t_k} \nu^+_k,e^{t_k}\nu^-_k)\to (\mu^+,\mu^-)$, and up to a subsequence, consider (applying \cite[Lemma 2.11]{VolI} to the sequence $(\delta_k^{-1})$) some $f\in \Gamma$ such that the sequence $\gamma_k=f\delta_k^{-1}$ satisfies $\gamma_k\to\gamma_+\in\partial_\infty \Gamma$ and $\gamma_k^{-1}\to\gamma_-\in\partial_\infty \Gamma$ with $\gamma_+\neq \gamma_-$. Write $(\mu_k^+,\mu^-_k)=\delta_k\cdot(e^{-t_k} \nu^+_k,e^{t_k}\nu^-_k)$.

Assuming that $\xi^+(\gamma_-)$ is transverse to $[\mu^-]$ leads to a contradiction: since $\mu_k^-\to\mu^-\neq 0$, Lemma \ref{lem: properness criterion general case} would imply that $\gamma_k\mu_k^-\to 0$,  yet $\gamma_k\mu_k^-=e^{t_k}f\nu_k^-$ and $\nu_k^-\to \nu^-\neq 0$. Since $(\mu^+,\mu^-)\in\tilde\M$, we must have $[\mu^+]\pitchfork \xi^-(\gamma_-)$, therefore $\gamma_k[\mu^+_k]\to \xi^+(\gamma_+)$ by \ref{property:P4}. This means that $f[\nu^+]=\xi^+(\gamma_+)$, and in particular $[\nu^+]\in\Lambda^+_\Gamma$.

This implies that $\xi^+(\gamma_+)\pitchfork f[\nu^-]$, so $\gamma_k^{-1}f[\nu^-_k]\to \xi^-(\gamma_-)$, i.e. $[\mu^-]=\xi^-(\gamma_-)\in\Lambda^-_\Gamma$.
\end{proof}

\begin{lem}[{C.f.\ \cite[Lemma 3.5]{VolI}}] \label{lem correspondence periodic orbits conjugacy classes general case} Let $x\in \M$ and consider a lift $(\nu^+,\nu^-)\in \tilde\M$. Then $x\in \mathrm{Per}(\phi^t)$ if and only if there is $\gamma\in\Gamma\setminus\{e\}$ such that $[\nu^+]=\xi^+(\gamma^+)$ and $[\nu^-]=\xi^-(\gamma^-)$.\footnote{Recall that we assume $\Gamma$ to be torsion-free here, so that any $\gamma\in\Gamma\setminus\{e\}$ has infinite order.} In this case, if $\gamma$ is primitive, then the period of $x$ is $\mathcal J^{\mathbf s}_\Theta(\gamma)$.
\end{lem}

\begin{proof} 
Assume that $x\in \mathrm{Per}(\phi^t)$, and let $T>0$ be such that $\phi^T(x)=x$. There is $\gamma\in\Gamma\setminus\{e\}$ such that $(e^{-T}\nu^+,e^{T}\nu^-)=\gamma\cdot(\nu^+,\nu^-)$. If $[\nu^-]$ were to be transverse to $\xi^+(\gamma^-)$, then Lemma \ref{lem: properness criterion general case} would imply that $\gamma^k\cdot\nu^-\to 0$ as $k\to +\infty$, which is a contradiction with $\gamma^k\cdot\nu^-=e^{kT}\nu^-$. As $(\nu^+,\nu^-)\in\tilde{\cM}$, we must have $[\nu^+]\transverse \xi^-(\gamma^-)$, so $\gamma^k[\nu^+]\to \xi^+(\gamma^+)$ as $k\to +\infty$, i.e. $[\nu^+]=\xi^+(\gamma^+)$. Since $[\nu^-]\pitchfork [\nu^+]=\xi^+(\gamma^+)$, we also find that $\gamma^k[\nu^-]\to \xi^-(\gamma^-)$ as $k\to -\infty$, therefore $[\nu^-]=\xi^-(\gamma^-)$.

If the assumption is now that $[\nu^+]=\xi^+(\gamma^+)$ and $[\nu^-]=\xi^-(\gamma^-)$ for some element $\gamma\in\Gamma\setminus\{e\}$, then combining  \eqref{eq:lLs}, \eqref{eq:defmalpha} and \eqref{eq:lambdaF} yields $\phi^{w_{{\mathbf s},\Theta}(\lambda(\gamma))}(x)=x$.
\end{proof}

\begin{proof}[Proof of Theorem \ref{thm - trapped is non wandering is closure of periodic points is K general case}]\label{proof:thm2} 
Any flow satisfies $\overline{\mathrm{Per}(\phi^t)}\subset \mathcal{NW}(\phi^t)$.  Let $x\in \mathcal{NW}(\phi^t)$, and consider sequences $x_k\to x$ in $\M$ and $t_k\to +\infty$ in $\R$ such that $\phi^{t_k}(x_k)\to x$. If $(\nu^+,\nu^-)\in \tilde{\cM}$ is a lift of $x$, then Lemma \ref{lem topological dynamics general case} shows that $[\nu^+]\in\Lambda^+_\Gamma$ and $[\nu^-]\in\Lambda^-_\Gamma$,  hence $x\in \cK$, i.e. $\mathcal{NW}(\phi^t)\subset\cK$.

The set of poles of $\Gamma$, i.e. pairs $(\gamma^+,\gamma^-)$ of fixed points of elements $\gamma\in\Gamma\setminus\{e\}$, is dense in the space $\partial_\infty\Gamma^{(2)}$ of pairs of distinct points \cite[Corollary 8.2.G]{Gr}, hence $\cK\subset\overline{\mathrm{Per}(\phi^t)}$ thanks to Lemma \ref{lem correspondence periodic orbits conjugacy classes general case}.

Let $U\subset \M$ be a relatively compact subset containing $\cK$, and consider $y\in \overline{\bigcap_{t\in\R}\phi^t(U)}$ and a lift $(\mu^+,\mu^-)\in\tilde{\cM}$. Consider first sequences $t_k\to +\infty$ and $x_k\in U$ such that $\phi^{t_k}(x_k)\to y$,  and assume without loss of generality that $x_k$ has a limit $x\in \M$.  Lemma \ref{lem topological dynamics general case} shows that $[\mu^-]\in\Lambda^-_\Gamma$. Consider now sequences $s_k\to -\infty$ and $z_k\in U$ such that $y=\lim\phi^{s_k}(z_k)$, and assume that $z_k\to z\in \M$. Applying Lemma \ref{lem topological dynamics general case} to  $\phi^{-s_k}\left( \phi^{s_k}(y)\right)$  shows that $[\mu^+]\in\Lambda^+_\Gamma$, hence $y\in \cK$. 

The fact that $\cK$ is the closure $\overline{\mathrm{Per}(\phi^t)}$ of periodic points shows that $\cK\subset \bigcap_{t\in\R}\phi^t(\cK)$, so finally $\overline{\bigcap_{t\in\R}\phi^t(U)}=\cK$.

Since $\cK$ is compact and $\phi^t$-invariant, we automatically get $\cK\subset {\mathcal T}(\phi^t)$. Let $x\in {\mathcal T}(\phi^t)$ and consider a lift $(\nu^+,\nu^-)\in\tilde{\cM}$. Let $y\in \M$ be an $\omega$-limit point of $x$, i.e. $y=\lim_{k\to +\infty}\phi^{t_k}(x)$ for some sequence $t_k\to +\infty$ (it exists because $x$ is trapped). Then Lemma \ref{lem topological dynamics general case} shows that $[\nu^+]\in \Lambda^+_\Gamma$. The same reasoning applied to an $\alpha$-limit point of $x$ (i.e. some $\lim\phi^{s_k}(x)$ with $s_k\to -\infty$) shows that $[\nu^-]\in\Lambda^-_\Gamma$, hence $x\in \cK$.

It remains to prove that the restriction of the flow $\phi^t$ to $\cK$ is topologically transitive. This is again done by complete analogy to the projective case \cite[Lemma 3.12]{VolI}: By \cite[8.2.H]{Gr}, we may consider $(a,b)\in \partial_\infty\Gamma^{(2)}$ whose $\Gamma$-orbit is dense in $\partial_\infty\Gamma^{(2)}$.  Let $x\in \cK$ be an element admitting a lift $(\nu^+,\nu^-)\in\tilde\cK$ with $[\nu^+]=\xi(a)$ and $[\nu^-]=\xi^*(b)$. Let $y\in \cK$ and consider a lift $(\mu^+,\mu^-)\in\tilde\cK$. Let $U\subset \cK$ be an open subset containing $y$, and $\tilde U\subset\tilde\cK$ its preimage. As the bundle projection $p:\L\to \Fpf$ is an open map, there is $\gamma\in\Gamma$ such that $\gamma\cdot ([\mu^+],[\mu^-])\in p(\tilde U)$, i.e., there is $t\in \R$ such that $\gamma\cdot (e^{-t} \mu^+,e^{t}\mu^-) \in\tilde U$, hence $\phi^t(x)\in U$.
\end{proof}

\subsection{Hyperbolicity of the non-wandering set} \label{sec:hyperbolicity}
We now turn to the differentiable dynamics of the flow $\phi^t$ on the manifold $\mathcal M$, and establish the last missing piece in the definition of Smale's Axiom A.
\begin{lem} \label{lem:hyperbolicset} The compact $\phi^t$-invariant subset $\mathcal K\subset\mathcal M$ is a hyperbolic set.
\end{lem}

\subsubsection{The original definition of Anosov subgroups} The proof of Lemma \ref{lem:hyperbolicset} is essentially a rephrasing of Labourie's original definition just as in \cite[Lem. 3.13]{VolI}. While Labourie worked with an Anosov flow on a closed manifold with a $\Gamma$-cover  \cite[Sec. 3.1]{labourie}, Guichard and Wienhard used less regular flows \cite[Def. 2.10]{GW12}, allowing us to work with Sambarino's refraction flow $\psi^t_\sigma:\mathfrak{R}_\sigma\to\mathfrak{R}_\sigma$ \cite[Def. 5.3.4]{SAM24} (already discussed in Section \ref{sec:refractionflows}). Consider the $\Gamma$-invariant maps
\begin{equation}
 f_+:\map{\partial_\infty\Gamma^{(2)}\times\R}{\F^+}{(z^+,z^-,\tau)}{\xi^+(z^+)} , \quad f_-:\map{\partial_\infty\Gamma^{(2)}\times\R}{\F^-}{(z^+,z^-,\tau)}{\xi^-(z^-)}.
\label{eqn f pm} 
\end{equation}
The vector bundles $V^\pm:=\Gamma\backslash f_\pm^*T\F^\pm$  over the refraction flow space $\mathfrak{R}_\sigma$ each come equipped with a flow by vector bundle isomorphisms
\begin{equation} \psi^t_\pm:\map{V^\pm}{V^\pm}{\Gamma\cdot((z^+,z^-,\tau),\mathbf v)}{\Gamma\cdot((z^+,z^-,t+\tau),\mathbf v)} \label{eqn:Labourie flow}
 \end{equation}
lifting the refraction flow $\psi_\sigma^t$, i.e. such that the following diagrams commute:
\begin{center}
\begin{tikzcd}
V^+ \arrow[r, "\psi_{+}^t"] \arrow[d] 
& V^+  \arrow[d] & V^- \arrow[r, "\psi_{-}^t"] \arrow[d] 
& V^-  \arrow[d] \\
\mathfrak{R}_{\sigma} \arrow[r,"\psi_{\sigma}^t" ]
& \mathfrak{R}_{\sigma} & \mathfrak{R}_{\sigma} \arrow[r,"\psi_{\sigma}^t" ]
& \mathfrak{R}_{\sigma}
\end{tikzcd}
\end{center}

The equivalence of Definition \ref{def:ThetaAnosov} with \cite[Def. 2.10, Rem. 2.11]{GW12}, i.e. \cite[Theorem 1.1]{KLP17}, gives the following property of $\Theta$-Anosov subgroups:
\begin{enumerate}[label=(A\arabic*)] \setcounter{enumi}{3}
\item The flow $\psi^t_+$ (resp. $\psi^t_-$) is uniformly dilating (resp. contracting).\label{property:P5}
\end{enumerate}

\subsubsection{The tangent bundle of the flow space}\label{sec:tangent bundle flow space} Given a point $\mu=(\mu^+,\mu^-)\in \L$, denote by $([\mu^+],[\mu^-])\in\Fpf$ its projection, and by $\fn^+\subset \fg$ (resp. $\fn^-\subset\fg$) the nilpotent radical of the parabolic Lie algebra of the stabilizer of $[\mu^+]\in\F^+$ (resp. $[\mu^-]\in\F^-$). Using once again the notation $\overline X(\mu)=\left.\frac{d\,}{dt}\right\vert_{t=0}\exp(tX)\cdot\mu$ for the fundamental vector field of $X\in\fg$, we get a grading of the tangent space $T_\mu\L$ as
\begin{equation}
T_\mu\L=E^+_\mu\oplus E^0_\mu\oplus E^-_\mu \label{tangent grading}
\end{equation}
where the vector subspaces $E^+_\mu,E^0_\mu,E^-_\mu\subset T_\mu\L$ 
\[
E^+_\mu = \set{\overline X(\mu)}{X\in \fn^-} ~,\qquad E^0_\mu = \R\cdot \left.\frac{d\,}{dt}\right\vert_{t=0}\phi^t(\mu)~,\qquad E^-_\mu = \set{\overline X(\mu)}{X\in \fn^+} ~,
\]
define $G\times\{d\phi^t\}$-equivariant subbundles $E^+,E^0,E^-\subset T\L$.

Restricting this structure to $\widetilde{\mathcal{M}}\subset \mathbb{L}$ and taking the quotient by the $\Gamma$-action, the grading \eqref{tangent grading} descends yielding a flow invariant splitting
\[
T\mathcal{M}= E^{+}\oplus E^{0}\oplus E^{-}\,.
\]
The sign exchange between $E^\pm$ and $\fn^\mp$ is chosen so that the projections $$\pi_\pm:=p_{\F^\pm}\circ p:\mathbb{L}\to \F^\pm$$ induce $G$-equivariant isomorphisms
\begin{equation}
d\pi_\pm:E^{\pm}\to \pi_{\pm}^{*}T\F^{\pm} \label{eq:derivative projection bundle isomorphism}
\end{equation}
as $G$-equivariant vector bundles over $\mathbb{L}.$

\subsubsection{Proof of hyperbolicity} We now have all the ingredients to prove Lemma \ref{lem:hyperbolicset}.
\begin{proof}[Proof of Lemma \ref{lem:hyperbolicset}] Both sides in \eqref{eq:derivative projection bundle isomorphism} carry a 1-parameter group of vector bundle isomorphisms: the restriction of $d\phi^t$ on $E^{\mp}$, and the the flow $\varphi_\pm^t:\pi_{\pm}^{*}T\F^{\pm}\to \pi_{\pm}^{*}T\F^{\pm}$ defined by 
\[ \varphi_\pm^t:\map{\pi_{\pm}^{*}T\F^{\pm}}{\pi_{\pm}^{*}T\F^{\pm}}{(x,\mathbf v)}{(\phi^t(x),\mathbf v)}
\]
The flow invariance $\pi_\pm\circ\phi^t=\pi_\pm$, when differentiated, shows that the isomorphism \eqref{eq:derivative projection bundle isomorphism} intertwines the flows. The $G$-equivariant vector bundles $ \pi_{\pm}^{*}T\F^{\pm}\to \L$ can be restricted to $\Gamma$-invariant bundles over $\widetilde{\mathcal M}\subset\L$ and descended to vector bundles $\mathcal T^\pm$ over $\mathcal M$ each equipped with a 1-parameter group of vector bundle isomorphisms still denoted by $\varphi^t_\pm$. Upon restriction to $\mathcal K\subset\mathcal M$, we obtain a commutative diagram 
\begin{equation}
\begin{tikzcd}
E^{\pm}\vert_{\mathcal K} \arrow[r] \arrow[d] 
& \mathcal T^\pm\vert_{\mathcal K}  \arrow[d] \\
\mathcal K \arrow[r]
& \mathcal{K} 
\end{tikzcd} \label{diagram tangent over K}
\end{equation}
of vector bundle isomorphisms intertwining the flows $d\phi^t$ and $\varphi^t_\pm$ (because of the invariance $\pi_\pm\circ\phi^t=\pi_\pm$) given by descending $d\pi_\pm$ for the top horizontal arrow and the identity on the bottom.

The discussion of the refraction flow in Section \ref{sec:refractionflows} can be summarized in the  commutative diagram
\begin{center}
\begin{tikzcd}
\partial_{\infty}\Gamma^{(2)}\times \R \arrow[r, "\widetilde{\Phi}_{\sigma}"] \arrow[d] 
& \widetilde{\mathcal{K}} \arrow[r] \arrow[d]
& \widetilde{\mathcal{M}} \arrow[d] \\
\mathfrak{R}_{\sigma} \arrow[r,"\Phi_{\sigma}" ]
& \mathcal{K} \arrow[r]
& \mathcal{M}
\end{tikzcd}
\end{center}
where the top row consists of $\Gamma$-spaces,
vertical arrows are $\Gamma$-covers, unlabeled horizontal arrows are inclusions, and $\Phi_{\sigma}$ is a bi-Hölder conjugacy between the refraction flow on $\mathfrak{R}_{\sigma}$ and the restriction of the flow $
\phi^{t}: \mathcal{M}\rightarrow \mathcal{M}$ to the compact invariant set $\mathcal{K}\subset \mathcal{M}.$  

Now the $\Gamma$-equivariant map
\[\map{f^\ast_\pm T\F^\pm}{\pi_\pm^\ast T\F^\pm\vert_{\widetilde{\mathcal K}}}{\big( (z^+,z^-,\tau),\mathbf v\big)}{\Big( \sigma\big( \xi^\pitchfork(z^+,z^-) \big), \mathbf v\Big) }
\]
descends to a vector bundle isomorphism $V^\pm\to \mathcal T^\pm\vert_{\mathcal K}$ that conjugates $\psi^t_\pm$ to $\varphi^t_\pm$, so combining with \eqref{diagram tangent over K} we get the following commutative diagram
\begin{center}
\begin{tikzcd}
V^\pm \arrow[r] \arrow[d] 
& \mathcal T^\pm\vert_{\mathcal K}  \arrow[d]
& E^{\pm}\vert_{\mathcal K} \arrow[l] \arrow[d] \\
\mathfrak{R}_{\sigma} \arrow[r,"\Phi_{\sigma}" ]
& \mathcal{K} 
& \mathcal{K} \arrow[l,swap,"\mathrm{Id}"]
\end{tikzcd}
\end{center}
of vector bundle isomorphisms conjugating the flows $\psi^t_\pm$, $\varphi^t_\pm$ and $d\phi^t$, so the hyperbolicity of $\mathcal K$ follows from \ref{property:P5}.
\end{proof}

\section{Geometric structures on homogeneous spaces} \label{sec:geometric_structures}

In this section, we will explain the existence of various symplectic and contact structures on the homogeneous transverse flag and flow spaces considered in this article.

\subsection{Geometry of transverse flag spaces}\label{sec:geometry transverse flag spaces}
\subsubsection{The universal para-Kähler structure} \label{sec:para kahler}

A para-K\"{a}hler structure on a smooth manifold $M$ of dimension $2k$ is a tuple  $(\langle - , - \rangle, I, \omega)$ 
where $\langle - , - \rangle$ is a neutral signature $(k,k)$ pseudo-Riemannian metric on $M$, the bundle homomorphism $I: TM\rightarrow TM$ is parallel for the Levi-Civita connection of $\langle -, - \rangle$ and idempotent (i.e.\ $I^{2}=\mathrm{Id}$), and $\langle I(-), -\rangle=\omega$ is a symplectic form.  It is the analog of a K\"{a}hler structure where the complex numbers are replaced by the paracomplex numbers $a+b\varepsilon$ where $a,b\in \R$ and $\varepsilon^{2}=1.$  The $\pm 1$-eigenspaces of $I: TM\rightarrow TM$ constitute a pair of transverse, integrable Lagrangian distributions for the symplectic form $\omega$, and the para-K\"{a}hler structure is equivalent to a symplectic manifold with a pair of transverse Lagrangian foliations.

 We now show this is precisely the structure canonically present on the transverse flag spaces $\F^{\pitchfork}\simeq G/L.$   Given a point $\mathbf x=(\mathbf x^+,\mathbf x^-)\in \Fpf=\F^+\trtimes\F^-$, denote by $\fn^+\subset \fg$ (resp. $\fn^-\subset\fg$) the nilpotent radical of the parabolic Lie algebra of the stabilizer of $\mathbf x^+\in\F^+$ (resp. $\mathbf x^-\in\F^-$). Using once again the notation $\overline X(\mathbf x)=\left.\frac{d\,}{dt}\right\vert_{t=0}\exp(tX)\cdot\mathbf x$ for the fundamental vector field of $X\in\fg$, we get a splitting of the tangent space $T_{\mathbf x}\Fpf$ as
\begin{equation}
T_{\mathbf x}\Fpf=\mathcal L^+_{\mathbf x}\oplus \mathcal L^-_{\mathbf x} \label{tangent grading transverse flag space}
\end{equation}
where the vector subspaces $\mathcal L^\pm_{\mathbf x} = \{\overline X(\mathbf x)\,|\,X\in \fn^\mp\}$ define $G$-equivariant distributions $\mathcal L^+,\mathcal L^-\subset T\Fpf$ tangent to the fibers of the projections $p_{\F^\pm}:\Fpf\to\F^\pm$. 
The $\mathrm{Stab}_G(\mathbf x)$-equivariant isomorphism
\begin{align}\label{iso}
\map{\fn^-\oplus\fn^+}{T_{\mathbf x}\Fpf}{X}{\overline{X}(\mathbf x)}
\end{align}
can be used to push forward the Killing form $\mathrm{B}_\g$ to a $G$-invariant pseudo-Riemannian metric $\langle -,-  \rangle_{\Fpf}$ of signature $(n,n)$ on $\F^{\pitchfork}$ (where $n=\dim\F^+=\dim\F^-$)\footnote{This is the same structure as defined in Section \ref{sec:pseudoriemannianstructure}.}
\begin{equation}
\langle \overline X_1(\mathbf x),\overline X_2(\mathbf x)\rangle_\Fpf := \mathrm{B}_\g(X_1,X_2)\,,\quad X_1,X_2\in \fn^-\oplus\fn^+ \label{eq:pseudo Riemannian metric transverse flag space}
\end{equation} for which the subbundles $\mathcal L^\pm\subset T\Fpf$ are totally isotropic. Note that since the Lie subalgebras $\fn^+,\fn^-$ satisfy $[\fn^+,\fn^-]\subset\fl$, where $\fl$ denotes the Lie algebra of the Levi stabilizer of $\mathbf x$, we have that
\begin{equation}
\langle \overline{[X_1,X_2]}(\mathbf x),\overline X_3(\mathbf x)\rangle_\Fpf = \mathrm{B}_\g([X_1,X_2],X_3),\qquad \forall\, X_1,X_2,X_3\in \fn^-\oplus\fn^+. \label{eq:pseudo Riemannian product Lie bracket}
\end{equation}
 Using the splitting \eqref{tangent grading transverse flag space}, define the endomorphism of the tangent bundle $T\Fpf$
\[ I_\Fpf:\map{\mathcal L^+\oplus \mathcal L^-}{\mathcal L^+\oplus \mathcal L^-}{\mathbf v^++\mathbf v^-}{\mathbf v^+-\mathbf v^-}\,.
\]
Defining $\omega_{\F^{\pitchfork}}:=\langle I_\Fpf (-),\mathbf -\rangle_\Fpf$, we compute
\[
\omega_{\F^{\pitchfork}}\left(\mathbf v^+_1+\mathbf v^-_1, \mathbf v^+_2+\mathbf v^-_2\right)=\langle \mathbf v^+_1,\mathbf v^-_2\rangle_\Fpf - \langle \mathbf v^-_1, \mathbf v^+_2\rangle_\Fpf\,,\quad \mathbf v^+_1, \mathbf v^+_2\in\mathcal L^+\,,~\mathbf v^-_1,\mathbf v^-_2\in\mathcal L^-\,,
\]
and find that $\omega_\Fpf$ is skew-symmetric and non-degenerate (because the Killing form induces a non-degenerate pairing $\mathrm{B}_\g:\fn^+\times\fn^-\to\R$). Cartan's magic formula applied to \eqref{eq:pseudo Riemannian product Lie bracket} shows that $d\omega_\Fpf=0$, i.e. that $\omega_\Fpf$ is a symplectic form, so we have defined a $G$-invariant  para-K\"{a}hler manifold $(\F^{\pitchfork}, \langle - , -\rangle_{\F^{\pitchfork}}, I_{\F^{\pitchfork}}, \omega_{\F^{\pitchfork}})$.

\subsubsection{Kostant-Kirillov-Souriau symplectic structures}\label{sec:KKS}

Given smooth functions $f,g: \fg^{*}\rightarrow \R,$ the differentials at $\eta\in \fg^{*},$
\bqn
df_{\eta},dg_{\eta}: \fg^{*}\rightarrow \R 
\eqn
may be understood as elements $df_{\eta}, dg_{\eta}\in \fg$  through the canonical  parallelism $T\fg^{*}\simeq \fg^{*}\times \fg^{*}$ and isomorphism $\left(\fg^{*}\right)^{*}\simeq \fg.$   The Poisson bracket
\[
\{f,g\}(\eta)=-\eta([df_{\eta}, dg_{\eta}])\in \R
\]
defines the Kostant-Kirillov-Souriau (KKS) Poisson structure on $\fg^{*}$ which is well-defined for an arbitrary finite dimensional real Lie algebra $\fg.$   The coadjoint orbits $G\cdot \eta\subset \fg^{*}$ of elements $\eta\in \g^\ast$ are equal to the symplectic leaves of this Poisson structure \cite{kirillov}, and we denote by $\omega_{\eta}^{\mathrm{KKS}}$ the corresponding symplectic form on the coadjoint orbit $\mathcal{O}_{\eta}=G\cdot \eta\subset \fg^{*}.$   Evaluating the KKS symplectic form at $T_\eta\mathcal O_\eta\simeq \g/\mathfrak l_\eta$, where $\mathfrak l_\eta$ is the Lie algebra of $\mathrm{Stab}_G(\eta)$, it is given by
\begin{equation} 
\omega^{\mathrm{KKS}}_\eta(\overline{X_1}(\eta),\overline{X_2}(\eta))=-\eta([X_1,X_2]), \label{eq:KKS symplectic form}
\end{equation}
where on the left-hand side we consider fundamental vector fields for the coadjoint action. 

Now, consider  $X\in \g_{\mathrm{hyp}}$ associated to  the transverse flag space $\F_X^\pitchfork:=\F_X\trtimes\F_{-X}$, and let $\eta:=\mathrm{B}_\g(X,-)\in\g^\ast$ be its Killing dual. Non-degeneracy of $\mathrm{B}_\g$ means that  $\mathrm{Stab}_G(X)=\mathrm{Stab}_G(\eta)=L_X$, resulting in a $G$-equivariant diffeomorphism
\[ \map{\mathcal O_\eta=G\cdot \eta}{\F_X^\pitchfork}{g\cdot\eta}{g\cdot (\fp_X,\fp_{-X})}\,.\] Apart from some very specific cases, the push-forward of   $\omega^{\mathrm{KKS}}$ on $\F_X$ is  different from the symplectic form $\omega_{\F_X^\pitchfork}$ defined in the previous section, and choosing a different $X'\in \g_{\mathrm{hyp}}$ with $\F_{X'}^\pitchfork=\F_{X}^\pitchfork$ results in different KKS symplectic forms. They however share the property that the distributions $\mathcal L^\pm$ of $T\Fpf$ are Lagrangian.

\subsubsection{Period functions and closed 2-forms}
Consider a transverse flag manifold $\Fpf=\F^+\trtimes\F^-$, and recall  from Definition \ref{def:space of period functions}  its space $\mathcal P(\Fpf)$ of period functions of $\Fpf$, i.e. real valued $G$-invariant functions on   $\Stab_\Fpf=\set{(\xbf,g)\in\Fpf\times G}{g\cdot\xbf=\xbf}$ that restrict to a group homomorphism on each fiber.
\begin{definition} \label{def:infinitesimal period function}
Let $\ell\in\mathcal P(\Fpf)$. The \emph{infinitesimal period function} of $\ell$ is the function $d\ell:\Fpf\to\g^*$ defined in the following way: for $\xbf\in\Fpf$, denote by $\fl$ the Lie algebra of $\Stab_G(\xbf)$, and let $d\ell(\xbf)\in\g^*$ be defined by
\[ d\ell(\xbf)(X)=\left\lbrace\begin{array}{cl}d_e\ell(\xbf,-)(X) & \textrm{if } X\in\fl \\ 0 & \textrm{if } X\in\fl^\perp \end{array}\right.
\]
where the orthogonal is considered with respect to the Killing form.
\end{definition}

The $G$-invariance of a period function $\ell\in\mathcal P(\Fpf)$ makes the infinitesimal period function $G$-equivariant, thus the range of $d\ell:\Fpf\to\g^*$ is a coadjoint orbit
\[ \mathcal O_\ell:= \set{d\ell(\xbf)}{\xbf\in\Fpf}\subset\g^*\,.\]
\begin{prop} \label{prop:2forms from periods}
The assignment \[\begin{array}{ccc}{\mathcal P(\Fpf)}& \to & {\Omega^2(\Fpf)}\\ {\ell}& \mapsto & {\omega^\ell:=d\ell^*\omega^{\mathrm{KKS}}}\end{array}\]
 is a a linear  injective  map into closed $G$-invariant 2-forms.
\end{prop}
\begin{proof}
Closedness is a consequence of the closedness of $\omega^{\mathrm{KKS}}$, and $G$-invariance comes from the invariance of both $d\ell$ and $\omega^{\mathrm{KKS}}$. Linearity comes from \eqref{eq:KKS symplectic form}:
\[ \omega^\ell_\xbf(\overline{X_1}(\xbf),\overline{X_2}(\xbf))=- d\ell(\xbf)\big( [X_1,X_2]\big)\,,\quad\forall X_1,X_2\in\g\,.\] 
 If $\omega^\ell=0$, then $\mathcal O_\ell$ consists of a single point, i.e. $\ell=0$.
\end{proof}

\begin{definition} \label{def:regular period function}
A period function $\ell\in\mathcal P(\Fpf)$ is \emph{regular} if $\omega^\ell$ is a symplectic form.
\end{definition}

Note that a period function $\ell\in\mathcal P(\Fpf)$ is regular if and only if $d\ell:\Fpf\to\mathcal O_\ell$ is a diffeomorphism.

The coadjoint orbit $\mathcal O_\ell$ is always  hyperbolic: the Killing dual of the linear form $d\ell(\xbf)$ at any $\xbf\in\Fpf$ must lie in $\fz_{\mathrm{split}}(\fl)$, because $d\ell(\xbf)$ vanishes on the Lie algebra of the maximal compact torus of $\fl$ as well as on the derived algebra $[\fl,\fl]$. In particular, $\mathcal O_\ell$ is always $G$-equivariantly diffeomorphic to some transverse flag space.

\begin{rem} If $\ell\in\mathcal P(\Fpf)$ is regular, then $\langle-,-\rangle_\ell:=\omega^\ell(I_\Fpf(-),-)$ is a pseudo-Riemannian metric of neutral signature, defining a (generally) different para-Kähler structure on $\Fpf$.
\end{rem}

\subsection{Geometry of flow spaces}

We now consider a $G$-equivariant affine line bundle $p:\L\to\Fpf$, denote by $\ell\in\mathcal P(\Fpf)$ its period function, and consider the vertical vector field $\mathcal X_\L$ on $\L$ defined by
\[ \mathcal X_\L(\hat\xbf)=\left.\frac{d\,}{dt}\right\vert_{t=0}\hat\xbf\cdot t\,,\quad \forall \hat\xbf\in\L\,.\]
\subsubsection{Pseudo-Riemannian structure}
Recall the $G\times\{d\phi^t\}$-equivariant splitting
\begin{equation*}
T\L=E^+\oplus E^0\oplus E^-
\end{equation*}
from Section \ref{sec:tangent bundle flow space}, where $E^0$ is spanned by the derivative of the flow, and $E^\pm$ are lifts of the Lagrangian distributions $\mathcal L^\pm$ of $T\Fpf$.
There is a unique $G$-invariant pseudo-Riemannian metric $\langle-,-\rangle_\L$ of signature $(n+1,n)$ on $\L$ with the following properties:
\begin{enumerate}
\item  $\langle\mathcal X_\L,\mathcal X_\L\rangle_\L\equiv 1\,,$
\item $E^+\oplus E^-\perp E^0\,,$
\item The projection $p:\L\to\Fpf$ is a pseudo-Riemannian submersion, i.e. its differential  restricts to an isometry from $(E^+\oplus E^-,\langle-,-\rangle_\L)$ to $(T\Fpf,\langle-,-\rangle_\Fpf)$.
\end{enumerate}
 
\subsubsection{The canonical 1-form} 
A similar procedure constructs a 1-form on $\L$.
\begin{definition}\label{def:canonical 1-form}
The \emph{canonical 1-form} $\tau\in\Omega^1(\L)$ is defined at $\hat\xbf\in \L$ by 
\[\tau_{\hat\xbf}\big(\mathcal X_\L(\hat\xbf)\big)=1  \textrm{ and } \tau_{\hat\xbf}(v)=0  \textrm{ if } v\in E^+_{\hat\xbf}\oplus E^-_{\hat\xbf}\,. \]
\end{definition}

Note that $\tau=\langle\mathcal X_\L,-\rangle_\L$ is invariant under the flow $\phi^t$, and it is non-zero on each fiber, so it can be seen as a  principal connection on the principal $\R$-bundle $p:\L\to\Fpf$. Since $\R$ is abelian, the curvature of this connection is simply $d\tau$.

\begin{prop} \label{prop:derivative canonical 1-form}
The canonical 1-form $\tau$ on $\L$ satisfies
\[ d\tau = p^*\omega^\ell\,.\]
It is a contact form if and only if $\ell$ is regular.
\end{prop}
\begin{proof}
Fix a base point $\hat\xbf\in\L$, and consider the orbit map \[\varphi_{\hat\xbf}:\map{G}{\L}{g}{g\cdot\hat\xbf.}\]
The  value of $\varphi_{\hat\xbf}^*\tau\in\Omega^1(G)$ at $e$ is
\[ \big(\varphi_{\hat\xbf}^*\tau\big)_e(X)=d\ell(\hat\xbf)(X)\,,\qquad \forall X\in\fg\,.\]
To simplify notations, write $\eta:=d\ell(\hat\xbf)\in\g^*$. Since $\varphi_{\hat\xbf}^*\tau$ is a left-invariant 1-form on $G$, we have that 
\[ \varphi_{\hat\xbf}^*\tau = \eta\circ \mathrm{m}\]
where $\mathrm{m}\in\Omega^1(G,\g)$ is the left-invariant Maurer-Cartan form (i.e. $\mathrm{m}_g(\mathbf v)=(d_eL_g)^{-1}(\mathbf v)\in \g$ for $\mathbf v\in T_gG$). The Maurer-Cartan equation tells us that for any $\mathbf v_1,\mathbf v_2\in T_gG$
\[
\big(\varphi_{\hat\xbf}^*d\tau\big)_g(\mathbf v_1,\mathbf v_2)=d\big(\varphi_{\hat\xbf}^*\tau\big)_g(\mathbf v_1,\mathbf v_2)= \big(\eta\circ d\mathrm{m}\big)_g(\mathbf v_1,\mathbf v_2)=-\eta\big( [\mathrm{m}_g(\mathbf v_1),\mathrm{m}_g(\mathbf v_2)]\big)\,.
\]
If we now consider the orbit map
\[ \varphi_\eta:\map{G}{\mathcal O_\eta}{g}{\eta\circ \Ad(g^{-1}),}\]
then we find that 
\[
 \big(\varphi_{\hat\xbf}^*p^*\omega^\ell\big)_g(\mathbf v_1,\mathbf v_2) = \big(\varphi_\eta^*\omega^{\mathrm{KKS}}\big)_g(\mathbf v_1,\mathbf v_2)= -\eta\big( [\mathrm{m}_g(\mathbf v_1),\mathrm{m}_g(\mathbf v_2)]\big)\,.
\]
By virtue of $\varphi_{\hat\xbf}$ being a surjective submersion, we find that $d\tau=p^*\omega^\ell$. In particular, $\tau$ is a contact form iff $\omega^\ell$ is a symplectic form, which is equivalent to $\ell$ being regular.
\end{proof}

\subsection{Geometry of quotient flows} \label{sec:geometry quotient}
\subsubsection{Regularity and multi-densities}
Considering opposite flag manifolds $\F^+,\F^-$ corresponding to subsets $\Theta,\iota_o(\Theta)\subset\Delta$ and $\mathbf s\in\R^\Theta$, it is particularly easy to tell when the flow space $\L_\Theta^\mathbf{s}$ over $\Fpf=\F^+\trtimes\F^-$ is regular, i.e. when its period function is regular.
\begin{prop} \label{prop:regular flow space}
The period function of the flow space $\L_\Theta^\mathbf{s}\to\Fpf$ is regular if and only if $\mathbf s_\alpha\neq 0$ for all $\alpha\in\Theta$.
\end{prop}
\begin{proof}
Let us fix a restricted root space decomposition, so that we are in the setting of Section \ref{sec:classification}. Fix some $\xbf\in\Fpf$ such that $\Stab_G(\xbf)=L_\Theta$. By \eqref{eq:lLs} and  Lemma \ref{lem:wmalpha} we have that
\[ \ell_{\L_\Theta^{\mathbf s}}(\xbf,\exp(X))=\sum_{\alpha\in\Theta}s_\alpha m_\alpha \mathrm{B}_\g(X_\alpha,X)\,,\quad \forall X\in\fl_\Theta\,,\]
where $X_\alpha$ is the co-root of $\alpha\in\Delta$ and $m_\alpha>0$. So $X:=\sum_{\alpha\in\Theta}s_\alpha m_\alpha X_\alpha$ is the Killing dual of $d\ell_{\L^{\mathbf s}_\Theta}(\xbf)$, and the map $d\ell_{\L^{\mathbf s}_\Theta}:\Fpf\to\mathcal O_{\ell_{\L^{\mathbf s}_\Theta}}$ is simply the projection $G/L_\Theta\to G/L_X$, so the flow space $\L_\Theta^{\mathbf s}$ is regular if and only if $L_X=L_\Theta$. The result follows from $L_X=L_{\Theta'}$ where $\Theta'=\set{\alpha\in\Theta}{s_\alpha\neq 0}$. 
\end{proof}

\subsubsection{Contact realization of refraction flows} \label{sec:contact refraction flows}

\begin{thm} \label{thm:contact flow}
Consider some transverse flag space $\Fpf$ and the associated subset $\Theta\subset\Delta$. Let $\Gamma<G$ be a torsion-free $\Theta$-Anosov subgroup, and $\mathbf s\in \R^\Theta$ a $(\Gamma,\Theta)$-admissible parameter. If $s_\alpha\neq 0$ for all $\alpha\in\Theta$, then the flow $\phi^t$ on the quotient manifold $\mathcal M_{\Gamma,\Theta}^{\mathbf s}$ is the Reeb flow of a contact form.

If it is not the case, there is a subset $\Theta'\subset\Theta$ and a $(\Gamma,\Theta')$-admissible parameter $\mathbf s'\in\R^{\Theta'}$ such that the flow on $\mathcal M_{\Gamma,\Theta'}^{\mathbf s'}$ is the Reeb flow of a contact form and there is a flow-equivariant surjective submersion $\mathcal M_{\Gamma,\Theta}^{\mathbf s}\to\mathcal M_{\Gamma,\Theta'}^{\mathbf s'}$ that restricts to  a homeomorphism $\mathcal K_{\Gamma,\Theta}^{\mathbf s}\to\mathcal K_{\Gamma,\Theta'}^{\mathbf s'}$.
\end{thm}
\begin{rem}\label{rem:refraction flows are contact} Note that a $\Theta$-Anosov subgroup is also $\Theta'$-Anosov for any $\Theta'\subset\Theta$. In particular, Theorem \ref{thm:contact flow} tells us that every refraction flow can be realized as the restriction of a contact Axiom A flow to its non-wandering set. 
\end{rem}
\begin{proof}
If $s_\alpha\neq 0$ for all $\alpha\in\Theta$, the result follows from Propositions \ref{prop:derivative canonical 1-form} and \ref{prop:regular flow space}. Now assume that $\Theta'=\set{\alpha\in\Theta}{s_\alpha\neq 0}\neq \Theta$, set $\mathbf s'=(s_\alpha)_{\alpha\in\Theta'}$ and consider the transverse flag space $\F_{\Theta'}^{\pitchfork}=\F_{\Theta'}\trtimes\F_{\iota_o(\Theta')}$. Then $\Gamma$ is $\Theta'$-Anosov, $\mathbf s'$ is $(\Gamma,\Theta')$-regular and by the first case the flow on the quotient manifold $\mathcal M^{\mathbf s'}_{\Gamma,\Theta'}$ is the Reeb flow of a contact form. The projection $\pi:\Fpf\to \F_{\Theta'}^{\pitchfork}$ induces a $G\times\R$-invariant map 
\[ \tilde f: \map{\L^{\mathbf s}_\Theta}{\L^{\mathbf s'}_{\Theta'}}{\big( (\mu^+_\alpha)_{\alpha\in\Theta},(\mu^-_\alpha)_{\alpha\in\Theta}\big)}{\big( (\mu^+_\alpha)_{\alpha\in\Theta'},(\mu^-_\alpha)_{\alpha\in\Theta'}\big)}\]
making the following diagram commutative
\begin{center}
\begin{tikzcd}
\L^{\mathbf s}_\Theta \arrow[d,swap, "p"]  \arrow[r, "\tilde f"] & \L_{\Theta'}^{\mathbf s'} \arrow[d,"p'"] \\
  \Fpf \arrow[r,"\pi"] & \F_{\Theta'}^{\mathbf s'}
\end{tikzcd}
\end{center}
The transverse limit maps $\xi^\pitchfork:\partial_\infty\Gamma^{(2)}\to \Fpf$ and $\xi_{\Theta'}^\pitchfork:\partial_\infty\Gamma^{(2)}\to \F_{\Theta'}^{\pitchfork}$  satify $\xi_{\Theta'}^\pitchfork=\pi\circ \xi^\pitchfork$ (because of Lemma \ref{lem:proxdomination}). This implies that $\tilde f$ maps $\widetilde{\mathcal M}^{\mathbf s}_{\Gamma,\Theta}$ onto  $\widetilde{\mathcal M}^{\mathbf s'}_{\Gamma,\Theta'}$ and $\widetilde{\mathcal K}^{\mathbf s}_{\Gamma,\Theta}$ onto $\widetilde{\mathcal K}^{\mathbf s'}_{\Gamma,\Theta'}$, inducing a flow-equivariant surjective submersion $f:\mathcal M_{\Gamma,\Theta}^{\mathbf s}\to\mathcal M_{\Gamma,\Theta'}^{\mathbf s'}$ that maps $\mathcal K^{\mathbf s}_{\Gamma,\Theta}$ onto $\mathcal K^{\mathbf s'}_{\Gamma,\Theta'}$. This restriction is also injective because of the flow-equivariance, so it must be a homeomorphism by compactness of $\mathcal K^{\mathbf s}_{\Gamma,\Theta}$.
\end{proof}

\section{Towers of homogeneous spaces}\label{sec:towers}

Given a $\Theta$-Anosov subgroup $\Gamma<G$, the space of $(\Gamma,\Theta)$-admissible parameters $\sbf\in\R^\Theta$ forms an open convex cone parameterizing a family of real analytic Axiom A flows. It is possible to assemble them all into a single dynamical system consisting of an action of the additive group $\R^\Theta$, also called a multiflow.  If $\lvert \Theta \rvert\geq 2,$ the natural invariant set will be a non-compact principal $\R^{|\Theta|-1}$-bundle over the basic hyperbolic set $\mathcal{K}_{\Gamma, \Theta}^{\mathbf{s}}.$

\subsection{The density multiflow space}

Consider opposite flag manifolds $\F^+,\F^-$ associated to subsets $\Theta,\iota_o(\Theta)\subset\Delta$ respectively, and let $\Fpf=\F^+\trtimes\F^-$ be the transverse flag space. Instead of considering multidensities on $\F^+$ and $\F^-$ as we did for the flow spaces in Section \ref{sec:flowspaces}, we will keep track of individual densities on each simple flag manifold dominated by $\F^+$ and $\F^-$. Recall from Corollary \ref{cor domination and opposition} that for every $\alpha\in \Theta$ we have a $G$-equivariant projection $\pi_\alpha:\Fpf\to \F^\pitchfork_\alpha=\F_\alpha\trtimes\F_{\iota_o(\alpha)}$.

\begin{definition} \label{def:multiflow space}
The multiflow space over $\Fpf$ is  the $\R^\Theta$-principal bundle
\[ \Wbb_\Theta:=\prod_{\alpha\in\Theta}\pi_\alpha^\ast \L^{1}_{\{\alpha\}}\,,\]
that is the bundle $\hat p: \Wbb_\Theta\to \Fpf$  with fiber
\[ \big(\Wbb_\Theta\big)_\xbf=\set{ (\nu_\alpha)_{\alpha\in\Theta}}{\nu_\alpha\in \big(\L_{\{\alpha\}}^{1}\big)_{\pi_\alpha(\xbf)}}\]
over $\xbf\in \Fpf$, and $\R^\Theta$-action defined by
\[ \big(\nu_\alpha\big)_{\alpha\in\Theta} \cdot \mathbf t = \big(\phi^{t_\alpha}(\nu_\alpha)\big)_{\alpha\in\Theta}\,,\qquad \forall\, \mathbf t=(t_\alpha)_{\alpha\in \Theta}\in\R^\Theta\,.\]
\end{definition}

\begin{lem}\label{lem:projection multiflow to flow}
For every $\mathbf s\in\R^\Theta$, there is a commutative diagram
\begin{center}
\begin{tikzcd}[row sep=tiny, column sep = tiny]
\Wbb_\Theta \arrow[drr,bend left," \Pi_\Theta^{\mathbf s}"] \arrow[ddr,bend right,swap, "\hat p"] & & \\& & \L_\Theta^{\mathbf s} \arrow[dl,bend left,"p"] \\
 & \Fpf &
\end{tikzcd}
\end{center}
of $G$-equivariant fiber bundles, where $\hat p$ and $p$ denote the projections, and $\Pi_\Theta^{\mathbf s}$ satisfies the equivariance property:
\[ \Pi_\Theta^{\mathbf s}(\nu)\cdot \mathbf t= \phi^{\mathbf t\cdot \mathbf s}\left( \Pi_\Theta^{\mathbf s}(\nu)\right)\,,\quad \forall\nu\in\Wbb_{\Theta}\, \forall\,\mathbf t\in\R^\Theta\,,\]
where $\mathbf t\cdot \mathbf s=\sum_{\alpha\in\Theta}t_\alpha s_\alpha$.
\end{lem}
\begin{proof}The construction that we present here uses (multi)densities and is an exercise in associativity of various product expansions; an algebraic version on the level of homogeneous coset spaces can be inferred from the description of stabilizers in the next section.

Consider $\nu\in \Wbb_\Theta$, let $\xbf=\hat p(\nu)\in\Fpf$ and write $\pi_\alpha(\xbf)=(\xbf_\alpha^+,\xbf_\alpha^-)$ for $\alpha\in\Theta$. By definition, we have that $\nu=(\nu_\alpha)_{\alpha\in\Theta}$ where $\nu_\alpha\in\L^1_{\{\alpha\}}$, that is $\nu_\alpha=(\nu_\alpha^+,\nu_\alpha^-)$ where $\nu_\alpha^+\in \mathscr D^1_{>0}\cF_{\alpha}$ and $\nu_\alpha^-\in \mathscr D^1_{>0}\cF_{\iota_o(\alpha)}$ are positive 1-densities with $\nu_\alpha^+\cdot \nu_\alpha^-=1$. In summary:
\[ \nu =(\nu_\alpha)_{\alpha\in\Theta}= (  \nu_\alpha^+,\nu_\alpha^-)_{\alpha\in\Theta}\,. \]
From  the $s_\alpha$-densities $(\nu_\alpha^+)^{s_\alpha}\in\mathscr D^{s_\alpha}_{>0}\F_\alpha$ and $(\nu_\alpha^-)^{s_\alpha}\in\mathscr D^{s_\alpha}_{>0}\F_{\iota_o(\alpha)}$ we assemble  the multidensities $\nu^{\pm\mathbf{s}}:=\prod_{\alpha\in\Theta}(\nu_\alpha^\pm)^{s_\alpha}\in\mathscr D^{\sbf}_{>0}\F^\pm_\bullet$, which satisfy $\nu^{+\mathbf{s}}\cdot \nu^{-\mathbf{s}}=1$, and set
\[ \Pi_\Theta^{\mathbf s}(\nu) := (\nu^{+\mathbf{s}},\nu^{-\mathbf{s}})\in \L_\Theta^{\mathbf s}\,. \]
For $\mathbf t\in\R^\Theta$, expressing $\phi^{t_\alpha}(\nu_\alpha)=(e^{-t_\alpha}\nu^+_\alpha,e^{t_\alpha}\nu^-_\alpha)$  we find
\begin{align*}
\Pi_\Theta^{\mathbf s}(\nu\cdot \mathbf t) &= \Big( \prod_{\alpha\in\Theta}(e^{-t_\alpha}\nu_\alpha^+)^{s_\alpha},\prod_{\alpha\in\Theta}(e^{t_\alpha}\nu_\alpha^-)^{s_\alpha}\Big)\\
&= \Big( \prod_{\alpha\in\Theta}e^{-t_\alpha s_\alpha}(\nu_\alpha^+)^{s_\alpha},\prod_{\alpha\in\Theta}e^{t_\alpha s_\alpha}(\nu_\alpha^-)^{s_\alpha}\Big)\\
&=\Big(e^{-\mathbf t\cdot\mathbf s} \prod_{\alpha\in\Theta}(\nu_\alpha^+)^{s_\alpha},e^{\mathbf t\cdot\mathbf s}\prod_{\alpha\in\Theta}(\nu_\alpha^-)^{s_\alpha}\Big)\\
&= \phi^{\mathbf t\cdot \mathbf s}\left( \Pi_\Theta^{\mathbf s}(\nu)\right)\,.\vspace*{-1em}
\end{align*}
\end{proof}

\begin{cor} \label{cor:domain of proper discontinuity multiflow space}
Suppose $\Gamma<G$ is $\Theta$-Anosov. Then $\Gamma$ acts properly discontinuously on the open subset $\hat p^{-1}(\Omega_\Gamma)$.
\end{cor}
\begin{proof}
Consider some $(\Gamma,\Theta)$-admissible parameter $\mathbf s$ (e.g. $s_\alpha=1$ for every $\alpha$). Then $\Pi_\Theta^{\mathbf s}$ maps $\hat p^{-1}(\Omega_\Gamma)$ onto $\widetilde{\mathcal M}^{\mathbf s}_{\Gamma,\Theta}$ in a $\Gamma$-equivariant way, so the proper discontinuity on $\widetilde{\mathcal M}^{\mathbf s}_{\Gamma,\Theta}$ implies the same property on $\hat p^{-1}(\Omega_\Gamma)$.
\end{proof}
\begin{rem} As we will see in Section \ref{sec:symmetric space shit}, the $G$-action on $\Wbb_\Delta$ is proper, so any discrete subgroup $\Gamma<G$ acts properly discontinuously on $\Wbb_\Delta$. This is not the case when $\Theta\neq \Delta$ and $G$ has rank at least two.
\end{rem}

We can therefore consider the quotient manifold $\mathcal{N}_{\Gamma,\Theta}:=\Gamma\backslash \hat p^{-1}(\Omega_\Gamma)$. The equivariance formula in Lemma \ref{lem:projection multiflow to flow} shows that $\mathcal{N}_{\Gamma,\Theta}$ (resp. $\widehat{\mathcal K}_{\Gamma,\Theta}:=\Gamma\backslash\hat p^{-1}(\Lambda^\pitchfork_\Gamma)$) is a principal $\mathbf s^\perp\simeq \R^{\lvert \Theta\rvert -1}$-bundle over $\mathcal M^{\mathbf s}_{\Gamma,\Theta}$ (resp. $\mathcal K^{\mathbf s}_{\Gamma,\Theta}$). In particular, $\widehat{\mathcal K}_{\Gamma,\Theta}$ is never compact when $\lvert\Theta\rvert \geq 2$.

\subsection{Algebraic description of homogeneous spaces}
The multiflow space $\Wbb_{\Theta}$ is a $G$-homogeneous space: given $\xbf\in\F_{\Theta}^{\pitchfork}$ with $L_\Theta=\Stab_G(\xbf)$, the action of $g\in L_\Theta$ on the fiber over $\xbf$ is given by
\[ g\cdot (\nu_\alpha)_{\alpha\in\Theta}= \big(e^{\mathbf J_{\Fpf}^\alpha(\xbf,g)}\nu_\alpha  \big)\,,\quad \forall (\nu_\alpha)_{\alpha\in\Theta}\in\big(\Wbb_{\Theta}\big)_\xbf\,,  \]
thus showing that the split center $Z_{\mathrm{split}}(L_{\Theta})$ of $L_\Theta=\Stab_G(\xbf)$ acts transitively on $\big(\Wbb_{\Theta}\big)_\xbf$. Recall the split central projection $\lambda_{L_\Theta}^Z\in\Hom(L_{\Theta},\fz_{\mathrm{split}}(\fl_{\Theta}))$ and its reductive kernel $M_{\Theta}=\ker\lambda_{L_\Theta}^Z$. We can identify the projection $\hat p$ with the natural projection
\[ \Wbb_{\Theta}\simeq G/M_{\Theta} \longrightarrow G/L_{\Theta}\simeq \F_{\Theta}^{\pitchfork}\,.\]
The intermediate homogeneous space $\L_\Theta^{\mathbf s}$ corresponds to $G/H_\bbf$ where $H_\bbf=\ker\mathbf{b}$ for some $\mathbf{b}\in\Hom(L_{\Theta},\R)$, and the map $\Pi_\Theta^{\mathbf s}:\Wbb_{\Theta}\to \L_\Theta^{\mathbf s}$ also corresponds to the natural map given by the chain of inclusions $M_{\Theta}<H_\bbf <L_{\Theta}$.

\subsection{Connection with the Riemannian symmetric space} \label{sec:symmetric space shit}
When $\Theta=\Delta$, i.e. when the flag manifolds $\F^+\simeq\F^-$ are isomorphic to quotients $G/P$ by a minimal parabolic subgroup $P<G$, the reductive subgroup $M=M_L<L$ in the Langlands decomposition $P=MAN$ is compact. In this setting, the action of $A\simeq \R^\Delta$ on $G/M$ is known as the \emph{Weyl chamber flow}, and corresponds to the action on the space of parametrized flats of the Riemannian symmetric space $\mathbb X$ of $G$.

For a general $\Theta\subset \Delta$, consider a maximal compact subgroup $K_{\Theta}<L_{\Theta}$, and a maximal compact subgroup $K<G$ containing $K_{\Theta}$. Note that $K_{\Theta}<M_{\Theta}$, thus adding another homogeneous space in the chain of projections
\[ G/K_{\Theta}\to \Wbb_\Theta\simeq G/M_{\Theta}\to \L_\Theta^{\mathbf s}\simeq G/H_\bbf\to \F_{\Theta}^{\pitchfork}\simeq G/L_{\Theta}\,.\]
It is notable that only the initial and terminal spaces in this list have a known geometric interpretation within the Riemannian symmetric space $\mathbb X$. Indeed, the subgroup $K<G$ is the stabilizer of a point $x\in\mathbb X$, and the subgroup $K_{\Theta}<K$ is the stabilizer of a unit tangent vector $v\in T_x^1\mathbb X$. The geodesic $x_v(t)$ starting at $x$ with initial velocity $v$ has distinct endpoints $v^\pm=\lim_{t\to\pm\infty} x_v(t)$ in the visual boundary $\partial_\infty\mathbb X$. Their stabilizers $P_{\Theta}^\pm=\Stab_G(v^\pm)$ are opposite parabolic subgroups with intersection $L_{\Theta}=\Stab_G(v^+,v^-)$.

\begin{center}
\begin{tikzcd}
T^{1}\mathbb{X} \supset G\cdot v\simeq G/K_{\Theta} \arrow[r] \arrow[d] 
& \Wbb_{\Theta}\simeq G/M_{\Theta} \arrow[d] \\
\mathbb{X}\simeq G/K
& \L_{\Theta}^{\mathbf s} \simeq G/H_\bbf \arrow[d] \\
&\F_{\Theta}^{\pitchfork}\simeq G/L_{\Theta} \arrow[dl] \arrow[dr] \\
\partial_\infty\mathbb X\supset G\cdot v^-\simeq \F_{\Theta}^{-}
&
&\F_{\Theta}^{+}\simeq G\cdot v^+\subset \partial_\infty\mathbb X.
\end{tikzcd}
\end{center}
\begin{rem}
This diagram exposes that the families $\mathbb{L}_{\Theta}^{\mathbf{s}}$ may be viewed as equivariant, inhomogeneous affine line bundles supported on the locally closed strata consisting of  endpoints $(v^{+}, v^{-})\in \partial_{\infty}\mathbb{X}\times \partial_{\infty}\mathbb{X}$ of geodesics of type $\Theta.$ Once $\Theta\subset \Delta$ is not fixed, one must face the stratified geometry of $\partial_{\infty}\mathbb{X}\times \partial_{\infty}\mathbb{X}$ which is a product of spherical buildings.    
 \end{rem}

\section{Examples}\label{sec:examples}

As is usual in constructions which operate within the language of roots systems and semisimple Lie algebras, consideration of explicit examples can dramatically simplify the nature of the objects under consideration.  Therefore, the examples below will not mirror the general construction in Section \ref{sec:flowspaces} identically, but will produce isomorphic models of the flow spaces $\mathbb{L}^{\mathbf{s}}_\Theta$ constructed therein.  We also recommend the first paper \cite{VolI} in our series which treats the real projective case by an independent method and construction, and also focuses in detail on the family of examples given by $\mathbb H^{p,q}$-convex cocompact subgroups.  We explain our results in several other examples below.

\subsection{Special linear groups over \texorpdfstring{$\mathrm{k}=\R,\C$}{real and complex numbers}}\label{sec:SLN}
Consider a  vector $\mathbf{d}=(d_{1},...,d_{r})\in\N^r$ with $d_1+\cdots+d_r=d$. Recall from Example \ref{example flag manifolds} the space of gradings of type $\mathbf d$
\[ \F_{\mathbf d}^\pitchfork(\mathrm{k}^d) = \set{ E_\bullet=(E_i)_{1\leq i\leq r}\in \prod_{i=1}^r\cG_{d_i}(\mathrm{k}^d)}{ E_1\oplus\cdots\oplus E_{r}=\mathrm{k}^d}\,,\]
which corresponds to the coset space $\SL(d,\mathrm{k})/L_{\mathbf d}(\mathrm k)$ where
\[L_{\mathbf d}(\mathrm k)=\set{\begin{pmatrix}g_1 &&0\\ &\ddots & \\ 0&& g_r  \end{pmatrix}}{\det(g_1)\cdots\det(g_r)=1}\subset \SL(d,\mathrm{k})\,.\]
Consider the trivial $\mathrm{k}$-vector bundle 
\[
\mathbb{E}_{\mathbf{d}}=\F_{\mathbf{d}}^{\pitchfork}(\mathrm{k}^d)\times \mathrm{k}^d\rightarrow \F_{\mathbf{d}}^{\pitchfork}(\mathrm{k}^d).
\]
The tautological grading 
\[
\mathbb{E}_{\mathbf{d},\bullet}=\bigoplus_{i=1}^{r} \mathbb{E}_{\mathbf{d}, i}
\]
is given at a basepoint $E_{\bullet}\in \F_{\mathbf{d}}^{\pitchfork}(\mathrm{k}^d)$ by the grading $E_{\bullet}$ itself.\footnote{This is the graded version of the usual tautological bundle construction in projective geometry.}  This construction yields an $\mathrm{SL}(d,\mathrm{k})$-equivariant graded vector bundle
\[
\mathbb{E}_{\mathbf{d},\bullet}\rightarrow \F_{\mathbf{d}}^{\pitchfork}(\mathrm{k}^d).
\]
Next, fix a weight vector $\mathbf{q}=(q_{1},...,q_{r})\in \R^{r}$ with $d_1q_1+\cdots+d_rq_r=0$ to produce the $\mathbf{q}$-multidensity bundle
\[
\mathscr{D}^{\mathbf{q}}\left(\mathbb{E}_{\mathbf{d},\bullet}\right)\rightarrow \F_{\mathbf{d}}^{\pitchfork}(\mathrm{k}^d).
\]  
Regardless of whether $\mathrm{k}=\R$ or $\mathrm{k}=\C,$ the multidensity construction treats the input vector bundle as an $\R$-vector bundle, and therefore the $\mathbf{q}$-multidensity bundle is an $\mathrm{SL}(d,\mathrm{k})$-equivariant, orientable real line bundle over $\F_{\mathbf{d}}^{\pitchfork}(\mathrm{k}^d)$, that can also be described as $\SL(d,\mathrm{k})/H_{\mathbf{q}}$ where  $H_{\mathbf{q}}<L_{\mathbf d}(\mathrm k)$ is the kernel of $ b_{\mathbf q} \in\Hom(L_{\mathbf d}(\mathrm k),\R)$ given by
\[ b_{\mathbf q}\begin{pmatrix}g_1 &&0\\ &\ddots & \\ 0&& g_r  \end{pmatrix}=\sum_{i=1}^r q_i \log  \vert \det(g_i)\vert\,. \]
The flow $\phi^t$ on $\mathscr{D}_{>0}^{\mathbf{q}}\left(\mathbb{E}_{\mathbf{d},\bullet}\right)$ consists in multiplication by $e^t$. On the level of the coset space $\mathrm{SL}(d,\mathrm k)/H_{\mathbf{q}}$, it is given by right multiplication by the element
\[ \begin{pmatrix}e^{t q_1} \mathbf{1}_{d_1} &&0\\ &\ddots & \\ 0&& e^{tq_r}\mathbf{1}_{d_r}  \end{pmatrix}\,.\]
\begin{rem} \label{rem:special linear example translation}
In terms of the general construction from Section \ref{sec:flowspaces}, we find
\[
\mathscr{D}_{>0}^{\mathbf{q}}\left(\mathbb{E}_{\mathbf{d},\bullet}\right)\simeq \mathrm{SL}(d,\mathrm k)/H_{\mathbf{q}}\simeq  \mathbb{L}_{\Theta_{\mathbf d}}^{\mathbf{s}(\mathbf{q})}.
\]
where $\mathbf{s}(\mathbf{q})=\frac{1}{d\dim_\R(\mathrm{k})}(q_1-q_2,\ldots,q_{r-1}-q_r)$, and  $\Theta_{\mathbf d}$ was introduced in Example \ref{example simple flag manifolds}.
\end{rem}

\subsubsection{Multiflows}

Finally, we discuss the multiflow spaces from Section \ref{sec:towers}.   Consider the sum of line bundles
\[
\bigoplus_{i=1}^{r} \mathscr{D}^{1}(\mathbb{E}_{i, \mathbf{d}})\rightarrow \F_{\mathbf{d}}^{\pitchfork}(\mathrm{k}^d).
\]
In its total space, the open subset of positive $1$-densities $
\bigoplus_{i=1}^{r} \mathscr{D}_{>0}^{1}(\mathbb{E}_{i, \mathbf{d}})$  is invariant by the fiber-wise transitive $\R^{r}$-action via
\[
(\tau_{i})_{1\leq i\leq r}\cdot (\nu_{i})_{1\leq i\leq r}=(e^{\tau_{i}}\nu_{i})_{1\leq i\leq r}.
\]
The trivial bundle $\mathbb{E}_{\mathbf{d}}=\F_{\mathbf{d}}^{\pitchfork}(\mathrm{k}^d)\times \mathrm{k}^d$ carries an $\mathrm{SL}(d,\mathrm{k})$-invariant fiberwise volume form $\omega.$ It can be used to define the product $\nu_1\cdots\nu_r\in\R$ of a collection $(\nu_i)$ of 1-densities on the spaces $E_i$ of $E_\bullet\in\F_{\mathbf d}^\pitchfork(\mathrm{k}^d)$ by letting
\[ \nu_1\cdots\nu_r:= \frac{\nu_1(\mathbf e_1)\cdots\nu_r(\mathbf e_r)}{|\omega(\mathbf e_1,\dots,\mathbf e_r)|}\,\]
where each basis $\mathbf e_i=(e_{i,1},\dots,e_{i,d_i})\in\mathcal B_{E_i}$ is concatenated into $(\mathbf e_1,\dots,\mathbf e_r)\in\mathcal B_E$. This quantity is independent from the chosen bases, and we can now consider the codimension one submanifold
\[
\Wbb_{\mathbf{d}}:=\set{(\nu_i)_{1\leq i\leq r}\in \bigoplus_{i=1}^{r} \mathscr{D}_{>0}^{1}(\mathbb{E}_{i, \mathbf{d}})}{\nu_1\cdots\nu_r=1}\subset \bigoplus_{i=1}^{r} \mathscr{D}_{>0}^{1}(\mathbb{E}_{i, \mathbf{d}}).
\]
This manifold is invariant under the subgroup $\set{\tau_\bullet\in\R^r}{d_1\tau_1+\cdots+d_r\tau_r=0}\simeq\R^{r-1}$, and homogeneous under the $\mathrm{SL}(d,\mathrm{k})$-action, with stabilizer conjugate to
\[M_{\mathbf d}(\mathrm k)=\set{\begin{pmatrix}g_1 &&0\\ &\ddots & \\ 0&& g_r  \end{pmatrix}\in L_{\mathbf d}(\mathrm k)}{|\det(g_i)|=1}\subset L_{\mathbf d}(\mathrm k)\,.\] 
 Concerning the $\R^{r-1}$-action, it corresponds to multiplication on the right by the subgroup
\[A_{\mathbf{d}}=\set{ \begin{pmatrix}e^{\tau_1} \mathbf{1}_{d_1} &&0\\ &\ddots & \\ 0&& e^{\tau_d}\mathbf{1}_{d_r}  \end{pmatrix}\in\SL(d,\mathrm{k})}{\tau_i\in\R}\,.\] 
 
Given a tuple $\mathbf{q}=(q_{1},...,q_{r}),$  the fiberwise monomial map
\begin{align*}\begin{split}
\label{homo projo} \bigoplus_{i=1}^{r} \mathscr{D}^{1}_{>0}(\mathbb{E}_{i, \mathbf{d}})&\rightarrow \mathscr{D}^{\mathbf{q}}_{>0}(\mathbb{E}_{\bullet, \mathbf{d}}) \\
(\nu_{i})_{1\leq i \leq r} &\mapsto  \prod_{1\leq i\leq r} \ \nu_{i}^{q_{i}}\end{split}
\end{align*}
corresponds to the projection  $\SL(d,\mathrm{k})/M_{\mathbf d}(\mathrm k) \to \SL(d,\mathrm{k})/H_{\mathbf q}$.

\subsection{Lagrangians in \texorpdfstring{$(E, \omega)$}{a symplectic vector space}}  

We next study the case of transverse Lagrangian subspaces in a symplectic real vector space.  Let $E$ be a finite dimensional real vector space of dimension $2d\geq 2$ and let $\omega\in \Lambda^{2}E^{*}$ be a symplectic 2-form. The space of Lagrangian subspaces
\[ \Lag=\set{L\in\cG_d(E)}{\omega\vert_{L\times L}=0}\]
is a simple flag manifold for the symplectic group $\Sp(E,\omega)$. It is opposite to itself, and the corresponding transverse flag space is
\[ \Lag^\pitchfork=\set{(L_1,L_2)\in\Lag\times \Lag}{L_1\oplus L_2=E}\,.\]
Consider $(L_1,L_2)\in\Lag^\pitchfork$, as well as positive 1-densities $\mu_i\in\mathscr D^1_{>0}(L_i)$ (that is $\mu_i=|\omega_i|$ for some non-zero $\omega^i\in \Lambda^dL_i^\ast$). Using the symplectic volume instead of the pseudo-Riemannian volume used in Section \ref{sec:flowspaces}, we can define the pairing $\mu_1\cdot\mu_2\in\R_{>0}$  as  the quantity
\[ \mu_1\cdot\mu_2= \frac{\mu_1(\mathbf e_1)\mu_2(\mathbf e_2)}{\left\vert\mathrm{vol}_\omega(\mathbf e_1,\mathbf e_2)\right\vert} \]
independent of the ordered bases $\mathbf e_i=(e_{i,1},\cdots,e_{i,d})\in\mathcal B_{L_i}$, where \[(\mathbf e_1,\mathbf e_2)=(e_{1,1},\cdots,e_{1,d},e_{2,1},\cdots,e_{2,d})\in\mathcal B_E\] stands for the concatenation and $\mathrm{vol}_\omega=\frac{1}{d!}\omega^d$. From there we can define the flow space $\L\to \Lag^\pitchfork$ with fiber
\[ \L_{(L_1,L_2)}=\set{(\mu_1,\mu_2)\in\mathscr D^1_{>0}(L_1)\times \mathscr D^1_{>0}(L_2)}{ \mu_1\cdot\mu_2=1}\]
and endowed with the flow
\[ \phi^t(\mu_1,\mu_2)=(e^{-t}\mu_1,e^t\mu_2)\,.\]

Now let $\Gamma<\Sp(E,\omega)$ be Anosov with respect to the simple restricted root defining the flag manifold $\Lag$ and torsion-free. As $\Lag$ is opposite to itself, we have a single limit map $\xi:\partial_\infty\Gamma\to\Lag$ and the corresponding invariant subsets
\begin{align*}
\Lambda_\Gamma^\pitchfork&= \set{(\xi(z^+),\xi(z^-))}{z^+\neq z^-}\subset\Lag^\pitchfork\,,\\
\Omega_\Gamma&= \set{(L_1,L_2)\in\Lag^\pitchfork}{\forall z\in\partial_\infty\Gamma~L_1\pitchfork \xi(z)\textrm{ or } \xi(z)\pitchfork L_2}
\end{align*}
lift respectively to invariant subsets $\widetilde{\mathcal K}_\Gamma\subset \widetilde{\mathcal M}_\Gamma\subset\L$ where the action of $\Gamma$ on $\widetilde{\mathcal M}_\Gamma$ is properly discontinuous and the restriction to $\widetilde{\mathcal K}_\Gamma$ is cocompact, leading to a real analytic contact Axiom A flow on the quotient $\mathcal M_\Gamma=\Gamma\backslash\widetilde{\mathcal M}_\Gamma$ with non-wandering set $\mathcal K_\Gamma=\Gamma\backslash\widetilde{\mathcal K}_\Gamma$.

\begin{rem} It is possible to use non-zero volume forms $\omega_i\in \Lambda^dL_i^*$ instead of densities, but in the absence of a natural orientation of a Lagrangrian subspace, one has to consider them up to multiplication by $-1$. This leads either to this density construction, or to considering the Plücker embedding $\Lag\subset\cG_d(E)\hookrightarrow\mathbb P(\Lambda^dE)$ then using the flow considered in \cite{VolI}.
\end{rem}

There is an injective homomorphism $h: \mathrm{SL}(2, \R)\rightarrow \mathrm{Sp}(E, \omega)$ defined by
\[
\begin{pmatrix}a & b\\c&d\end{pmatrix}\mapsto \begin{pmatrix}a\mathbf{1}_d & b\mathbf{1}_d\\c\mathbf{1}_d&d\mathbf{1}_d\end{pmatrix}
\]
preserving a standard splitting $E\simeq L\oplus L^{*},$ and centralized by the compact subgroup 
\[
\begin{pmatrix}
A & 0 \\
0 & A
\end{pmatrix}
\]
where $A\in \mathrm{O}(d).$  If $\Gamma_{0}<\mathrm{SL}(2, \R)$ is a cocompact Fuchsian group and $\rho: \Gamma\rightarrow \mathrm{O}(d)$ is a homomorphism, the product $\gamma\mapsto \rho(\gamma)h(\gamma)$ defines a $\mathrm{Lag}^{\pitchfork}$-Anosov subgroup $\Gamma_{\rho}<\mathrm{Sp}(E, \omega)$ with the property that all of the associated Axiom A systems $\mathcal{M}_{\rho}^{s}$ have the same periods for fixed $s\in \R.$  Since the parameter space $\mathrm{Hom}(\Gamma, \mathrm{O}(d))$ is positive dimensional once $d>1,$ this shows that a general Anosov subgroup is not determined (up to conjugacy) by the periods of its associated Axiom A flows.  This is in contrast to geodesic flows in negative curvature where length spectrum rigidity is a common phenomenon.  The corresponding refraction flows are all conjugate, and therefore indistinguishable using classical Ergodic theory.  In this case, microlocal objects (e.g. Ruelle resonant states and invariant distributions) should be helpful to distinguish the corresponding Axiom A systems. 

\providecommand{\bysame}{\leavevmode\hbox to3em{\hrulefill}\thinspace}

\bigskip

\end{document}